\documentclass[11pt,reqno,english]{amsart}
\usepackage[LGR,T1]{fontenc}
\usepackage[latin9]{inputenc}
\usepackage{geometry}
\usepackage{color}
\usepackage{babel}
\usepackage{mathrsfs}
\usepackage{mathtools}
\usepackage{bm}
\usepackage{amsbsy}
\usepackage{amstext}
\usepackage{amsthm}
\usepackage{amssymb}
\usepackage{stmaryrd}
\usepackage{graphicx}
\usepackage{setspace}
\usepackage{esint}
\usepackage[unicode=true,pdfusetitle,
 bookmarks=true,bookmarksnumbered=false,bookmarksopen=false,
 breaklinks=false,pdfborder={0 0 1},backref=false,colorlinks=false]
 {hyperref}

\makeatletter

\ProvideTextCommand{\~}{LGR}[1]{\char126#1}

\numberwithin{equation}{section}
\numberwithin{figure}{section}
\theoremstyle{plain}
\newtheorem{theorem}{\protect\theoremname}[section]
\theoremstyle{plain}
\newtheorem{assumption}[theorem]{\protect\assumptionname}
\theoremstyle{definition}
\newtheorem{definition}[theorem]{\protect\definitionname}
\theoremstyle{remark}
\newtheorem{remark}[theorem]{\protect\remarkname}
\theoremstyle{plain}
\newtheorem{proposition}[theorem]{\protect\propositionname}
\theoremstyle{plain}
\newtheorem{lemma}[theorem]{\protect\lemmaname}
\theoremstyle{plain}
\newtheorem{corollary}[theorem]{\protect\corollaryname}
\theoremstyle{remark}
\newtheorem*{acknowledgement*}{\protect\acknowledgementname}

\usepackage{cite}
\usepackage{bbm}

\providecommand{\leftsquigarrow}{%
  \mathrel{\mathpalette\reflect@squig\relax}%
}
\newcommand{\reflect@squig}[2]{%
  \reflectbox{$\m@th#1\rightsquigarrow$}%
}

\makeatother

\providecommand{\acknowledgementname}{Acknowledgement}
\providecommand{\assumptionname}{Assumption}
\providecommand{\corollaryname}{Corollary}
\providecommand{\definitionname}{Definition}
\providecommand{\lemmaname}{Lemma}
\providecommand{\propositionname}{Proposition}
\providecommand{\remarkname}{Remark}
\providecommand{\theoremname}{Theorem}

\begin{document}
\title{Mixing of metastable diffusion processes with Gibbs invariant distribution}
\author{Jungkyoung Lee}
\begin{abstract}
In this article, we study the mixing properties of metastable diffusion
processes which possess a Gibbs invariant distribution. For systems
with multiple stable equilibria, so-called metastable transitions
between these equilibria are required for mixing since the unique
invariant distribution is concentrated on these equilibria. Consequently,
these systems exhibit slower mixing compared to those with a unique
stable equilibrium, as analyzed in Barrera and Jara (Ann. Appl. Probab.
30:1164--1208, 2020). Our proof is based on the theory of metastability,
which is a primary tool for studying systems with multiple stable
equilibria. Within this framework, we compute the total variation
distance between the distribution of the diffusion process and its
invariant distribution for any time scale larger than $\epsilon^{-1}$.
Finally, we derive precise asymptotics for the mixing time.
\end{abstract}

\address{June E Huh Center for Mathematical Challenges, Korea Institute for Advanced Study, Republic
of Korea., e-mail: \texttt{jklee@kias.re.kr}}
\maketitle

\section{Introduction}

Let $E$ be a certain space (a subset of $\mathbb{R}^{d}$ or finite
set) and let $\mu$, $\nu$ be probability measures on $E$. The \textit{total
variation distance} between $\mu$ and $\nu$ is defined as
\[
{\color{blue}d_{{\rm TV}}(\mu,\,\nu)}:=\sup_{A\subset E}\left|\,\mu(A)-\nu(A)\,\right|\,,
\]
where the supremum is taken over all measurable subsets $A\subset E$.
If $E$ is finite, the total variation distance can be written as
\[
d_{{\rm TV}}(\mu,\,\nu)=\frac{1}{2}\,\sum_{x\in E}|\,\mu(x)-\nu(x)\,|\,.
\]
For a stochastic process $\{X(t)\}_{t\ge0}$ on $E$, a probability
measure $\mu$, and $s\ge0$, denote by \textcolor{blue}{$X(s;\,\mu)$}
the distribution of $X(s)$ with initial distribution $\mu$. If $\mu$
is a dirac mass {\color{blue}$\delta_{x}$} on $x\in E$, we write ${\color{blue}X(s;\,x)}:=X(s;\,\delta_{x})$.
Suppose that the Markov process $X(\cdot)$ has a unique invariant
distribution $\nu$. By ergodic property, the distribution of $X(t)$
converges to $\nu$ as $t\to\infty$, i.e., the total variation distance
between $X(t,\,x)$ and $\nu$ converges to $0$ as $t\to\infty$,
regardless of $x\in E$. This phenomenon is referred to as \emph{mixing}.
We refer to monograph \cite{LPW} for a comprehensive review on the
study of mixing.

For $\delta>0$, the $\delta$-mixing time of the ergodic process
$X(\cdot)$, which measures the time required to be close to the invariant
distribution, is defined as
\[
{\color{blue}T^{{\rm mix}}(\delta)}:=\inf\,\{\,t\ge0\,:\,\sup_{x\in E}d_{{\rm TV}}(X(t;\,x),\,\nu)\le\delta\,\}\,.
\]
In the context of applied mathematics, since the mixing time indicates
the number of steps required to approximate the target measure $\nu$
from the distribution of $X(\cdot)$, it is crucial for evaluating
the performance of Markov chain Monte Carlo (MCMC) algorithms (see
\cite{Brooks} and references therein for more information on MCMC).
Consequently, the mixing time has attracted the attention of many
researchers not only in mathematics but also in engineering, computer
science, and many others.

\subsection{Metastable diffusion processes}

Let us consider a stochastic differential equation (SDE) in $\mathbb{R}^{d}$
of the form
\begin{equation}
d\bm{x}_{\epsilon}(t)=\bm{b}(\bm{x}_{\epsilon}(t))\,dt\,+\,\sqrt{2\epsilon}d\bm{w}_{t}\,,\label{e: SDE}
\end{equation}
where $\bm{b}:\mathbb{R}^{d}\to\mathbb{R}^{d}$ is a smooth vector
field, $(\bm{w}_{t})_{t\ge0}$ represents a $d$-dimensional Brownian
motion, and $\epsilon>0$ is a small parameter corresponding to the
temperature. We assume that the process $\bm{x}_{\epsilon}(\cdot)$
has a unique invariant distribution $\mu_{\epsilon}$ given by\footnote{Without confusion, we abuse notation as $\mu_{\epsilon}(\boldsymbol{x})=\frac{1}{Z_{\epsilon}}e^{-U(\boldsymbol{x})/\epsilon}$.}
\begin{equation}
\mu_{\epsilon}(d\boldsymbol{x})=\frac{1}{Z_{\epsilon}}\,e^{-U(\boldsymbol{x})/\epsilon}\,d\boldsymbol{x}\label{e_Gibbs}
\end{equation}
for some smooth function $U:\mathbb{R}^{d}\to\mathbb{R}$, where $Z_{\epsilon}:=\int_{\mathbb{R}^{d}}e^{-U(\bm{x})/\epsilon}d\bm{x}$
is a normalization constant. This model has been studied not only
in pure mathematics \cite{FW} but also in various applied mathematics
fields including deep learning \cite{GSS} and MCMC \cite{SGMCMC}.
In this article, we study the mixing of the class of the diffusion
processes given by \eqref{e: SDE} with multiple stable equilibria
and the invariant distribution \eqref{e_Gibbs} in the small temperature
regime. The main result of this article provides sharp asymptotics
for the mixing time as well.

\subsection{Fast mixing }

The process $\bm{x}_{\epsilon}(\cdot)$ is regarded as a small random
perturbation of the dynamical system given by an ordinary differential
equation (ODE) of the form
\begin{equation}
d\bm{x}(t)=\bm{b}(\bm{x}(t))\,dt\,.\label{e: ODE}
\end{equation}
The behavior of $\bm{x}_{\epsilon}(\cdot)$ varies depending on the
number of stable equilibria of the dynamical system described by the
ODE \eqref{e: ODE}. When the dynamical system \eqref{e: ODE} has
only one equilibrium, it exhibits \emph{fast mixing}. Denote by $\bm{0}=(0,\,\dots,\,0)\in\mathbb{R}^{d}$
the origin. Suppose that $\bm{0}$ is the unique equilibrium of the
dynamical system \eqref{e: ODE} and there exist positive constants
$\delta,\,c_{0},\,c_{1}>0$ such that
\begin{equation}
\begin{aligned}\bm{x}\cdot D\bm{b}(\bm{y})\, \bm{x} & \le-\delta|\bm{x}|^{2}\ \ \ \ \text{for any}\,\bm{x},\,\bm{y}\in\mathbb{R}^{d}\,,\\
|\bm{b}(\bm{x})| & \le c_{0}e^{c_{1}|\bm{x}|^{2}}\ \text{for any}\,\bm{x}\in\mathbb{R}^{d}\,,
\end{aligned}
\label{e: BJ}
\end{equation}
where {\color{blue}$|\bm{x}|$} represents the Euclidean norm of $\boldsymbol{x}\in\mathbb{R}^{d}$
and {\color{blue}$D\bm{A}(\bm{x})$} represents the Jacobian matrix of vector field
$\bm{A}:\mathbb{R}^{d}\to\mathbb{R}^{d}$ at $\bm{x}\in\mathbb{R}^{d}$
(cf. Conditions (H) and (G) in \cite{BarJara20}). Then, the solution
to the ODE \eqref{e: ODE} starting at a fixed point converges to
the unique equilibrium $\bm{0}$ exponentially fast. Consequently,
the small-perturbed dynamics $\bm{x}_{\epsilon}(\cdot)$ reaches a neighborhood
of the equilibrium rapidly as well. Since the invariant distribution
$\mu_{\epsilon}$ is concentrated on the neighborhood of the unique
equilibrium, the distribution
of $\bm{x}_{\epsilon}(\cdot)$ converges to $\mu_{\epsilon}$ very
quickly.

Under the condition \eqref{e: BJ}, Barrera and Jara \cite{BarJara20}
proved that the process $\bm{x}_{\epsilon}(\cdot)$ exhibits not only
fast mixing but also a cut-off phenomenon. More precisely, they found
a time scale\footnote{For two positive sequences $(\alpha_{\epsilon})_{\epsilon>0}$ and
$(\beta_{\epsilon})_{\epsilon>0}$, write $\alpha_{\epsilon}\prec\beta_{\epsilon}$
or $\beta_{\epsilon}\succ\alpha_{\epsilon}$ if $\lim_{\epsilon\to0}\alpha_{\epsilon}/\beta_{\epsilon}=0$.} $1\prec t_{\epsilon}\prec\epsilon^{-1}$ and proved that before time
$t_{\epsilon}$, the total variation distance between the distribution
of the process $\bm{x}_{\epsilon}(\cdot)$ and the invariant distribution
$\mu_{\epsilon}$ is close to $1$, and that shortly after time $t_{\epsilon}$,
the distance is close to $0$. It is noteworthy that the invariant
distribution $\mu_{\epsilon}$ does not necessarily have to be a Gibbs
measure under the condition \eqref{e: BJ}. For further details, see
\cite{BarJara20}.

\subsection{Slow mixing and metastability}

\begin{figure}
\centering
\includegraphics[scale=0.25]{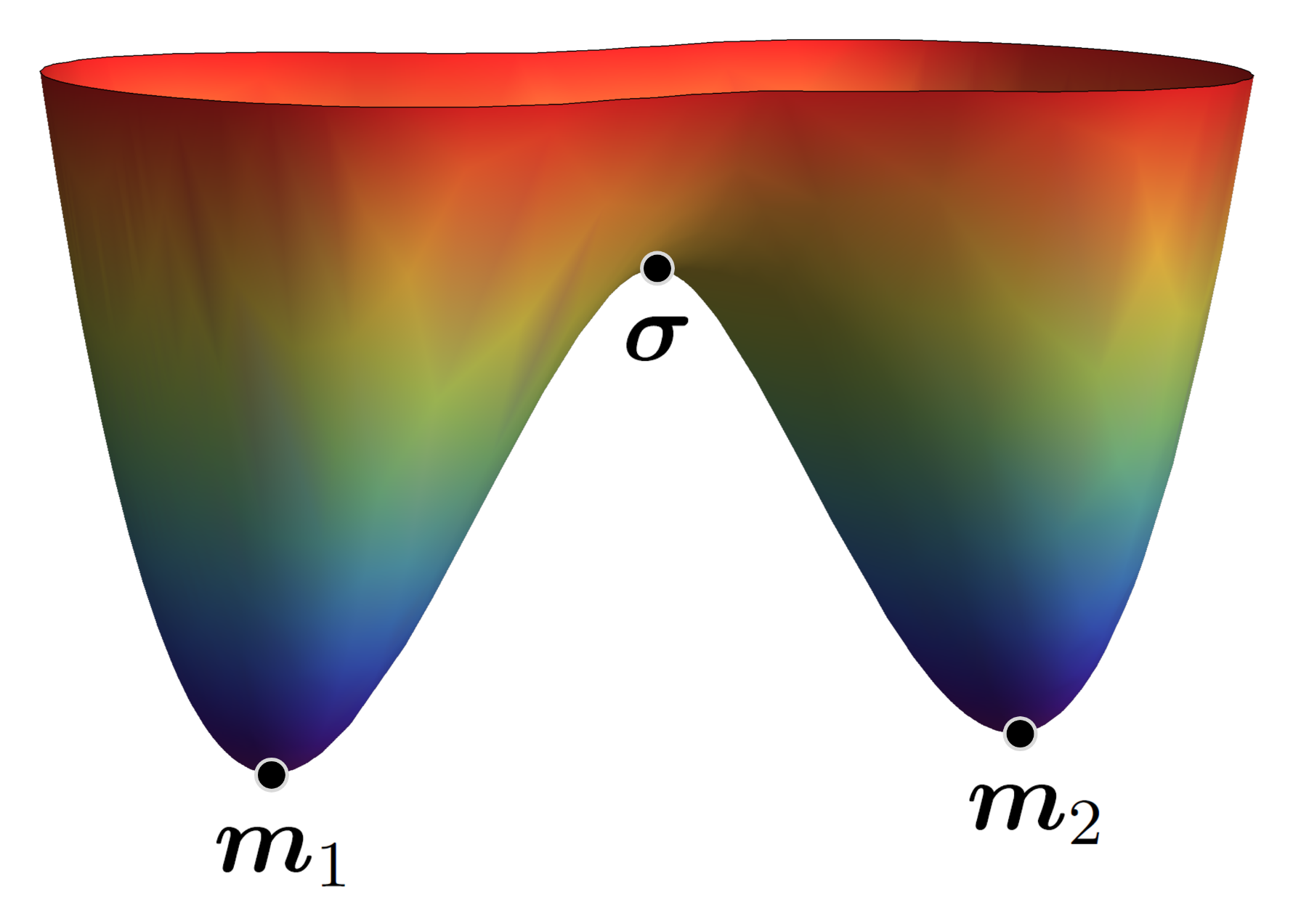}
\caption{Example of potential $U$ with two local minima $\bm{m}_{1}$, $\bm{m}_{2}$ and a saddle point $\bm{\sigma}$.}
\label{fig: double-well}
\end{figure}

From now on, assume that the vector field $\boldsymbol{b}:\mathbb{R}^{d}\rightarrow\mathbb{R}^{d}$
has a decomposition of the form 
\begin{equation}
\boldsymbol{b}=-(\nabla U+\boldsymbol{\ell})\text{\ \;where\ \;}\nabla\cdot\boldsymbol{\ell}\equiv0\;\ \text{and}\;\;\nabla U\cdot\boldsymbol{\ell}\equiv0\,,\label{e_decb}
\end{equation}
for some smooth vector field $\boldsymbol{\ell}\colon\mathbb{R}^{d}\to\mathbb{R}^{d}$,
and that the dynamical system \eqref{e: ODE} has more than one stable
equilibrium. It has been shown in \cite{LS-22} that the diffusion
process $\boldsymbol{x}_{\epsilon}(\cdot)$ has the invariant distribution
\eqref{e_Gibbs} if, and only if, the vector field $\bm{b}(\cdot)$
is decomposed as \eqref{e_decb}. This dynamics is a generalization
of the reversible overdamped Langevin dynamics, which is the process
$\bm{x}_{\epsilon}(\cdot)$ in the special case $\bm{\ell}=\bm{0}$.
For the sake of simplicity, assume that $\bm{\ell}=\bm{0}$ and there
are two local minima $\bm{m}_{1}$, $\bm{m}_{2}$, as well as one
saddle point $\bm{\sigma}$ of $U$ as in Figure \ref{fig: double-well},
and denote by $D=U(\bm{\sigma})-U(\bm{m}_{1})$. Suppose that $U(\bm{m}_{1})\ge U(\bm{m}_{2})$
and that $\bm{x}_{\epsilon}(\cdot)$ starts at $\bm{m}_{1}$. Since
the measure $\mu_{\epsilon}$ has a mass on a neighborhood of $\bm{m}_{2}$,
the process $\bm{x}_{\epsilon}(\cdot)$ has to make a transition to
$\bm{m}_{2}$ to mix. However, since $\bm{m}_{1}$ is stable and the
drift $\bm{b}$ pushes the system toward $\bm{m}_{1}$, the process
$\bm{x}_{\epsilon}(\cdot)$ remains near $\bm{m}_{1}$ for a very
long time. Eventually, as the small noise $\sqrt{2\epsilon}d\bm{w}_{t}$
accumulates sufficiently over time, the process escapes the domain
of attraction of $\bm{m}_{1}$ and makes a transition to $\bm{m}_{2}$.
It is well known from Freidlin and Wentzell \cite{FW} that an exponentially
long time $e^{D/\epsilon}$ is required for such a transition to occur,
causing the process $\bm{x}_{\epsilon}(\cdot)$ to exhibit \emph{slow mixing}. 

The transition phenomenon between stable states is referred to as
metastability. Now, assume that $U$ has more than one but finitely
many local minima and denote by \textcolor{blue}{$\mathcal{M}_{0}$}
the set of local minima of $U$. It has been shown in \cite{LS-22}
that $\mathcal{M}_{0}$ is the set of stable equilibria of the dynamical
system \eqref{e: ODE}. In \cite{LS-22b,RS}, it was proven that the
metastable behavior of process $\bm{x}_{\epsilon}(\cdot)$ between
global minima is described by a certain Markov chain whose state space
is a partition of the set of global minima. In \cite{LLS-1st,LLS-2nd},
this result was extended to the metastability between all local minima.
More precisely, there exist a positive integer $\mathfrak{q}$, time
scales $\theta_{\epsilon}^{(1)}\prec\cdots\prec\theta_{\epsilon}^{(\mathfrak{q})}$
and Markov chains ${\bf y}^{(1)}(\cdot),\,\dots,\,{\bf y}^{(\mathfrak{q})}(\cdot)$
whose state spaces are partitions of subsets of $\mathcal{M}_{0}$
such that the Markov chain\footnote{For $a\le b$, let $\llbracket a,\,b\rrbracket:=[a,\,b]\cap\mathbb{Z}$.}
${\bf y}^{(p)}(\cdot)$, $p\in\llbracket1,\,\mathfrak{q}\rrbracket$,
describes the metastable behavior of $\bm{x}_{\epsilon}(\cdot)$ in
time scale $\theta_{\epsilon}^{(p)}$. Since the Markov chains describing
metastable behavior of the original process $\bm{x}_{\epsilon}(\cdot)$
are much simpler, this approach is referred to as \emph{Markov chain model
reduction} \cite{BeltranLandim1,BeltranLandim11,BeltranLandim2}. Additionally, since the mixing of the simpler Markov chains
is more tractable, the theory of Markov chain model reduction is the
most essential tool of studying slow mixing of the process $\bm{x}_{\epsilon}(\cdot)$.

\subsection{Main contribution of this article}

In this article, we express and prove the slow mixing of the process
$\bm{x}_{\epsilon}(\cdot)$ with Gibbs invariant distribution in the
language of metastability. Firstly, we extend the theory of metastability
developed in \cite{BeltranLandim1,LLM} to the mixing property by
generalizing the \emph{convergence in finite-dimensional distribution} obtained
in \cite{LLS-1st,LLS-2nd} (cf. Definition \ref{def_FC}) to the \emph{convergence
in total variation distance} (cf. Definition \ref{def_TV}) in Theorem
\ref{t: main_TV}. Specifically, at the time scale $\theta_{\epsilon}^{(p)}$,
$p\in\llbracket1,\,\mathfrak{q}\rrbracket$, we show that the distance
between the distribution of $\bm{x}_{\epsilon}(\cdot)$ and a lifting
of the distribution of a Markov chain ${\bf y}^{(p)}(\cdot)$ converges
to $0$ as $\epsilon\to0$.

Using the triangle inequality, at the time scale $\theta_{\epsilon}^{(p)}$,
$p\in\llbracket1,\,\mathfrak{q}\rrbracket$, we then reduce the distance
between the distribution of $\bm{x}_{\epsilon}(\cdot)$ and $\mu_{\epsilon}$
to the distance between the distribution of the Markov chain ${\bf y}^{(p)}(\cdot)$
and $\mu_{\epsilon}$. Since the lifting of the distribution of ${\bf y^{(p)}}(\cdot)$
is a convex combination of conditioned measures of $\mu_{\epsilon}$,
the latter distance can be computed explicitly (cf. Lemma \ref{l: TV_density}).
By the monotonicity of the total variation distance (cf. Lemma \ref{l: d_TV}),
the total variation distance in any intermediate time scale $\theta_{\epsilon}^{(p-1)}\prec\varrho_{\epsilon}\prec\theta_{\epsilon}^{(p)}$
is computed as well. This result is presented in Theorem \ref{t: main}.

Finally, we derive a sharp asymptotics for the mixing time of the
process $\bm{x}_{\epsilon}(\cdot)$ in Theorem \ref{t: mixing_time}.
First, we prove that the total variation distance between finite ergodic
Markov chains and their invariant distribution decreases strictly in
time (cf. Lemma \ref{l: d_TV-2}). This refines the well-known property
of the monotonicity of the total variation distance and is crucial
to prove the upper bound of mixing time (cf. Lemma \ref{l: mix_up}).
The mixing time of $\bm{x}_{\epsilon}(\cdot)$ is expressed as the
product of the mixing time of the last Markov chain ${\bf y}^{(\mathfrak{q})}(\cdot)$
and the last time scale $\theta_{\epsilon}^{(\mathfrak{q})}$. Since
the last Markov chain ${\bf y}^{(\mathfrak{q})}(\cdot)$ describes
the transitions between global minima, on which the invariant measure
$\mu_{\epsilon}$ is concentrated, it is a natural result one can
expect. To the best of the author's knowledge, this is the first result
on the precise asymptotics for the mixing time of metastable diffusion processes
which exhibit slow mixing.

Metastability is a main tool to treat stochastic systems with complex
potentials. In this article, we develop the theory of metastability
to study slow mixing of such systems. We emphasize that the main argument
(cf. Section \ref{sec: meta_TV}) is \emph{model-independent}. We
only assume the following conditions
\begin{itemize}
\item Convergence in finite-dimensional distributions (cf. Condition \ref{def_FC}).
\item Mixing condition (cf. Condition \ref{def_mixing}).
\item Generalization of starting points (cf. Proposition \ref{p: Cond_H}).
\end{itemize}
Therefore, one can apply this approach to study mixing of metastable
systems such as particle systems \cite{Beltran Landim-ZRP,Kim-IP3rd,Seo -ZRP},
ferromagnetic systems \cite{BC,Cuff DLLPS,Kim Seo-IsingPotts1,Kim Seo-IsingPotts2,Lee-Potts},
machine learning \cite{BDM}, and others.

The metastability of the process \eqref{e: SDE} in the general case, without assumption \eqref{e_decb}, has attracted considerable attention.
However, since the quantitative study of the metastability of general processes remains beyond our current reach, the slow mixing of such processes has remained an open problem. Nevertheless, since our methodology is model-independent, we expect that the slow mixing of metastable diffusion processes with non-Gibbs invariant distributions can be derived provided that convergence in finite-dimensional distribution is achieved.

\subsection{Related works}

Fast mixing is widely observed in various stochastic systems with
a unique equilibrium. Aldous and Diaconis \cite{AldousDiacon} demonstrated
the cut-off phenomenon in the card-shuffling model, which has drawn
significant interest in the study of this phenomenon. For further
exploration of this phenomenon, we refer to Cuff et al. \cite{Cuff DLLPS}
for the Curie--Weiss--Potts model, Barrera and Pardo \cite{BarP}
for Ornstein--Uhlenbeck processes driven by L\'evy processes, and Lee,
Ramil, and Seo \cite{LRS} for the underdamped Langevin dynamics.
It is worth noting that fast mixing can occur without cut-off. For
Langevin dynamics with gradual fast mixing, see \cite{BarJara22}.

In contrast, when a system has multiple equilibria, mixing occurs exponentially slowly. For instance, in the low-temperature regime, Glauber dynamics for mean-field Ising model exhibits slow mixing \cite{DingLP}. This phenomenon arises due to metastability. Comprehensive summaries of research on metastability can be
found in the monographs \cite{BH,ov}. The Eyring--Kramers formula, which provides precise asymptotics
for the expected value of the transition time, was first derived by Eyring \cite{Ey} and Kramers \cite{Kra} who computed it for one-dimensional diffusion processes. Generalizations to higher dimensions
can be found in \cite{BEGK,LMS2,LS-22}.
The Eyring--Kramers formula for the optimal constant in the Poincar\'{e} and logarithmic Sobolev inequality was studied in \cite{MenzSch}.

Another aspect of metastability involves successive transitions between metastable states, rather than a single transition. The general theory for describing hopping dynamics between metastable states using simpler Markov chains has been developed recently in \cite{BeltranLandim1,BeltranLandim11,BeltranLandim2,LLM,LMS} and such approach for diffusion processes was studied in \cite{BBDiG, LS-22b, RS}.
For further discussion, we refer to \cite{BGK2, BroDiG, DiG, HKN, LM, SchSlowik} for sharp asymptotics for
the low-lying spectra of the process $\bm{x}_{\epsilon}(\cdot)$, \cite{QSD1,QSD2,QSD3} for a quasi-stationary distribution approach, \cite{bgl2,GesuMariani,LMS-Gamma} for Gamma expansions of large deviation rate functionals.

\section{Model and results}

\subsection{Model}

Assume that the drift is decomposed to $\bm{b}=-(\nabla U+\bm{\ell})$
for some $U\in C^{3}(\mathbb{R}^{d})$ and $\bm{\ell}\in C^{2}(\mathbb{R}^{d};\,\mathbb{R}^{d})$
satisfying \eqref{e_decb}, where $C^{k}(\mathbb{R}^{d})$ and $C^{k}(\mathbb{R}^{d};\,\mathbb{R}^{d})$
represent the spaces of $k$-times differentiable functions and vector
fields on $\mathbb{R}^{d}$, respectively. Furthermore, we assume
that $U$ satisfies the following growth condition:
\begin{equation}
\begin{gathered}\lim_{n\to\infty}\inf_{|\boldsymbol{x}|\geq n}\frac{U(\boldsymbol{x})}{|\boldsymbol{x}|}=\infty\;,\quad\lim_{|\boldsymbol{x}|\to\infty}\frac{\boldsymbol{x}}{|\boldsymbol{x}|}\cdot\nabla U(\boldsymbol{x})=\infty\,,\\
\lim_{|\boldsymbol{x}|\to\infty}\big\{\,|\nabla U(\boldsymbol{x})|\,-\,2\,\Delta U(\boldsymbol{x})\,\big\}=\infty\,.
\end{gathered}
\label{e: growth}
\end{equation}
Denote by {\color{blue}$\mathcal{C}_{0}$} the set of critical points of $U$ and by
{\color{blue}$\nabla^{2}U(\bm{x})$} the Hessian of $U$ at $\bm{x}\in\mathbb{R}^{d}$.
Assume that $U$ is a Morse function (cf. \cite[Definition 1.7]{Nic18}),
i.e., all critical points $\bm{c}\in\mathcal{C}_{0}$ are nondegenerate
in the sense that $\nabla^{2}U(\bm{c})$ is invertible. These assumptions
are generic as in \cite{LS-22,LS-22b} guaranteeing the following.
Recall that $\mathcal{M}_{0}$ is the set of local minima of $U$
and $\mu_{\epsilon}$ is the Gibbs distribution given in \eqref{e_Gibbs}.
\begin{theorem}[{\cite[Theorems 2.1 and 2.2]{LS-22}}]
Suppose that $\bm{b}=-(\nabla U+\bm{\ell})$ for some $U\in C^{3}(\mathbb{R}^{d})$
and $\bm{\ell}\in C^{2}(\mathbb{R}^{d};\,\mathbb{R}^{d})$ satisfying
\eqref{e_decb} and \eqref{e: growth}. Then, we have the following.
\begin{enumerate}
\item The set $\mathcal{C}_{0}$ is exactly the set of all the equilibria of
the dynamical system \eqref{e: ODE}, and the set $\mathcal{M}_{0}$
is exactly the set of all the stable equilibria of \eqref{e: ODE}. 
\item The process $\bm{x}_{\epsilon}(\cdot)$ is positive recurrent and $\mu_{\epsilon}$
is the unique stationary distribution of the process $\bm{x}_{\epsilon}(\cdot)$.
\end{enumerate}
\end{theorem}

Since we assumed multiple stable equilibria, we have $|\mathcal{M}_{0}|\ge2$.
For $\bm{c}_{1},\,\bm{c}_{2}\in\mathcal{C}_{0}$, a \textit{\textcolor{blue}{heteroclinic
orbit}} $\phi$ from $\bm{c}_{1}$ to $\bm{c}_{2}$ is a smooth path
$\phi:\mathbb{R}\to\mathbb{R}^{d}$ satisfying $\dot{\phi}(t)=\bm{b}(\phi(t))$
for all $t\in\mathbb{R}$ and
\[
\lim_{t\to-\infty}\phi(t)=\bm{c}_{1}\ ,\ \ \lim_{t\to\infty}\phi(t)=\bm{c}_{2}\,.
\]
Denote by {\color{blue}$\mathcal{S}_{0}$} the set of saddle points of $U$. Since
$U$ is a Morse function, $\mathcal{S}_{0}$ is the set of all $\bm{\sigma}\in\mathcal{C}_{0}$
where the Hessian $\nabla^{2}U(\bm{\sigma})$ has one negative eigenvalue
and $d-1$ positive eigenvalues. In particular, for all $\bm{\sigma}\in\mathcal{S}_{0}$,
by the Hartman-Grobman theorem (cf. \cite[Section 2.8]{Perko}), there
exist two heteroclinic orbits $\phi$ satisfying $\lim_{t\to-\infty}\phi(t)=\bm{\sigma}$.
Now, the following is the main assumption as in \cite{LLS-1st,LLS-2nd}.
\begin{assumption}
\label{assu: hetero}For all $\boldsymbol{\sigma}\in\mathcal{S}_{0}$,
two heteroclinic orbits are connected to local minima, i.e., letting
$\phi_{\pm}$ be two different heteroclinic orbits satisfying $\lim_{t\to-\infty}\phi_{\pm}(t)=\bm{\sigma}$,
we have $\lim_{t\to\infty}\phi_{\pm}(t)\in\mathcal{M}_{0}$.
\end{assumption}

\subsection{Tree-structure}

In this subsection , we recall the tree-structure introduced in \cite{LLS-2nd}.
The rigorous definitions of depths and Markov chains are presented
in Appendix \ref{app: Tree}.
\begin{definition}[Tree-structure]
\label{def:tree}The tree-structure consists of the following objects.
\begin{enumerate}
\item A positive integer ${\color{blue}\mathfrak{q}\ge1}$ denoting the
number of scales. 
\item A finite sequence of depths $0<d^{(1)}<\cdots<d^{(\mathfrak{q})}<\infty$
and time-scales
\[
{\color{blue}\theta_{\epsilon}^{(p)}}:=\exp\frac{d^{(p)}}{\epsilon}\;\;;\ \;p\in\llbracket1,\,\mathfrak{q}\rrbracket\,.
\]
\item A sequence of partitions $\mathscr{V}^{(p)}\cup\mathscr{N}^{(p)}$,
$p\in\llbracket1,\,\mathfrak{q}\rrbracket$, of $\mathcal{M}_{0}$.
\item A sequence of finite-state Markov chains $\widehat{{\bf y}}^{(p)}(\cdot)$
and ${\bf y}^{(p)}(\cdot)$, $p\in\llbracket1,\,\mathfrak{q}\rrbracket$,
on $\mathscr{V}^{(p)}\cup\mathscr{N}^{(p)}$ and $\mathscr{V}^{(p)}$,
respectively.
\end{enumerate}
\end{definition}

\begin{remark}
For $p\in\llbracket1,\,\mathfrak{q}\rrbracket$, $\theta_{\epsilon}^{(p)}$
is the critical time scale at which the asymptotic behavior of $\bm{x}_{\epsilon}(\cdot)$
undergoes a quantitative change. $\mathscr{V}^{(p)}$ represents metastable states with energy barriers large enough to trap the process $\bm{x}_{\epsilon}(\cdot)$ at the time scale $\theta_\epsilon^{(p)}$, while the Markov chain ${\bf y}^{(p)}(\cdot)$
captures the hopping dynamics of $\bm{x}_{\epsilon}(\cdot)$ within $\mathscr{V}^{(p)}$
at the same time scale, $\theta_{\epsilon}^{(p)}$ (cf. \cite[Theorem 3.1]{LLS-2nd} and Theorem \ref{t: main_TV}).
\end{remark}

Let
\[
{\color{blue}\mathscr{V}^{(1)}}:=\left\{ \,\{\boldsymbol{m}\}:\boldsymbol{m}\in\mathcal{M}_{0}\,\right\} \;\ \text{and}\ \;{\color{blue}\mathscr{N}^{(1)}}:=\varnothing\,.
\]
Let \textcolor{blue}{$\{{\bf y}^{(1)}(t)\}_{t\ge0}$} be a $\mathscr{V}^{(1)}$-valued
Markov chain defined in Appendix \ref{subsec: MC1}. Denote by $\mathfrak{n}_{1}$
the number of irreducible classes of ${\bf y}^{(1)}(\cdot)$ and by
${\color{blue}\mathscr{R}_{1}^{(1)},\dots,\mathscr{R}_{\mathfrak{n}_{1}}^{(1)}}$,
${\color{blue}\mathscr{T}^{(1)}}$ the closed irreducible classes
and the transient states of the Markov chain ${\bf y}^{(1)}(\cdot)$,
respectively.

Fix $n\ge1$ and assume that the construction has been carried up
to layer $n$. Denote by $\mathfrak{n}_{n}$ the number of irreducible
classes of the Markov chain ${\bf y}^{(n)}(\cdot)$ and by ${\color{blue}\mathscr{R}_{1}^{(n)},\dots,\mathscr{R}_{\mathfrak{n}_{n}}^{(n)}}$,
${\color{blue}\mathscr{T}^{(n)}}$ the closed irreducible classes
and the transient states of the Markov chain ${\bf y}^{(n)}(\cdot)$,
respectively. If $\mathfrak{n}_{n}=1$, the construction is complete
and we obtain $\mathfrak{q}=n$. On the other hand, if $\mathfrak{n}_{n}>1$
we add a new layer to the construction as follows. For $j\in\llbracket1,\,\mathfrak{n}_{n}\rrbracket$,
let
\begin{equation}
{\color{blue}\mathcal{M}_{j}^{(n+1)}}:=\bigcup_{\mathcal{M}\in\mathscr{R}_{j}^{(n)}}\mathcal{M}\;,\ \ {\color{blue}\mathscr{V}^{(n+1)}}:=\left\{ \,\mathcal{M}_{1}^{(n+1)}\,\dots,\mathcal{M}_{\mathfrak{n}_{n}}^{(n+1)}\,\right\} \;,\ \ {\color{blue}\mathscr{N}^{(n+1)}}:=\mathscr{N}^{(n)}\cup\mathscr{T}^{(n)}\,,\label{e: M_j-R_j}
\end{equation}
and let ${\color{blue}\mathscr{S}^{(n+1)}}:=\mathscr{V}^{(n+1)}\cup\mathscr{N}^{(n+1)}$.
Let \textcolor{blue}{$\{\widehat{{\bf y}}^{(n+1)}(t)\}_{t\ge0}$}
be a $\mathscr{S}^{(n+1)}$-valued Markov chain defined in Appendix
\ref{subsec: MC2}. Then, we define \textcolor{blue}{$\{{\bf y}^{(n+1)}(t)\}_{t\ge0}$}
as a trace of ${\widehat{\bf y}}^{(n+1)}(\cdot)$ on $\mathscr{V}^{(n+1)}$
(cf. Appendix \ref{app: trace}).

By Proposition \ref{p: tree}-(1), there exists $\mathfrak{q}\in\mathbb{N}$
such that $\mathfrak{n}_{1}>\cdots>\mathfrak{n}_{\mathfrak{q}}=1$
so that the constructions ends at $n=\mathfrak{q}$. Finally, we extend
the definition \eqref{e: M_j-R_j} to $n=\mathfrak{q}$ as
\[
{\color{blue}\mathcal{M}_{1}^{(\mathfrak{q}+1)}}:=\bigcup_{\mathcal{M}\in\mathscr{R}_{1}^{(\mathfrak{q})}}\mathcal{M}\ ,\ \ \mathscr{V}^{(\mathfrak{q}+1)}:=\{\,\mathcal{M}_{1}^{(\mathfrak{q}+1)}\,\}\ ,\ \ \mathscr{N}^{(\mathfrak{q}+1)}:=\mathscr{N}^{(\mathfrak{q})}\cup\mathscr{T}^{(\mathfrak{q})}\,.
\]
From the construction, it follows that
\begin{itemize}
\item If $A\in\mathscr{V}^{(n)}$, $n\in\llbracket1,\,\mathfrak{q}\rrbracket$,
then either $A\in\mathscr{N}^{(n+1)}$ or there exists a unique $B\in\mathscr{V}^{(n+1)}$
such that $A\subset B$. In contrast, if $A\in\mathscr{N}^{(n)}$
then $A\in\mathscr{N}^{(n+1)}$.
\item For each $n\in\llbracket1,\,\mathfrak{q}+1\rrbracket$, $\mathscr{V}^{(n)}\cap\mathscr{N}^{(n)}=\varnothing$
and the elements of $\mathscr{S}^{(n)}$ form a partition of $\mathcal{M}_{0}$.
\item For $n\in\llbracket1,\,\mathfrak{q}\rrbracket$, the transient states
of the chain $\mathbf{y}^{(n)}(\cdot)$ are the sets in $\mathscr{V}^{(n)}$
which are transferred to $\mathscr{N}^{(n+1)}$ and the closed irreducible
classes form the elements of $\mathscr{V}^{(n+1)}$.
\end{itemize}
Denote by ${\color{blue}\mathcal{M}_{\star}}$ the set of global minima
of $U$. By \cite[Proposition 4.11-(2)]{LLS-2nd},
\begin{equation}
\mathcal{M}_{\star}=\mathcal{M}_{1}^{(\mathfrak{q}+1)}=\bigcup_{\mathcal{M}\in\mathscr{R}_{1}^{(\mathfrak{q})}}\mathcal{M}\,.\label{e: min}
\end{equation}
This tree-structure has the following properties. 
\begin{itemize}
\item $\mathcal{M}_{0}$ is the root of the tree.
\item There are $\mathfrak{q}+1$ generations and vetices of each generation
form a partition of $\mathcal{M}_{0}$.
\item For $n\in\llbracket1,\,\mathfrak{q}+1\rrbracket$, the elements of
$\mathscr{S}^{(\mathfrak{q}+2-n)}=\mathscr{V}^{(\mathfrak{q}+2-n)}\cup\mathscr{N}^{(\mathfrak{q}+2-n)}$
are vertices of the $n$-th generation. In particular,
\begin{itemize}
\item $\mathcal{M}_{\star}\in\mathscr{V}^{(\mathfrak{q}+1)}$ and the elements
of $\mathscr{N}^{(\mathfrak{q}+1)}=\mathscr{T}^{(\mathfrak{q})}\cup\mathscr{N}^{(\mathfrak{q})}$
are vertices of the first generation,
\item the elements of $\mathscr{V}^{(1)}=\mathcal{M}_{0}$, which are singletons
are the leaves of the tree, i.e., vertices of the last generation.
\end{itemize}
\end{itemize}
An example of the tree-structure and its visualization is presented
in Appendix \ref{subsec: example}.

\subsection{Main result on metastability}

In this subsection, we present the result on metastability. First,
we recall the notions from \cite{LLS-2nd}. For $H\in\mathbb{R}$,
define
\[
{\color{blue}\{U<H\}}:=\{\boldsymbol{x}\in\mathbb{R}^{d}:U(\boldsymbol{x})<H\}\;\;\text{and \;\;}{\color{blue}\{U\le H\}}:=\{\boldsymbol{x}\in\mathbb{R}^{d}:U(\boldsymbol{x})\le H\}\,.
\]
For each $\boldsymbol{m}\in\mathcal{M}_{0}$, denote by ${\color{blue}\mathcal{W}^{r}(\boldsymbol{m})}$
the connected component of the set $\{U\le U(\boldsymbol{m})+r\}$
containing $\boldsymbol{m}$. Fix $r_0 > 0$ small enough so that
\begin{equation}
r_{0}\in\left(0,\,\frac{1}{3}\min\{d^{(1)},\,d^{(2)}-d^{(1)},\,\dots,\,d^{(\mathfrak{q})}-d^{(\mathfrak{q}-1)}\}\right)\label{e: r_0-0} \, ,
\end{equation}
and for all $\bm{m}\in\mathcal{M}_{0}$, condition
(a)-(e) at the paragraph before the display (2.12) in \cite{LLS-1st}
are satisfied. In particular,
\begin{equation}
\bm{m}\ \text{is the only critical point in}\ \mathcal{W}^{3r_{0}}(\bm{m})\,.\label{e: r_0}
\end{equation}
 Let us define the\emph{ valley }around $\boldsymbol{m}$ as 
\begin{equation}
{\color{blue}\mathcal{E}(\boldsymbol{m}})=\mathcal{W}^{r_{0}}(\boldsymbol{m})\,,\label{e_Em}
\end{equation}
and for $\mathcal{M}\subset\mathcal{M}_{0}$, define
\begin{equation}
{\color{blue}\mathcal{E}(\mathcal{M})}:=\bigcup_{\boldsymbol{m}\in\mathcal{M}}\mathcal{E}(\boldsymbol{m})\,.\label{e_E_M}
\end{equation}

For $p\in\llbracket1,\,\mathfrak{q}\rrbracket$, let {\color{blue}$\widehat{\mathcal{Q}}_{\mu}^{(p)}$} and {\color{blue}$\mathcal{Q}_{\nu}^{(p)}$} be the law of the processes $\widehat{{\bf y}}^{(p)}(\cdot)$
and ${\bf y}^{(p)}(\cdot)$ starting from probability measures $\mu$ and $\nu$ on $\mathscr{S}^{(p)}$ and $\mathscr{V}^{(p)}$, respectively. If $\mu$ or $\nu$ is a dirac measure $\delta_\mathcal{M}$ for some $\mathcal{M}\in\mathscr{S}^{(p)}$, we write $\widehat{\mathcal{Q}}_{\mu}^{(p)}=\widehat{\mathcal{Q}}_{\mathcal{M}}^{(p)}$ or $\mathcal{Q}_{\nu}^{(p)}=\mathcal{Q}_{\mathcal{M}}^{(p)}$.
For $\boldsymbol{m}\in\mathcal{M}_{0}$ and $p\in\llbracket1,\,\mathfrak{q}\rrbracket$,
denote by ${\color{blue}\mathcal{M}(p,\,\boldsymbol{m})}$ the element
in $\mathscr{S}^{(p)}$ which contains $\boldsymbol{m}$. For $p\in\llbracket1,\,\mathfrak{q}\rrbracket$,
$\mathcal{M}\in\mathscr{V}^{(p)}$ and $\boldsymbol{m}\in\mathcal{M}_{0}$,
let $\mathfrak{a}_{\bm{m}}^{(p-1)}(\mathcal{M})$ be the probability
that the auxiliary Markov chain $\widehat{\mathbf{y}}^{(p)}(\cdot)$
starting from $\mathcal{M}(p,\,\boldsymbol{m})$ firstly reaches an
element of $\mathscr{V}^{(p)}$ exactly at $\mathcal{M}$: 
\begin{equation}
{\color{blue}\mathfrak{a}_{\bm{m}}^{(p-1)}(\mathcal{M})}:=\widehat{\mathcal{Q}}_{\mathcal{M}(p,\,\boldsymbol{m})}^{(p)}\left[\,\tau_{\mathscr{V}^{(p)}}=\tau_{\mathcal{M}}\,\right]\,,\label{e: def_a}
\end{equation}
where $\tau_{A}$, $A\subset\mathscr{V}^{(p)}\cup\mathscr{N}^{(p)}$,
represents the hitting time\footnote{There is an abuse of notation since we write $\tau_{\mathcal{M}}$
for $\tau_{\{\mathcal{M}\}}$. Also, we abuse notation of hitting
time of the diffusion process $\bm{x}_{\epsilon}(\cdot)$ as $\tau_{\mathcal{A}}:=\inf\left\{ \,t>0\,:\,\bm{x}_{\epsilon}(t)\in\mathcal{A}\,\right\} $
for $\mathcal{A}\subset\mathbb{R}^{d}$.} of the set $A$.
\[
{\color{blue}\tau_{A}}:=\inf\,\left\{ \,t>0\,:\,\widehat{\mathbf{y}}^{(p)}(t)\in A\,\right\} \,.
\]
By the definition, for $p\in\llbracket1,\,\mathfrak{q}\rrbracket$
and $\bm{m}\in\mathcal{M}_{0}$, $\mathfrak{a}_{\bm{m}}^{(p-1)}$
is a probability measure on $\mathscr{V}^{(p)}$. Clearly, for $\boldsymbol{m}\in\mathcal{M}_{0}$,
$\mathfrak{a}_{\bm{m}}^{(0)}=\delta_{\bm{m}}$ and if $\bm{m}\in\mathcal{M}$
for some $\mathcal{M}\in\mathscr{V}^{(p)}$, $\mathfrak{a}_{\bm{m}}^{(p-1)}=\delta_{\mathcal{M}}$.

For $p\in\llbracket1,\,\mathfrak{q}\rrbracket$, $n\in\mathbb{N}$
and $0<t_{1}<\cdots<t_{n}$, let $\mu_{\epsilon,\,\bm{x}}^{(p),\,t_{1},\,\dots,\,t_{n}}(\cdot)$
be the joint law of $\left(\boldsymbol{x}_{\epsilon}(\theta_{\epsilon}^{(p)}\,t_{1}),\,\dots,\,\boldsymbol{x}_{\epsilon}(\theta_{\epsilon}^{(p)}\,t_{n})\right)$
starting at $\bm{x}$ and let $\mu_{\epsilon}^{\mathcal{A}}$, $\mathcal{A}\subset\mathbb{R}^{d}$,
be a conditioned measure on $\mathcal{A}$ defined as
\[
{\color{blue}\mu_{\epsilon}^{\mathcal{A}}(d\bm{x})}:=\frac{1}{\mu_{\epsilon}(\mathcal{A})}\mathcal{X}_{\mathcal{A}}(\bm{x})\,\mu_{\epsilon}(d\bm{x})\,,
\]
where {\color{blue}$\mathcal{X}_{\mathcal{A}}$} is a characteristic function of
$\mathcal{A}$.
For $\bm{m}\in\mathcal{M}_0$, denote by {\color{blue}$\mathcal{D}(\bm{m})$} the domain of attraction of $\bm m$:
$$
\mathcal{D}(\bm m) := \{\bm x \in \mathbb{R}^d : \text{The ODE \eqref{e: ODE} starting from }\bm x \, \text{converges to } \bm m \}\, .
$$

Now, we present the main result on metastability,
which states that the Markov chain ${\bf y}^{(p)}(\cdot)$ describes
the metastable behavior of the diffusion process $\bm{x}_{\epsilon}(\cdot)$
in time scale $\theta_{\epsilon}^{(p)}$.
\begin{theorem}[Convergence in total variation]
\label{t: main_TV}Suppose that $\bm{b}=-(\nabla U+\bm{\ell})$ for
some $U\in C^{3}(\mathbb{R}^{d})$ and $\bm{\ell}\in C^{2}(\mathbb{R}^{d};\,\mathbb{R}^{d})$
satisfying \eqref{e_decb}, \eqref{e: growth} and that Assumption
\ref{assu: hetero} is in force. Let $p\in\llbracket1,\,\mathfrak{q}\rrbracket$,
$\bm{m}\in\mathcal{M}_{0}$, $n\in\mathbb{N}$ and $0<t_{1}<\cdots<t_{n}$.
Then, for any $\bm{x}\in\mathcal{D}(\bm{m})$ and any sequences $(t_{j,\,\epsilon})_{\epsilon>0}$,
$j\in\llbracket1,\,n\rrbracket$, such that $t_{j,\,\epsilon}\to t_{j}$,
\[
\lim_{\epsilon\to0}d_{{\rm TV}}\left(\,\mu_{\epsilon,\,\bm{x}}^{(p),\,t_{1,\,\epsilon},\,\dots,\,t_{n,\,\epsilon}},\,\sum_{\mathcal{M}_{1},\,\dots,\,\mathcal{M}_{n}\in\mathscr{V}^{(p)}}\mathcal{Q}_{\mathfrak{a}_{\bm{m}}^{(p-1)}}^{(p)}\left[\,\bigcap_{j=1}^{n}\{\mathbf{y}^{(p)}(t_{j})=\mathcal{M}_{j}\}\,\right]\,\mu_{\epsilon}^{\mathcal{M}_{1},\,\dots,\,\mathcal{M}_{n}}\,\right)=0\,,
\]
where $\mu_{\epsilon}^{\mathcal{M}_{1},\,\dots,\,\mathcal{M}_{n}}=\mu_{\epsilon}^{\mathcal{E}(\mathcal{M}_{1})}\times\cdots\times\mu_{\epsilon}^{\mathcal{E}(\mathcal{M}_{n})}$
is a product measure of $\mu_{\epsilon}^{\mathcal{E}(\mathcal{M}_{1})},\,\dots,\,\mu_{\epsilon}^{\mathcal{E}(\mathcal{M}_{n})}$.
\end{theorem}

\begin{remark}
This extends the convergence in finite-dimensional distribution \cite[Theorem 3.1]{LLS-2nd}
to the convergence in total variation distance. Indeed, in \cite[Theorem 3.1]{LLS-2nd},
the value of $\lim_{\epsilon\to0}\mu_{\epsilon,\,\bm{x}}^{(p),\,t_{1,\,\epsilon},\,\dots,\,t_{n,\,\epsilon}}(\mathcal{A})$
was obtained only if $\mathcal{A}=\mathcal{E}(\mathcal{M}_{1})\times\cdots\times\mathcal{E}(\mathcal{M}_{n})$
as follows:
\begin{align*}
 & \lim_{\epsilon\to0}\mu_{\epsilon,\,\bm{x}}^{(p),\,t_{1,\,\epsilon},\,\dots,\,t_{n,\,\epsilon}}(\mathcal{E}(\mathcal{M}_{1})\times\cdots\times\mathcal{E}(\mathcal{M}_{n}))\\
 & =\mathcal{Q}_{\mathfrak{a}_{\bm{m}}^{(p-1)}}^{(p)}\left[\,\bigcap_{j=1}^{n}\{\mathbf{y}^{(p)}(t_{j})=\mathcal{M}_{j}\}\,\right]\,\mu_{\epsilon}^{\mathcal{M}_{1},\,\dots,\,\mathcal{M}_{n}}(\mathcal{E}(\mathcal{M}_{1})\times\cdots\times\mathcal{E}(\mathcal{M}_{n}))\,.
\end{align*}
\end{remark}

\subsection{Main result on mixing}

\subsubsection{Total variation distance}

If $\mathcal{M}\subset\mathcal{M}_{0}$ satisfies
\[
U(\bm{m})=U(\bm{m}')\ \ \text{for all}\ \bm{m},\,\bm{m}'\in\mathcal{M}\,,
\]
$\mathcal{M}$ is said to be \textit{\textcolor{blue}{simple}} and
we denote by \textcolor{blue}{$U(\mathcal{M})$ }the common value.
For $p\in\llbracket1,\,\mathfrak{q}\rrbracket$, define
\[
\mathscr{V}_{\star}^{(p)}:=\{\,\mathcal{M}\in\mathscr{V}^{(p)}\,:\,\mathcal{M}\subset\mathcal{M}_{\star}\,\}\,.
\]
$\mathscr{V}_{\star}^{(p)}$ is well defined since the elements of $\mathscr{V}^{(p)}$
are simple by Proposition \ref{p: tree}-(1). We claim that for all
$p\in\llbracket1,\,\mathfrak{q}\rrbracket$,
\begin{equation}
\mathcal{M}_{\star}=\bigcup_{\mathcal{M}\in\mathscr{V}_{\star}^{(p)}}\mathcal{M}\,.\label{e: M_star-V_star}
\end{equation}
Indeed, the right-hand side is a subset of the left-hand side. On the
other hand, if $\bm{m}\in\mathcal{M}_{\star}$, by \eqref{e: min},
$\mathcal{M}(\mathfrak{q},\,\bm{m})\in\mathscr{V}^{(\mathfrak{q})}$
. Hence, by the construction, for all $p\in\llbracket1,\,\mathfrak{q}\rrbracket$,
$\mathcal{M}(p,\,\bm{m})\in\mathscr{V}_{\star}^{(p)}$ so that $\bm{m}\in\mathcal{M}$
for some $\mathcal{M}\in\mathscr{V}_{\star}^{(p)}$. The weights $\nu(\boldsymbol{m})$
and $\nu(\mathcal{M})$ for $\boldsymbol{m}\in\mathcal{M}_{0}$ and
$\mathcal{M}\subset\mathcal{M}_{0}$ are defined as 
\[
{\color{blue}\nu(\boldsymbol{m})}:=\frac{1}{\sqrt{\det(\nabla^{2}U)(\boldsymbol{m})}}\;,\ \ {\color{blue}\nu(\mathcal{M})}:=\sum_{\boldsymbol{m}\in\mathcal{M}}\nu(\boldsymbol{m})\;,\ \ \nu_{\star}:=\nu(\mathcal{M}_{\star})\,.
\]
Then, for $p\in\llbracket1,\,\mathfrak{q}\rrbracket$, by \eqref{e: M_star-V_star},
a measure \textcolor{blue}{$\nu^{(p)}$} on $\mathscr{V}^{(p)}$ defined
by
\begin{equation}
\nu^{(p)}(\mathcal{M}):=\begin{cases}
\frac{\nu(\mathcal{M})}{\nu_{\star}} & \mathcal{M}\in\mathscr{V}_{\star}^{(p)}\,,\\
0 & \text{otherwise}\,,
\end{cases}\label{e: nu^p}
\end{equation}
is a probability measure. The following is the second main result
of this article. 
\begin{theorem}
\label{t: main}Suppose that $\bm{b}=-(\nabla U+\bm{\ell})$ for some
$U\in C^{3}(\mathbb{R}^{d})$ and $\bm{\ell}\in C^{2}(\mathbb{R}^{d};\,\mathbb{R}^{d})$
satisfying \eqref{e_decb}, \eqref{e: growth}, and that Assumption
\ref{assu: hetero} is in force. For all $\bm{m}\in\mathcal{M}_{0}$,
the following hold.
\begin{enumerate}
\item For all $p\in\llbracket1,\,\mathfrak{q}\rrbracket$, $\bm{x}\in\mathcal{D}(\bm{m})$
and $t>0$,
\[
\lim_{\epsilon\to0}\,d_{{\rm TV}}(\boldsymbol{x}_{\epsilon}(\theta_{\epsilon}^{(p)}t;\,\bm{x}),\,\mu_{\epsilon})=d_{{\rm TV}}\left({\bf y}^{(p)}\left(t;\,\mathfrak{a}_{\bm{m}}^{(p-1)}\right),\,\nu^{(p)}\right)\,.
\]
\item For all $p\in\llbracket1,\,\mathfrak{q}\rrbracket$, $\bm{x}\in\mathcal{D}(\bm{m})$
and sequences $(\varrho_{\epsilon})_{\epsilon>0}$ such that $\theta_{\epsilon}^{(p-1)}\prec\varrho_{\epsilon}\prec\theta_{\epsilon}^{(p)}$,
\[
\lim_{\epsilon\to0}\,d_{{\rm TV}}(\boldsymbol{x}_{\epsilon}(\varrho_{\epsilon};\,\bm{x}),\,\mu_{\epsilon})=d_{{\rm TV}}\left(\mathfrak{a}_{\bm{m}}^{(p-1)},\,\nu^{(p)}\right)
\]
where ${\color{blue}\theta_{\epsilon}^{(0)}}:=\epsilon^{-1}$.
\item For all $\bm{x}\in\mathbb{R}^{d}$ and sequences $(\varrho_{\epsilon})_{\epsilon>0}$
such that $\varrho_{\epsilon}\succ\theta_{\epsilon}^{(\mathfrak{q})}$,
\[
\lim_{\epsilon\to0}d_{{\rm TV}}(\boldsymbol{x}_{\epsilon}(\varrho_{\epsilon};\,\bm{x}),\,\mu_{\epsilon})=0\,.
\]
\end{enumerate}
\end{theorem}

\begin{remark}[The lower bound $\theta_{\epsilon}^{(0)}$]
The time scale $\theta_{\epsilon}^{(0)}=\epsilon^{-1}$ is the time
required for mixing in a neighborhood of one minimum (cf. Theorem
\ref{t: mix_F}). However, as proved in \cite[Theorem 2.2]{BarJara20},
since the mixing occurs in time scale
\[
A\log\frac{1}{\epsilon}+B\log\log(\frac{1}{\epsilon})+C
\]
for some $A,\,B,\,C>0$, this time scale is not optimal.
\end{remark}

\begin{remark}
For $\epsilon^{-1}\prec\varrho_{\epsilon}\prec\theta_{\epsilon}^{(1)}$,
since $\mathfrak{a}_{\bm{m}}^{(0)}=\delta_{\bm{m}}$,
\[
\lim_{\epsilon\to0}d_{{\rm TV}}(\bm{x}_{\epsilon}(\varrho_{\epsilon};\,\bm{x}),\,\mu_{\epsilon})=d_{{\rm TV}}\left(\delta_{\bm{m}},\,\nu^{(1)}\right)\,.
\]
In particular, if $\bm{m}\notin\mathcal{M}_{\star}$, $d_{{\rm TV}}\left(\delta_{\bm{m}},\,\nu^{(1)}\right)=1$.
In other words, since $\mu_{\epsilon}$ is negligible on $\bm{m}$
and there is no transition in a time scale smaller than the first
time scale $\theta_{\epsilon}^{(1)}$, there is no mixing at all.
On the other hand, if $\varrho_{\epsilon}\succ\theta_{\epsilon}^{(\mathfrak{q})}$,
for all $\bm{x}\in\mathbb{R}^{d}$,
\[
\lim_{\epsilon\to0}d_{{\rm TV}}(\bm{x}_{\epsilon}(\varrho_{\epsilon};\,\bm{x}),\,\mu_{\epsilon})=0\,.
\]
Therefore, the process is fully mixed in a time scale larger than
the last time scale $\theta_{\epsilon}^{(\mathfrak{q})}$, regardless
of starting points.
\end{remark}

Now, we extend definition of $\mathfrak{a}_{\bm{m}}^{(k)}$ and $\nu^{(k+1)}$
to $k=\mathfrak{q}$ as a probability measure on $\{\mathcal{M}_{1}^{(\mathfrak{q}+1)}\}$,
i.e.,
\[
{\color{blue}\mathfrak{a}_{\bm{m}}^{(\mathfrak{q})}=\nu^{(\mathfrak{q}+1)}}:=\delta_{\mathcal{M}_{1}^{(\mathfrak{q}+1)}}\,.
\]
For $\bm{m}\in\mathcal{M}_{0}$ and $p\in\llbracket1,\,\mathfrak{q}\rrbracket$,
write
\[
f_{\bm{m}}^{(p)}(t):=d_{{\rm TV}}\left({\bf y}^{(p)}\left(t;\,\mathfrak{a}_{\bm{m}}^{(p-1)}\right),\,\nu^{(p)}\right)\,.
\]

\begin{remark}
Since ${\bf y}^{(p)}\left(0;\,\mathfrak{a}_{\bm{m}}^{(p-1)}\right)=\mathfrak{a}_{\bm{m}}^{(p-1)}$,
$\lim_{t\to0}f_{\bm{m}}^{(p)}(t)=d_{{\rm TV}}\left(\mathfrak{a}_{\bm{m}}^{(p-1)},\,\nu^{(p)}\right)$.
Also, by Lemma \ref{l: 48} below, $\lim_{t\to\infty}f_{\bm{m}}^{(p)}(t)=d_{{\rm TV}}\left(\mathfrak{a}_{\bm{m}}^{(p)},\,\nu^{(p+1)}\right)$.
Then, by Theorem \ref{t: main}, for $k\in\llbracket2,\,\mathfrak{q}\rrbracket$,
\begin{align*}
\lim_{t\to0}f_{\bm{m}}^{(1)}(t) & =d_{{\rm TV}}\left(\delta_{\bm{m}},\,\nu^{(1)}\right)\,,\\
\lim_{t\to\infty}f_{\bm{m}}^{(k-1)}(t)=\lim_{t\to0}f_{\bm{m}}^{(k)}(t) & =d_{{\rm TV}}\left(\mathfrak{a}_{\bm{m}}^{(k-1)},\,\nu^{(k)}\right)\,,\\
\lim_{t\to\infty}f_{\bm{m}}^{(\mathfrak{q})}(t) & =0\,.
\end{align*}
\end{remark}

\subsubsection{Mixing time}

Finally, we present the sharp asymptotics for the mixing time. Fix
$H_{0}>0$ sufficiently large so that $\{U<H_{0}\}$ is
connected and $\{U<H_{0}\}$ contains all critical points of $U$.
For $H\ge H_{0}$, define
\[
\mathcal{K}(H):=\{U\le H\}\,.
\]
For each $H\ge H_{0}$ and $\delta\in(0,\,1)$, define the $\delta$-mixing
time ${\color{blue}T_{\epsilon}^{{\rm mix}}(\delta;\,H)}$ as the
time required for the total variation distance between the distribution
of process $\bm{x}_{\epsilon}(\cdot)$ and its invariant distribution
to be at most $\delta$ in the worst case, i.e.,
\[
T_{\epsilon}^{{\rm mix}}(\delta;\,H):=\inf\,\{\,t\ge0\,:\,\sup_{\bm{x}\in\mathcal{K}(H)}d_{{\rm TV}}(\bm{x}_{\epsilon}(t;\,\bm{x}),\,\mu_{\epsilon})\le\delta\,\}\,.
\]
Since ${\bf y}^{(\mathfrak{q})}$ has only one irreducible class $\mathscr{R}_{1}^{(\mathfrak{q})}$,
${\bf y}^{(\mathfrak{q})}$ is ergodic. Therefore, one can define
the $\delta$-mixing time for the Markov chain ${\bf y}^{(\mathfrak{q})}$
as
\[
{\color{blue}\mathfrak{T}^{{\rm mix}}(\delta)}:=\inf\{\,t\ge0\,:\,\max_{\mathcal{M}\in\mathscr{V}^{(\mathfrak{q})}}d_{{\rm TV}}({\bf y}^{(\mathfrak{q})}(t;\,\mathcal{M}),\,\nu^{(\mathfrak{q})})\le\delta\,\}\,.
\]
The last main result of this article reads as follows. 
\begin{theorem}
\label{t: mixing_time}Suppose that $\bm{b}=-(\nabla U+\bm{\ell})$
for some $U\in C^{3}(\mathbb{R}^{d})$ and $\bm{\ell}\in C^{2}(\mathbb{R}^{d};\,\mathbb{R}^{d})$
satisfying \eqref{e_decb}, \eqref{e: growth}, and that Assumption
\ref{assu: hetero} is in force. Then, for all $H\ge H_{0}$ and $\delta>0$,
\[
\lim_{\epsilon\to0}\frac{T_{\epsilon}^{{\rm mix}}(\delta;\,H)}{\theta_{\epsilon}^{(\mathfrak{q})}}=\mathfrak{T}^{{\rm mix}}(\delta)\,.
\]
\end{theorem}

\begin{remark}
We consider the potential $U$ to grow quadratically at infinity in this article. As a result, the starting point in the previous theorem should be restricted to the compact set $\mathcal{K}(H)$ for $H \ge H_0$. However, we expect that if the potential grows sufficiently fast at infinity, since the drift will push the dynamics into the compact set strongly enough, allowing the starting point to be generalized to $\mathbb{R}^d$.
\end{remark}

\begin{remark}
In this article, we assume that the potential $U$ is a Morse function. The metastability of diffusion processes without this assumption, particularly regarding the mean transition time between metastable states, was studied in \cite{Avelin, BG-EK}. We expect that if metastability, in the sense of finite-dimensional cvergence, is also achieved, our result can be extended to diffusion processes with potentials that are not Morse functions.
\end{remark}

\subsection{Double-well case and Eyring--Kramers formula}

The Eyring--Kramers formula denotes a precise asymptotics for the mean transition
time between metastable states. In this subsection, we compare our
result with the Eyring--Kramers formula for a simple case. Suppose
that $\mathcal{M}_{0}=\{\bm{m}_{1},\,\bm{m}_{2}\}$ and there exists
a unique saddle point $\bm{\sigma}$ separating $\bm{m}_{1}$ and
$\bm{m}_{2}$ as in Figure \ref{fig: double-well}. Furthermore, suppose that $U(\bm{m}_{1})>U(\bm{m}_{2})$.
For $\bm{x}\in\mathbb{R}^{d}$ and $r>0$, denote by {\color{blue}$B(\bm{x},\,r)$}
the open ball centred at $\bm{x}$ with radius $r>0$:
\[
B(\bm{x},\,r):=\{\bm{y}\in\mathbb{R}^{d}:|\bm{x}-\bm{y}|<r\}\,.
\]
The sharp asymptotics for the mean transition time $\mathbb{E}_{\bm{m}_{1}}^{\epsilon}[\tau_{B(\bm{m}_{2},\,\epsilon)}]$
was derived in \cite{BEGK, LS-22, MenzSch} as follows.
\begin{theorem}
\label{t_EK}
We have
\[
\mathbb{E}_{\bm{m}_{1}}^{\epsilon}[\tau_{B(\bm{m}_{2},\,\epsilon)}]=[1+a_{\epsilon}]\frac{2\pi}{\mu(\boldsymbol{\sigma})}\sqrt{\frac{-\det(\nabla^{2}U)(\boldsymbol{\sigma})}{\det(\nabla^{2}U)(\bm{m}_{1})}}\exp\frac{U(\bm{\sigma})-U(\bm{m}_{1})}{\epsilon}\,,
\]
for some sequence $(a_{\epsilon})_{\epsilon>0}$ such that $\lim_{\epsilon}a_{\epsilon}=0$.
\end{theorem}

The mixing time is related to the Eyring--Kramers formula as follows.
\begin{corollary}
For all $H\ge H_{0}$ and $\delta>0$,
\[
T_{\epsilon}^{{\rm mix}}(\delta;\,H)=[1+a_{\epsilon}]\,\mathbb{E}_{\bm{m}_{1}}^{\epsilon}[\tau_{B(\bm{m}_{2},\,\epsilon)}]\,\log\frac{1}{\delta}\,,
\]
for some sequence $(a_{\epsilon})_{\epsilon>0}$ such that $\lim_{\epsilon}a_{\epsilon}=0$.
\end{corollary}

\begin{proof}
In this case, there is a unique time scale $e^{[U(\bm\sigma)-U(\bm{m}_1)]/\epsilon}$ and limiting Markov chain $\bf{y}^{(1)}$ on state space $\{\bm m_1,\, \bm m_2\}$ with jump rates (cf. Appendix \ref{app: Tree})
\[
\begin{aligned}r^{(1)}(\bm{m}_{1},\,\bm{m}_{2}) & =\frac{\omega(\bm{\sigma})}{\nu(\bm{m}_{1})}=\frac{\mu(\bm{\sigma})}{2\pi}\sqrt{\frac{\det(\nabla^{2}U)(\bm{m}_{1})}{-\det(\nabla^{2}U)(\boldsymbol{\sigma})}}\,,\\
r^{(1)}(\bm{m}_{2},\,\bm{m}_{1}) & =0\,.
\end{aligned}
\]
Since $\bm{m}_{2}$ is the only absorbing state, the unique invariant distribution
of ${\bf y}^{(1)}(\cdot)$ is $\delta_{\bm{m}_{2}}$ and $d_{{\rm TV}}({\bf y}^{(1)}(t;\,\bm m_2),\,\delta_{\bm{m}_{2}})=0$ so that
\[
\begin{aligned}\max_{i\in\{1,\,2\}}d_{{\rm TV}}({\bf y}^{(1)}(t;\,\bm m_i),\,\delta_{\bm{m}_{2}}) & =d_{{\rm TV}}({\bf y}^{(1)}(t;\,1),\,\delta_{\bm{m}_{2}})\\
 & =\frac{1}{2}\left|\mathcal{Q}_{\bm{m}_{1}}^{(1)}[{\bf y}^{(1)}(t)=\bm m _1]\right|+\frac{1}{2}\left|1-\mathcal{Q}_{\bm{m}_{1}}^{(1)}[{\bf y}^{(1)}(t)=\bm m_2]\right|\\
 & =\mathcal{Q}_{\bm{m}_{1}}^{(1)}[{\bf y}^{(1)}(t)=\bm m_1]\\
 & =\mathcal{Q}_{\bm{m}_{1}}^{(1)}[\tau_{\bm{m}_{2}}>t]\,.
\end{aligned}
\]
Since $\mathcal{Q}_{\bm{m}_{1}}^{(1)}[\tau_{\bm{m}_{2}}>t]=e^{-r^{(1)}(\bm{m}_{1},\,\bm{m}_{2})t}$,
\[
\mathfrak{T}^{{\rm mix}}(\delta)=\inf\{t\ge0:e^{-r^{(1)}(\bm{m}_{1},\,\bm{m}_{2})t}\le\delta\}=\frac{1}{r^{(1)}(\bm{m}_{1},\,\bm{m}_{2})}\log\frac{1}{\delta}\,.
\]
Therefore, by Theorem \ref{t: mixing_time},
\[
\lim_{\epsilon\to0}\frac{T_{\epsilon}^{{\rm mix}}(\delta;\,H)}{e^{[U(\bm{\sigma})-U(\bm{m}_{1})]/\epsilon}}=\frac{1}{r^{(1)}(\bm{m}_{1},\,\bm{m}_{2})}\log\frac{1}{\delta}=\frac{2\pi}{\mu(\boldsymbol{\sigma})}\sqrt{\frac{-\det(\nabla^{2}U)(\boldsymbol{\sigma})}{\det(\nabla^{2}U)(\bm{m}_{1})}}\log\frac{1}{\delta}\,.
\]
Finally, Theorem \ref{t_EK} completes the proof.
\end{proof}

\subsection*{Organization of the article}

In Section \ref{sec: meta_TV}, we present a general framework to
prove Theorem \ref{t: main_TV} (cf. Definition \ref{def_TV}) assuming
convergence in finite-dimensional distribution (cf. Definition \ref{def_FC})
and mixing condition (cf. Definition \ref{def_mixing}). In Section
\ref{sec: pf_t_main_TV}, we prove Theorem \ref{t: main_TV} assuming
the mixing condition (cf. Proposition \ref{p: TV-mix}). We prove general
properties of the total variation distance in Section \ref{sec: TV_dist},
which are crucial to prove the main theorems. The proofs of Theorems \ref{t: main}
and \ref{t: mixing_time} are presented in Section \ref{sec: pf_t_main}
and Section \ref{sec: pf_mixing_time}, respectively. Finally, in
Section \ref{sec: pf_p_TV-mix}, we prove Proposition \ref{p: TV-mix}.

\section{\label{sec: meta_TV}Convergence in total variation distance}

In this section, we provide a general framework to prove Theorem \ref{t: main_TV}.
The results of this section are independent of the model and can be extended
to general Markov processes. 

Let \textcolor{blue}{$\{\bm{x}_{\epsilon}(t)\}_{t \ge 0}$}, $\epsilon>0$, be a positive-recurrent
diffusion process on $\mathbb{R}^{d}$ and let \textcolor{blue}{$\mu_{\epsilon}$}
be the unique stationary distribution.
Let \textcolor{blue}{$\mathscr{V}$} be a finite set with at least
two elements. The metastable behavior of $\bm{x}_{\epsilon}(\cdot)$
can be characterized by three factors: a time-scale $1\prec{\color{blue}\theta_{\epsilon}}\prec\infty$,
metastable valleys \textcolor{blue}{$\mathscr{E}=\{\mathcal{E}_{i}\,:\,i\in\mathscr{V}\}$},
which is a collection of disjoint sets $\mathcal{E}_{i}\subset\mathbb{R}^{d}$,
and limiting Markov chain \textcolor{blue}{$\{{\bf y}(t)\}_{t\ge0}$}
on $\mathscr{V}$.
\begin{remark}
We develop the approach presented in \cite{LLM}. In particular, the
following assumption
\[
\lim_{\epsilon\to0}\frac{\mu_{\epsilon}(\mathbb{R}^{d}\setminus\bigcup_{i\in\mathscr{V}}\mathcal{E}_{i})}{\mu_{\epsilon}(\mathcal{E}_{j})}=0\ \ ;\ \ j\in\mathscr{V}
\]
as in display (2.10) in \cite{LLM} is not necessary.
\end{remark}

\subsection{Modes of convergence}

In this subsection, we present two ways of describing metastable behavior
and their relation. The first one is convergence in finite-dimensional
distribution, which corresponds to \cite[Definition 3.2]{LLS-2nd}.
Denote by ${\color{blue}\mathcal{Q}_{i}}$ the law of the process ${\bf y}(\cdot)$ starting at $i \in \mathscr{V}$.
\begin{definition}[Finite-dimensional distributions]
\label{def_FC}We say that the condition $\mathfrak{C}_{{\rm fdd}}(\theta_{\epsilon},\,\mathscr{E},\,{\bf y})$
holds if for all $i\in\mathscr{V}$, $n\ge1$, $0<t_{1}<\cdots<t_{n}$,
sequences $(t_{j,\,\epsilon})_{\epsilon>0}$, $j\in\llbracket1,\,n\rrbracket$,
such that $t_{j,\,\epsilon}\to t_{j}$, and $k_{1},\,\dots,\,k_{n}\in\mathscr{V}$,
we have 
\[
\lim_{\epsilon\rightarrow0}\sup_{\boldsymbol{x}\in\mathcal{E}_{i}}\left|\,\mathbb{P}_{\boldsymbol{x}}^{\epsilon}\Big[\,\bigcap_{j=1}^{n}\{\boldsymbol{x}_{\epsilon}(\theta_{\epsilon}\,t_{j,\,\epsilon})\in\mathcal{E}_{k_{j}}\}\,\Big]-\mathcal{Q}_{i}\Big[\,\bigcap_{j=1}^{n}\{{\bf y}(t_{j})=k_{j}\}\,\Big]\,\right|=0\,,
\]
\end{definition}

The other is convergence in total variation distance. For $\bm{x}\in\mathbb{R}^{d}$
and $0<t_{1}<\cdots<t_{n}$, denote by \textcolor{blue}{$\mu_{\epsilon,\,\bm{x}}^{t_{1},\,\dots,\,t_{n}}(\cdot)$}
the joint law of $\left(\boldsymbol{x}_{\epsilon}(\theta_{\epsilon}t_{1}),\,\dots,\,\boldsymbol{x}_{\epsilon}(\theta_{\epsilon}t_{n})\right)$
with $\bm{x}_{\epsilon}(0)=\bm{x}$ and let \textcolor{blue}{$\mu_{\epsilon}^{\mathcal{A}}$},
$\mathcal{A}\subset\mathbb{R}^{d}$, be the conditioned measure of $\mu_\epsilon$ on $\mathcal{A}$.
\begin{definition}[Total variation distance]
\label{def_TV}We say that the condition $\mathfrak{C}_{{\rm TV}}(\theta_{\epsilon},\,\mathscr{E},\,{\bf y})$
holds if for all $i\in\mathscr{V}$, $n\in\mathbb{N}$, $0<t_{1}<\cdots<t_{n}$,
sequences $(t_{j,\,\epsilon})_{\epsilon>0}$, $j\in\llbracket1,\,n\rrbracket$,
such that $t_{j,\,\epsilon}\to t_{j}$, we have 
\[
\lim_{\epsilon\to0}\sup_{\bm{x}\in\mathcal{E}_{i}}d_{{\rm TV}}\left(\,\mu_{\epsilon,\,\bm{x}}^{t_{1,\,\epsilon},\,\dots,\,t_{n,\,\epsilon}}(\cdot),\,\sum_{k_{1},\,\dots,\,k_{n}\in\mathscr{V}}\,\mathcal{Q}_{i}\left[\,\bigcap_{j=1}^{n}\{{\bf y}(t_{j})=k_{j}\}\,\right]\,\mu_{\epsilon}^{\mathcal{E}_{k_{1}}}\times\cdots\times\mu_{\epsilon}^{\mathcal{E}_{k_{n}}}\,\right)=0\ .
\]
\end{definition}

Notice that $\mathfrak{C}_{{\rm TV}}$ is stronger than $\mathfrak{C}_{{\rm fdd}}$.
Now, we introduce another condition, which implies $\mathfrak{C}_{{\rm TV}}$
together with $\mathfrak{C}_{{\rm fdd}}$. Let ${\color{blue}\mathscr{U}:=\{\mathcal{U}_{i}:i\in\mathscr{V}\}}$
be a collection of connected subsets of $\mathbb{R}^{d}$ satisfying
for $i,\,j\in\mathscr{V}$ such that $i\ne j$,
\begin{equation}
\mathcal{E}_{i}\subset\mathcal{U}_{i}\,,\ \overline{\mathcal{U}_{i}}\cap\overline{\mathcal{U}_{j}}=\varnothing\,,\ \lim_{\epsilon\to0}\,\frac{\mu_{\epsilon}(\mathcal{E}_{i})}{\mu_{\epsilon}(\mathcal{U}_{i})}=1\,.\label{e: E-U}
\end{equation}
For a connected subset $\mathcal{A}\subset\mathbb{R}^{d}$, let \textcolor{blue}{$\bm{x}_{\epsilon}^{\mathcal{A}}(\cdot)$}
be a diffusion $\bm{x}_{\epsilon}(\cdot)$ reflected at the boundary
of $\mathcal{A}$ and denote by \textcolor{blue}{$\mathbb{P}_{\bm{x}}^{\epsilon,\,\mathcal{A}}$}
be the law of $\bm{x}_{\epsilon}^{\mathcal{A}}(\cdot)$ starting at
$\bm{x}\in\mathcal{A}$. The following is the mixing condition which corresponds
to condition (M2) in \cite{LLM}.
\begin{definition}[Mixing condition]
\label{def_mixing}We say that the condition $\mathfrak{M}(\theta_{\epsilon},\,\mathscr{E},\,\mathscr{U})$
holds if there exists a time scale $\rho_{\epsilon}\prec\theta_{\epsilon}$
such that for each $i\in\mathscr{V}$,
\begin{enumerate}
\item the process $\bm{x}_{\epsilon}(\cdot)$ starting from the well $\mathcal{E}_{i}$
cannot escape from the well $\mathcal{U}_{i}$ within the time scale
$\rho_{\epsilon}$, i.e.,
\[
\lim_{\epsilon\to0}\sup_{\bm{x}\in\mathcal{E}_{i}}\mathbb{P}_{\bm{x}}^{\epsilon}\left[\tau_{\partial\mathcal{U}_{i}}\le\rho_{\epsilon}\right]=0\,,
\]
\item and we have 
\[
\lim_{\epsilon\to0}\sup_{\bm{x}\in\mathcal{E}_{i}}d_{{\rm TV}}(\bm{x}_{\epsilon}^{\mathcal{U}_{i}}(\rho_{\epsilon};\,\bm{x}),\,\mu_{\epsilon}^{\mathcal{U}_{i}})=0\,.
\]
\end{enumerate}
The following is the main result of this section.
\end{definition}

\begin{proposition}
\label{p: TV}Conditions $\mathfrak{C}_{{\rm fdd}}(\theta_{\epsilon},\,\mathscr{E},\,{\bf y})$
and $\mathfrak{M}(\theta_{\epsilon},\,\mathscr{E},\,\mathscr{U})$
imply condition $\mathfrak{C}_{{\rm TV}}(\theta_{\epsilon},\,\mathscr{E},\,{\bf y})$.
\end{proposition}

\subsection{Proof of Proposition \ref{p: TV}}

We need the following lemma which is derived from $\mathfrak{M}$.
From now on in this article, all subsets of $\mathbb{R}^{d}$ are
Lebesgue measurable.
\begin{lemma}
\label{l: conv_TV}Suppose that $\mathfrak{M}(\theta_{\epsilon},\,\mathscr{E},\,\mathscr{U})$
holds and let $\rho_{\epsilon}\prec\theta_{\epsilon}$ be given in
$\mathfrak{M}(\theta_{\epsilon},\,\mathscr{E},\,\mathscr{U})$. Then,
for any $i\in\mathscr{V}$,
\[
\limsup_{\epsilon\to0}\sup_{\bm{x}\in\mathcal{E}_{i}}\sup_{\mathcal{A}\in\mathscr{B}(\mathbb{R}^{d})}\left|\,\mathbb{P}_{\bm{x}}\left[\bm{x}_{\epsilon}(\rho_{\epsilon})\in\mathcal{A}\right]-\mu_{\epsilon}^{\mathcal{E}_{i}}(\mathcal{A})\,\right|=0\,.
\]
\end{lemma}

\begin{proof}
We have
\[
\mathbb{P}_{\bm{x}}^{\epsilon}\left[\bm{x}_{\epsilon}(\rho_{\epsilon})\in\mathcal{A}\right]=\mathbb{P}_{\bm{x}}^{\epsilon}\left[\bm{x}_{\epsilon}(\rho_{\epsilon})\in\mathcal{A},\,\tau_{\mathcal{U}_{i}}>\rho_{\epsilon}\right]+R_{\epsilon}^{(1)}(\bm{x},\,\mathcal{A})\,,
\]
where
\[
|\,R_{\epsilon}^{(1)}(\bm{x},\,\mathcal{A})\,|\le\mathbb{P}_{\bm{x}}^{\epsilon}\left[\tau_{\mathcal{U}_{i}}\le\rho_{\epsilon}\right]\,.
\]
By the coupling of $\bm{x}_{\epsilon}$ and $\bm{x}_{\epsilon}^{\mathcal{U}_{i}}$,
\begin{align*}
\mathbb{P}_{\bm{x}}^{\epsilon}\left[\bm{x}_{\epsilon}(\rho_{\epsilon})\in\mathcal{A},\,\tau_{\mathcal{U}_{i}}>\rho_{\epsilon}\right] & =\mathbb{P}_{\bm{x}}^{\epsilon,\,\mathcal{U}_{i}}\left[\bm{x}_{\epsilon}^{\mathcal{U}_{i}}(\rho_{\epsilon})\in\mathcal{A},\,\tau_{\mathcal{U}_{i}}>\rho_{\epsilon}\right]\\
 & =\mathbb{P}_{\bm{x}}^{\epsilon,\,\mathcal{U}_{i}}\left[\bm{x}_{\epsilon}^{\mathcal{U}_{i}}(\rho_{\epsilon})\in\mathcal{A}\right]+R_{\epsilon}^{(2)}(\bm{x},\,\mathcal{A})\,,
\end{align*}
where
\[
|\,R_{\epsilon}^{(2)}(\bm{x},\,\mathcal{A})\,|\le\mathbb{P}_{\bm{x}}^{\epsilon,\,\mathcal{U}_{i}}\left[\tau_{\mathcal{U}_{i}}\le\rho_{\epsilon}\right]=\mathbb{P}_{\bm{x}}^{\epsilon}\left[\tau_{\mathcal{U}_{i}}\le\rho_{\epsilon}\right]\,.
\]
Therefore, by $\mathfrak{M}(\theta_{\epsilon},\,\mathscr{E},\,\mathscr{U})$-(1),
\[
\limsup_{\epsilon\to0}\sup_{\bm{x}\in\mathcal{E}_{i}}\sup_{\mathcal{A}\subset\mathbb{R}^{d}}\left|\,\mathbb{P}_{\bm{x}}^{\epsilon}\left[\bm{x}_{\epsilon}(\rho_{\epsilon})\in\mathcal{A}\right]-\mathbb{P}_{\bm{x}}^{\epsilon,\,\mathcal{U}_{i}}\left[\bm{x}_{\epsilon}^{\mathcal{U}_{i}}(\rho_{\epsilon})\in\mathcal{A}\right]\,\right|=0\,.
\]
Now, by $\mathfrak{M}(\theta_{\epsilon},\,\mathscr{E},\,\mathscr{U})$-(2)
and triangle inequality, we have
\[
\limsup_{\epsilon\to0}\sup_{\bm{x}\in\mathcal{E}_{i}}\sup_{\mathcal{A}\subset\mathbb{R}^{d}}\left|\,\mathbb{P}_{\bm{x}}\left[\bm{x}_{\epsilon}(\rho_{\epsilon})\in\mathcal{A}\right]-\mu_{\epsilon}^{\mathcal{U}_{i}}(\mathcal{A})\,\right|=0\,.
\]
Therefore, it suffices to prove
\[
\limsup_{\epsilon\to0}\sup_{\mathcal{A}\subset\mathbb{R}^{d}}\left|\,\mu_{\epsilon}^{\mathcal{U}_{i}}(\mathcal{A})-\mu_{\epsilon}^{\mathcal{E}_{i}}(\mathcal{A})\,\right|=0\,.
\]
This is obvious from \eqref{e: E-U} and the fact that
\[
\left|\,\mu_{\epsilon}^{\mathcal{U}_{i}}(\mathcal{A}\cap\mathcal{U}_{i})-\mu_{\epsilon}^{\mathcal{U}_{i}}(\mathcal{A}\cap\mathcal{E}_{i})\,\right|\le\mu_{\epsilon}^{\mathcal{U}_{i}}(\mathcal{U}_{i}\setminus\mathcal{E}_{i})\,,
\]
and
\[
\left|\,\mu_{\epsilon}^{\mathcal{U}_{i}}(\mathcal{A}\cap\mathcal{E}_{i})-\mu_{\epsilon}^{\mathcal{E}_{i}}(\mathcal{A}\cap\mathcal{E}_{i})\,\right|=\mu_{\epsilon}(\mathcal{A}\cap\mathcal{E}_{i})\left|\,\frac{1}{\mu_{\epsilon}(\mathcal{U}_{i})}-\frac{1}{\mu_{\epsilon}(\mathcal{E}_{i})}\,\right|\le\frac{\mu_{\epsilon}(\mathcal{U}_{i}\setminus\mathcal{E}_{i})}{\mu_{\epsilon}(\mathcal{U}_{i})}\,.
\]
\end{proof}
Now, we are ready to prove Proposition \ref{p: TV}.
\begin{proof}[Proof of Proposition \ref{p: TV}]
Fix $i\in\mathscr{V}$. Let $\bm{x}\in\mathcal{E}_{i}$ and $\mathcal{A}\subset\mathbb{R}^{d}$.
It suffices to prove for $1$ dimensional case. Fix $t>0$ and sequence
$t_{\epsilon}\to t$. Let $\rho_{\epsilon}\prec\theta_{\epsilon}$
be given in $\mathfrak{M}(\theta_{\epsilon},\,\mathscr{E},\,\mathscr{U})$.
Then,
\begin{align}
 & \mathbb{P}_{\bm{x}}^{\epsilon}\left[\,\bm{x}_{\epsilon}(\theta_{\epsilon}\,t_{\epsilon})\in\mathcal{A}\,\right]\label{e: pf_t_TV-1}\\
 & =\sum_{j\in\mathscr{V}}\mathbb{P}_{\bm{x}}^{\epsilon}\left[\,\bm{x}_{\epsilon}(\theta_{\epsilon}\,t_{\epsilon}-\rho_{\epsilon})\in\mathcal{E}_{j},\,\bm{x}_{\epsilon}(\theta_{\epsilon}\,t_{\epsilon})\in\mathcal{A}\,\right]+R_{\epsilon}^{(1)}(\bm{x},\,\mathcal{A})\,,\nonumber 
\end{align}
where
\[
R_{\epsilon}^{(1)}(\bm{x},\,\mathcal{A})=\mathbb{P}_{\bm{x}}^{\epsilon}\left[\,\bm{x}_{\epsilon}(\theta_{\epsilon}\,t_{\epsilon}-\rho_{\epsilon})\notin\bigcup_{k\in\mathscr{V}}\mathcal{E}_{k},\,\bm{x}_{\epsilon}(\theta_{\epsilon}\,t_\epsilon)\in\mathcal{A}\,\right]\,.
\]
Since $t_{\epsilon}-\rho_{\epsilon}/\theta_{\epsilon}\to t$, by $\mathfrak{C}_{{\rm fdd}}(\theta_{\epsilon},\,\mathscr{E},\,{\bf y})$,
\[
\limsup_{\epsilon\to0}\sup_{\bm{x}\in\mathcal{E}_{i}}\sup_{\mathcal{A}\subset\mathbb{R}^{d}}\,|\,R_{\epsilon}^{(1)}(\bm{x},\,\mathcal{A})\,|\le\limsup_{\epsilon\to0}\sup_{\bm{x}\in\mathcal{E}_{i}}\mathbb{P}_{\bm{x}}^{\epsilon}\left[\,\bm{x}_{\epsilon}(\theta_{\epsilon}\,t_{\epsilon}-\rho_{\epsilon})\notin\bigcup_{k\in\mathscr{V}}\mathcal{E}_{k}\right]=0\,.
\]

By the Markov property, the summation in the right-hand side of \eqref{e: pf_t_TV-1}
can be written as
\begin{equation}
\sum_{j\in\mathscr{V}}\mathbb{E}_{\bm{x}}^{\epsilon}\left[{\bf 1}\left\{ \bm{x}_{\epsilon}(\theta_{\epsilon}\,t_{\epsilon}-\rho_{\epsilon})\in\mathcal{E}_{j}\right\} \mathbb{P}_{\bm{x}_{\epsilon}(\theta_{\epsilon}\,t_{\epsilon}-\rho_{\epsilon})}^{\epsilon}\left[\bm{x}_{\epsilon}(\rho_{\epsilon})\in\mathcal{A}\right]\right]\,.\label{e: pf_t_TV-2}
\end{equation}
By Lemma \ref{l: conv_TV}, we have
\[
\limsup_{\epsilon\to0}\sup_{\mathcal{A}\subset\mathbb{R}^{d}}\sup_{\bm{x}\in\mathcal{E}_{i}}\left|\,\mathbb{P}_{\bm{x}}^{\epsilon}\left[\bm{x}_{\epsilon}(\rho_{\epsilon})\in\mathcal{A}\right]-\mu_{\epsilon}^{\mathcal{E}_{i}}(\mathcal{A})\,\right|=0\,,
\]
so that \eqref{e: pf_t_TV-2} can be written as
\begin{equation}
\sum_{j\in\mathscr{V}}\mathbb{P}_{\bm{x}}^{\epsilon}\left[\bm{x}_{\epsilon}(\theta_{\epsilon}\,t_{\epsilon}-\rho_{\epsilon})\in\mathcal{E}_{j}\right]\,\mu_{\epsilon}^{\mathcal{E}_{j}}(\mathcal{A})+R_{\epsilon}^{(2)}(\bm{x},\,\mathcal{A})\label{e: pf_t_TV-3}
\end{equation}
where
\[
\limsup_{\epsilon\to0}\sup_{\bm{x}\in\mathcal{E}_{i}}\sup_{\mathcal{A}\subset\mathbb{R}^{d}}\,|\,R_{\epsilon}^{(2)}(\bm{x},\,\mathcal{A})\,|=0\,.
\]
The sum of \eqref{e: pf_t_TV-3} can be written as
\[
\sum_{j\in\mathscr{V}}\mathcal{Q}_{i}\left[{\bf y}(t)=j\right]\mu_{\epsilon}^{\mathcal{E}_{j}}(\mathcal{A})+R_{\epsilon}^{(3)}(\bm{x},\,\mathcal{A})\,,
\]
where
\[
\sup_{\bm{x}\in\mathcal{E}_{i}}|\,R_{\epsilon}^{(3)}(\bm{x},\,\mathcal{A})\,|\le\sup_{\bm{x}\in\mathcal{E}_{i}}\left|\,\mathbb{P}_{\bm{x}}^{\epsilon}\left[\bm{x}_{\epsilon}(\theta_{\epsilon}\,t_{\epsilon}-\rho_{\epsilon})\in\mathcal{E}_{j}\right]-\mathcal{Q}_{i}\left[{\bf y}(t)=j\right]\,\right|\,.
\]
Since $t_{\epsilon}-\rho_{\epsilon}/\theta_{\epsilon}\to t$, by $\mathfrak{C}_{{\rm fdd}}(\theta_{\epsilon},\,\mathscr{E},\,{\bf y})$,
\[
\limsup_{\epsilon\to0}\sup_{\bm{y}\in\mathcal{E}_{i}}\sup_{\mathcal{B}\subset\mathbb{R}^{d}}\,|\,R_{\epsilon}^{(3)}(\bm{y},\,\mathcal{B})\,|=0\,.
\]

So far, we obtain
\[
\mathbb{P}_{\bm{x}}^{\epsilon}\left[\,\bm{x}_{\epsilon}(\theta_{\epsilon}\,t_{\epsilon})\in\mathcal{A}\,\right]=\sum_{j\in\mathscr{V}}\mathcal{Q}_{i}\left[{\bf y}(t)=j\right]\,\mu_{\epsilon}^{\mathcal{E}_{j}}(\mathcal{A})+R_{\epsilon}(\bm{x},\,\mathcal{A})
\]
where
\[
\limsup_{\epsilon\to0}\sup_{\bm{x}\in\mathcal{E}_{i}}\sup_{\mathcal{A}\subset\mathbb{R}^{d}}\,|\,R_{\epsilon}(\bm{x},\,\mathcal{A})\,|=0\,.
\]
Therefore, we conclude
\[
\lim_{\epsilon\to0}\sup_{\bm{x}\in\mathcal{E}_{i}}d_{{\rm TV}}\left(\,\mu_{\epsilon,\,\bm{x}}^{t_{\epsilon}},\,\sum_{j\in\mathscr{V}}\,\mathcal{Q}_{i}\left[\,{\bf y}(t)=j\,\right]\,\mu_{\epsilon}^{\mathcal{E}_{j}}\,\right)=0\,.
\]
\end{proof}

\section{\label{sec: pf_t_main_TV}Proof of Theorem \ref{t: main_TV}}

In this section, we prove Theorem \ref{t: main_TV} based on the argument
of the previous section.

\subsection{Convergence in total variation distance}

Let ${\color{blue}R^{(1)}}:=2r_{0}$. Then, by \eqref{e: r_0-0} and
\eqref{e: r_0}, $R^{(1)}<d^{(1)}$ and there is no critical point
$\bm{c}$ satisfying $U(\bm{c})=U(\bm{m})+R^{(1)}$ for some $\bm{m}\in\mathcal{M}_{0}$.
For $n\in\llbracket2,\,\mathfrak{q}\rrbracket$, fix a positive number ${\color{blue}R^{(n)}}\in(d^{(n-1)}+r_{0},\,d^{(n)})$
such that there is no critical point $\bm{c}$
satisfying $U(\bm{c})=U(\bm{m})+R^{(n)}$ for some $\bm{m}\in\mathcal{M}_{0}$.
Since there are finitely many critical points, by \eqref{e: r_0-0},
such $R^{(n)}$ exists. 
\begin{lemma}
\label{l: U_M-0}Fix $n\in\llbracket1,\,\mathfrak{q}\rrbracket$ and
$\mathcal{M}\in\mathscr{V}^{(n)}$. Then, there exists a connected
component of $\{U<U(\mathcal{M})+R^{(n)}\}$ containing $\mathcal{M}$.
\end{lemma}

For $n\in\llbracket1,\,\mathfrak{q}\rrbracket$ and $\mathcal{M}\in\mathscr{V}^{(n)}$,
denote by \textcolor{blue}{$\mathcal{U}^{(n)}(\mathcal{M})$} the connected
component of $\{U<U(\mathcal{M})+R^{(n)}\}$ containing $\mathcal{M}$.
Then, by Proposition \ref{p_FW}, for all $A<R^{(n)}-r_{0}$,
\begin{equation}
\limsup_{\epsilon\rightarrow0}\,\sup_{\boldsymbol{x}\in\mathcal{E}(\mathcal{M})}\,\mathbb{P}_{\boldsymbol{x}}^{\epsilon}\left[\,\tau_{\partial\mathcal{U}^{(n)}}<e^{A/\epsilon}\,\right]=0\,.\label{e: FW_U}
\end{equation}
 
\begin{lemma}
\label{l: U_M}For $n\in\llbracket1,\,\mathfrak{q}\rrbracket$ and
$\mathcal{M},\,\mathcal{M}'\in\mathscr{V}^{(n)}$ such that $\mathcal{M}\ne\mathcal{M}'$,
\[
\mathcal{E}(\mathcal{M})\subset\mathcal{U}^{(n)}(\mathcal{M})\,,\ \overline{\mathcal{U}^{(n)}(\mathcal{M})}\cap\overline{\mathcal{U}^{(n)}(\mathcal{M}')}=\varnothing\,,\ \lim_{\epsilon\to0}\,\frac{\mu_{\epsilon}(\mathcal{E}(\mathcal{M}))}{\mu_{\epsilon}(\mathcal{U}^{(n)}(\mathcal{M}))}=1\,.
\]
\end{lemma}

The proofs are presented in Appendix \ref{app: pf_l_U_M} since it
needs several notions defined in \cite{LLS-2nd}.

For $p\in\llbracket1,\,\mathfrak{q}\rrbracket$, write
\[
{\color{blue}\mathscr{U}^{(p)}}:=\{\mathcal{U}^{(p)}(\mathcal{M})\,:\,\mathcal{M}\in\mathscr{V}^{(p)}\}\, ,\ {\color{blue}\mathscr{E}^{(p)}}:=\{\mathcal{E}(\mathcal{M})\,:\,\mathcal{M}\in\mathscr{V}^{(p)}\}
\]
and
\[
{\color{blue}\mathfrak{M}^{(p)}}:=\mathfrak{M}(\theta_{\epsilon}^{(p)},\,\mathscr{E}^{(p)},\,\mathscr{U}^{(p)})\;,\ {\color{blue}\mathfrak{C}_{{\rm fdd}}^{(p)}}:=\mathfrak{C}_{{\rm fdd}}(\theta_{\epsilon}^{(p)},\,\mathscr{E}^{(p)},\,{\bf y}^{(p)})\ ,\ {\color{blue}\mathfrak{\mathfrak{C}}_{{\rm TV}}^{(p)}}:=\mathfrak{C}_{{\rm TV}}(\theta_{\epsilon}^{(p)},\,\mathscr{E}^{(p)},\,{\bf y}^{(p)})\,.
\]
By Lemma \ref{l: U_M}, elements of $\mathscr{U}^{(p)}$ satisfy \eqref{e: E-U}.
Now, we have the following proposition whose proof is postponed to
Section \ref{sec: pf_p_TV-mix}.
\begin{proposition}
\label{p: TV-mix}We have the following.
\begin{enumerate}
\item For $p\in\llbracket1,\,\mathfrak{q}-1\rrbracket$, $\mathfrak{C}_{{\rm TV}}^{(p)}$
implies $\mathfrak{M}^{(p+1)}$.
\item Condition $\mathfrak{M}^{(1)}$ holds with $\rho_{\epsilon}=\epsilon^{-1}$.
\end{enumerate}
\end{proposition}

\begin{corollary}
\label{c: TV}Conditions $\mathfrak{C}_{{\rm TV}}^{(1)},\,\dots,\,\mathfrak{C}_{{\rm TV}}^{(\mathfrak{q})}$
hold.
\end{corollary}

\begin{proof}
By \cite[Theorem 3.1]{LLS-2nd}, $\mathfrak{C}_{{\rm fdd}}^{(1)},\,\dots,\,\mathfrak{C}_{{\rm fdd}}^{(\mathfrak{q})}$
hold. By Propositions \ref{p: TV-mix}-(2) and \ref{p: TV}, $\mathfrak{C}_{{\rm TV}}^{(1)}$
holds. Finally, the induction argument and Propositions \ref{p: TV-mix},
\ref{p: TV} complete the proof.
\end{proof}

\subsection{Generalization of starting points}

Note that in $\mathfrak{C}_{{\rm TV}}^{(p)}$, $p\in\llbracket1,\,\mathfrak{q}\rrbracket$,
the starting points are restricted to $\mathcal{E}(\mathcal{M})$,
$\mathcal{M}\in\mathscr{V}^{(p)}$. In this subsection, we generalize
the starting points. We first recall the following condition corresponding
to \cite[Definition 3.10]{LLS-2nd}, which is proved in \cite[Proposition 3.13]{LLS-2nd}.
For $p\in\llbracket1,\,\mathfrak{q}\rrbracket$, denote by
\[
{\color{blue}\mathcal{E}^{(p)}}:=\bigcup_{\mathcal{M}\in\mathscr{V}^{(p)}}\mathcal{E}(\mathcal{M})\,.
\]

\begin{proposition}[{\cite[Proposition 3.13]{LLS-2nd}}]
\label{p: Cond_H}For $p\in\llbracket1,\,\mathfrak{q}\rrbracket$,
we have the following.
\begin{enumerate}
\item For all sequence $(\alpha_{\epsilon})_{\epsilon>0}$ such that $\alpha_{\epsilon}\succ\theta_{\epsilon}^{(p-1)}$,
\begin{equation}
\lim_{\epsilon\to0}\max_{\mathcal{M}\in\mathscr{N}^{(p)}}\sup_{\boldsymbol{x}\in\mathcal{E}(\mathcal{M})}\mathbb{P}_{\boldsymbol{x}}^{\epsilon}\left[\tau_{\mathcal{E}^{(p)}}>\alpha_{\epsilon}\right]=0\,.\label{eq:Hp-1}
\end{equation}
\item For all $\mathcal{M}\in\mathscr{N}^{(p)}$ and $\mathcal{M}'\in\mathscr{V}^{(p)}$,
\begin{equation}
\lim_{\epsilon\to0}\sup_{\boldsymbol{x}\in\mathcal{E}(\mathcal{M})}\left|\,\mathbb{P}_{\boldsymbol{x}}^{\epsilon}[\tau_{\mathcal{E}^{(p)}}=\tau_{\mathcal{E}(\mathcal{M}')}]-\widehat{\mathcal{Q}}_{\mathcal{M}}^{(p)}[\tau_{\mathscr{V}^{(p)}}=\tau_{\mathcal{M}'}\,]\,\right|=0\,.\label{eq:Hp-2}
\end{equation}
\end{enumerate}
\end{proposition}

Now, we extend the starting point to all neighborhoods of local minima.
\begin{lemma}
\label{l: main1}Let $p\in\llbracket1,\,\mathfrak{q}\rrbracket$,
$\bm{m}\in\mathcal{M}_{0}$, and $t>0$. Then, for all $\bm{x}\in\mathcal{E}(\bm{m})$
and any sequence $(t_{\epsilon})_{\epsilon>0}$ such that $t_{\epsilon}\to t$
as $\epsilon\to0$,
\[
\mathbb{P}_{\bm{x}}^{\epsilon}\left[\bm{x}_{\epsilon}(\theta_{\epsilon}^{(p)}\,t_{\epsilon})\in\mathcal{A}\right]=\sum_{\mathcal{M}'\in\mathscr{V}^{(p)}}\mathcal{Q}_{\mathfrak{a}_{\bm{m}}^{(p-1)}}^{(p)}\left[\,{\bf y}^{(p)}(t)=\mathcal{M}'\,\right]\,\mu_{\epsilon}^{\mathcal{E}(\mathcal{M}')}(\mathcal{A})+R_{\epsilon}(\bm{x},\,\mathcal{A})\,,
\]
where
\[
\limsup_{\epsilon\to0}\sup_{\bm{x}\in\mathcal{E}(\bm{m})}\sup_{\mathcal{A}\subset\mathbb{R}^{d}}|\,R_{\epsilon}(\bm{x},\,\mathcal{A})\,|=0\,.
\]
\end{lemma}

\begin{proof}
Let $\mathcal{M}=\mathcal{M}(p,\,\bm{m})$. If $\mathcal{M}\in\mathscr{V}^{(p)}$, since $\mathfrak{a}_{\bm{m}}^{(p-1)}(\mathcal{M}'')={\bf 1}\{\mathcal{M}''=\mathcal{M}(p,\,\bm{m})\}$, the claim holds.
Now, suppose that $\mathcal{M}\in\mathscr{N}^{(p)}$ and $\bm{x}\in\mathcal{E}(\bm{m})$.
Let $\mathcal{A}\subset\mathbb{R}^{d}$ be a Lebesgue measurable subset.
It suffices to prove for $1$ dimensional case. Fix $t>0$ and a sequence
$t_{\epsilon}\to t$. Let $\theta_{\epsilon}^{(p-1)}\prec\alpha_{\epsilon}\prec\theta_{\epsilon}^{(p)}$.
Then, we have
\begin{equation}
\mathbb{P}_{\bm{x}}^{\epsilon}\left[\bm{x}_{\epsilon}(\theta_{\epsilon}^{(p)}\,t_{\epsilon})\in\mathcal{A}\right]=\mathbb{P}_{\bm{x}}^{\epsilon}\left[\bm{x}_{\epsilon}(\theta_{\epsilon}^{(p)}\,t_{\epsilon})\in\mathcal{A},\,\tau_{\mathcal{E}^{(p)}}\le\alpha_{\epsilon}\right]+R_{\epsilon}^{(1)}(\bm{x},\,\mathcal{A})\label{e: pf_l_main-1}
\end{equation}
where
\[
|\,R_{\epsilon}^{(1)}(\bm{x},\,\mathcal{A})\,|\le\mathbb{P}_{\boldsymbol{x}}^{\epsilon}\left[\tau_{\mathcal{E}^{(p)}}>\alpha_{\epsilon}\right]\,.
\]
By the strong Markov property, the last probability of \eqref{e: pf_l_main-1}
is equal to
\begin{align*}
 & \mathbb{E}_{\bm{x}}^{\epsilon}\left[\mathbb{P}_{\bm{x}_{\epsilon}(\tau_{\mathcal{E}^{(p)}})}^{\epsilon}\left[\,\bm{x}_{\epsilon}(\theta_{\epsilon}^{(p)}\,t_{\epsilon}-\tau_{\mathcal{E}^{(p)}})\in\mathcal{A}\,\right]\,{\bf 1}\{\,\tau_{\mathcal{E}^{(p)}}\le\alpha_{\epsilon}\,\}\right]\\
 & =\sum_{\mathcal{M}''\in\mathscr{V}^{(p)}}\mathbb{E}_{\bm{x}}^{\epsilon}\left[\mathbb{P}_{\bm{x}_{\epsilon}(\tau_{\mathcal{E}(\mathcal{M}'')})}^{\epsilon}\left[\,\bm{x}_{\epsilon}(\theta_{\epsilon}^{(p)}\,t_{\epsilon}-\tau_{\mathcal{E}^{(p)}})\in\mathcal{A}\,\right]\,{\bf 1}\{\,\tau_{\mathcal{E}^{(p)}}=\tau_{\mathcal{E}(\mathcal{M}'')}\le\alpha_{\epsilon}\,\}\right]
\end{align*}
By $\mathfrak{C}_{{\rm TV}}^{(p)}$, the sum of the above displayed
equation can be written as
\begin{equation}
\begin{aligned} & \sum_{\mathcal{M}''\in\mathscr{V}^{(p)}}\mathbb{E}_{\bm{x}}^{\epsilon}\left[\sum_{\mathcal{M}'\in\mathscr{V}^{(p)}}\mathcal{Q}_{\mathcal{M}''}^{(p)}\left[\,{\bf y}^{(p)}(t)=\mathcal{M}'\,\right]\,\mu_{\epsilon}^{\mathcal{E}(\mathcal{M}')}(\mathcal{A})\,{\bf 1}\{\mathcal{B}_\epsilon(\mathcal{M}'')\}\right]+R_{\epsilon}^{(2)}(\bm{x},\,\mathcal{A})\\
 & =\sum_{\mathcal{M}''\in\mathscr{V}^{(p)}}\mathbb{P}_{\bm{x}}^{\epsilon}\left[\mathcal{B}_\epsilon(\mathcal{M}'')\right]\sum_{\mathcal{M}'\in\mathscr{V}^{(p)}}\mathcal{Q}_{\mathcal{M}''}^{(p)}\left[\,{\bf y}^{(p)}(t)=\mathcal{M}'\,\right]\,\mu_{\epsilon}^{\mathcal{E}(\mathcal{M}')}(\mathcal{A})+R_{\epsilon}^{(2)}(\bm{x},\,\mathcal{A})\,,
\end{aligned}
\label{e: pf_l_main-2}
\end{equation}
where $\mathcal{B}_\epsilon(\mathcal{M}''):=\{\tau_{\mathcal{E}^{(p)}}=\tau_{\mathcal{E}(\mathcal{M}'')}\le\alpha_{\epsilon}\}$ and
\begin{equation}
\limsup_{\epsilon\to0}\sup_{\bm{x}\in\mathcal{E}(\bm{m})}\sup_{\mathcal{A}\subset\mathbb{R}^{d}}\,|\,R_{\epsilon}^{(2)}(\bm{x},\,\mathcal{A})\,|=0\,.\label{e: pf_l_main-3}
\end{equation}
 By \eqref{eq:Hp-2}, the last sum of \eqref{e: pf_l_main-2} can
be written as
\begin{align*}
 & \sum_{\mathcal{M}''\in\mathscr{V}^{(p)}}\widehat{\mathcal{Q}}_{\mathcal{M}}^{(p)}[\tau_{\mathscr{V}^{(p)}}=\tau_{\mathcal{M}''}\,]\sum_{\mathcal{M}'\in\mathscr{V}^{(p)}}\mathcal{Q}_{\mathcal{M}''}^{(p)}\left[\,{\bf y}^{(p)}(t)=\mathcal{M}'\,\right]\,\mu_{\epsilon}^{\mathcal{E}(\mathcal{M}')}(\mathcal{A})+R_{\epsilon}^{(3)}(\bm{x},\,\mathcal{A})\\
 & =\sum_{\mathcal{M}'\in\mathscr{V}^{(p)}}\mathcal{Q}_{\mathfrak{a}_{\bm{m}}^{(p-1)}}^{(p)}\left[\,{\bf y}^{(p)}(t)=\mathcal{M}'\,\right]\,\mu_{\epsilon}^{\mathcal{E}(\mathcal{M}')}(\mathcal{A})+R_{\epsilon}^{(3)}(\bm{x},\,\mathcal{A})\,,
\end{align*}
where
\[
|\,R_{\epsilon}^{(3)}(\bm{x},\,\mathcal{A})\,|\le\mathbb{P}_{\bm{x}}^{\epsilon}\left[\tau_{\mathcal{E}^{(p)}}>\alpha_{\epsilon}\right]\,.
\]

So far, we obtain
\[
\mathbb{P}_{\bm{x}}^{\epsilon}\left[\bm{x}_{\epsilon}(\theta_{\epsilon}^{(p)}\,t_{\epsilon})\in\mathcal{A}\right]=\sum_{\mathcal{M}'\in\mathscr{V}^{(p)}}\mathcal{Q}_{\mathfrak{a}_{\bm{m}}^{(p-1)}}^{(p)}\left[\,{\bf y}^{(p)}(t)=\mathcal{M}'\,\right]\,\mu_{\epsilon}^{\mathcal{E}(\mathcal{M}')}(\mathcal{A})+R_{\epsilon}(\bm{x},\,\mathcal{A})
\]
where
\[
|\,R_{\epsilon}(\bm{x},\,\mathcal{A})\,|\le2\mathbb{P}_{\boldsymbol{x}}^{\epsilon}\left[\tau_{\mathcal{E}^{(p)}}>\alpha_{\epsilon}\right]+|\,R_{\epsilon}^{(2)}(\bm{x},\,\mathcal{A})\,|\,.
\]
By \eqref{eq:Hp-1} and \eqref{e: pf_l_main-3}, we have
\[
\limsup_{\epsilon\to0}\sup_{\bm{x}\in\mathcal{E}(\bm{m})}\sup_{\mathcal{A}\subset\mathbb{R}^{d}}|\,R_{\epsilon}(\bm{x},\,\mathcal{A})\,|=0\,,
\]
which completes the proof.
\end{proof}
Next, we generalize the starting point to the domain of attraction. From
now on, for $p\in\llbracket1,\thinspace\mathfrak{q}\rrbracket$, $\bm{m}\in\mathcal{M}_{0}$,
$t>0$, and $\mathcal{M}\in\mathscr{V}^{(p)}$, let us write
\begin{equation}
{\color{blue}P_{t}^{(p),\,\bm{m}}(\mathcal{M})}:=\mathcal{Q}_{\mathfrak{a}_{\bm{m}}^{(p-1)}}^{(p)}\left[\,{\bf y}^{(p)}(t)=\mathcal{M}\,\right]\,.\label{e: P_t-Q}
\end{equation}

\begin{lemma}
\label{l: main2}Fix $p\in\llbracket1,\,\mathfrak{q}\rrbracket$,
compact subset $\mathcal{K}\subset\mathbb{R}^{d}$, and $t>0$. Then,
for all $\bm{x}\in\mathcal{K}$, subset $\mathcal{A}\subset\mathbb{R}^{d}$,
and any sequence $(t_{\epsilon})_{\epsilon>0}$ such that $t_{\epsilon}\to t$
as $\epsilon\to0$, we have
\begin{align*}
 & \mathbb{P}_{\bm{x}}^{\epsilon}\left[\bm{x}_{\epsilon}(\theta_{\epsilon}^{(p)}\,t_{\epsilon})\in\mathcal{A}\right]\\
 & =\sum_{\mathcal{M}\in\mathscr{V}^{(p)}}\sum_{\bm{m}\in\mathcal{M}_{0}}\mathbb{P}_{\bm{x}}^{\epsilon}\left[\tau_{\mathcal{E}(\mathcal{M}_{0})}=\tau_{\mathcal{E}(\bm{m})}\right]P_{t}^{(p),\,\bm{m}}(\mathcal{M})\,\mu_{\epsilon}^{\mathcal{E}(\mathcal{M})}(\mathcal{A})+R_{\epsilon}(\bm{x},\,\mathcal{A})\,,
\end{align*}
where 
\[
\limsup_{\epsilon\to0}\sup_{\bm{x}\in\mathcal{K}}\sup_{\mathcal{A}\subset\mathbb{R}^{d}}|\,R_{\epsilon}(\bm{x},\,\mathcal{A})\,|=0\,.
\]
\end{lemma}

\begin{proof}
Fix a compact subset $\mathcal{K}\subset\mathbb{R}^{d}$, $p\in\llbracket1,\,\mathfrak{q}\rrbracket$,
and $t>0$. Let $\bm{x}\in\mathcal{K}$ and $\mathcal{A}\subset\mathbb{R}^{d}$.
Then, we have
\begin{equation}
\mathbb{P}_{\bm{x}}^{\epsilon}\left[\bm{x}_{\epsilon}(\theta_{\epsilon}^{(p)}\,t_{\epsilon})\in\mathcal{A}\right]=\mathbb{P}_{\bm{x}}^{\epsilon}\left[\bm{x}_{\epsilon}(\theta_{\epsilon}^{(p)}\,t_{\epsilon})\in\mathcal{A},\,\tau_{\mathcal{E}(\mathcal{M}_{0})}\le\epsilon^{-1}\right]+R_{\epsilon}^{(1)}(\bm{x},\,\mathcal{A})\label{e: pf_l_main2-1}
\end{equation}
where
\[
|\,R_{\epsilon}^{(1)}(\bm{x},\,\mathcal{A})\,|\le\mathbb{P}_{\boldsymbol{x}}^{\epsilon}\left[\tau_{\mathcal{E}(\mathcal{M}_{0})}>\epsilon^{-1}\right]\ .
\]
By the strong Markov property, the last probability of \eqref{e: pf_l_main2-1}
is equal to
\begin{align*}
 & \mathbb{E}_{\bm{x}}^{\epsilon}\left[\mathbb{P}_{\bm{x}_{\epsilon}(\tau_{\mathcal{E}(\mathcal{M}_{0})})}^{\epsilon}\left[\,\bm{x}_{\epsilon}(\theta_{\epsilon}^{(p)}\,t_{\epsilon}-\tau_{\mathcal{E}(\mathcal{M}_{0})})\in\mathcal{A}\,\right]{\bf 1}\{\tau_{\mathcal{E}(\mathcal{M}_{0})}\le\epsilon^{-1}\}\right]\\
 & =\sum_{\bm{m}\in\mathcal{M}_{0}}\mathbb{E}_{\bm{x}}^{\epsilon}\left[\mathbb{P}_{\bm{x}_{\epsilon}(\tau_{\mathcal{E}(\bm{m})})}^{\epsilon}\left[\,\bm{x}_{\epsilon}(\theta_{\epsilon}^{(p)}\,t_{\epsilon}-\tau_{\mathcal{E}(\mathcal{M}_{0})})\in\mathcal{A}\,\right]\,{\bf 1}\{\tau_{\mathcal{E}(\mathcal{M}_{0})}=\tau_{\mathcal{E}(\bm{m})}\le\epsilon^{-1}\}\right]
\end{align*}
By Lemma \ref{l: main1}, the sum of the above displayed equation
can be written as
\begin{align}
 & \sum_{\bm{m}\in\mathcal{M}_{0}}\mathbb{E}_{\bm{x}}^{\epsilon}\left[\sum_{\mathcal{M}\in\mathscr{V}^{(p)}}P_{t}^{(p),\,\bm{m}}(\mathcal{M})\,\mu_{\epsilon}^{\mathcal{E}(\mathcal{M})}(\mathcal{A})\,{\bf 1}\{\tau_{\mathcal{E}(\mathcal{M}_{0})}=\tau_{\mathcal{E}(\bm{m})}\le\epsilon^{-1}\}\right]+R_{\epsilon}^{(2)}(\bm{m},\,\mathcal{A})\nonumber \\
 & =\sum_{\bm{m}\in\mathcal{M}_{0}}\mathbb{P}_{\bm{x}}^{\epsilon}\left[\tau_{\mathcal{E}(\mathcal{M}_{0})}=\tau_{\mathcal{E}(\bm{m})}\le\epsilon^{-1}\right]\sum_{\mathcal{M}\in\mathscr{V}^{(p)}}P_{t}^{(p),\,\bm{m}}(\mathcal{M})\,\mu_{\epsilon}^{\mathcal{E}(\mathcal{M})}(\mathcal{A})+R_{\epsilon}^{(2)}(\bm{x},\,\mathcal{A})\,,\label{e: pf_l_main2-2}
\end{align}
where
\begin{equation}
\limsup_{\epsilon\to0}\sup_{\mathcal{A}\subset\mathbb{R}^{d}}|\,R_{\epsilon}^{(2)}(\bm{m},\,\mathcal{A})\,|=0\,.\label{e: pf_l_main2-3}
\end{equation}
The sum of \eqref{e: pf_l_main2-2} can be written as
\[
\sum_{\mathcal{M}\in\mathscr{V}^{(p)}}\sum_{\bm{m}\in\mathcal{M}_{0}}\mathbb{P}_{\bm{x}}^{\epsilon}\left[\tau_{\mathcal{E}(\mathcal{M}_{0})}=\tau_{\mathcal{E}(\bm{m})}\right]P_{t}^{(p),\,\bm{m}}(\mathcal{M})\,\mu_{\epsilon}^{\mathcal{E}(\mathcal{M})}(\mathcal{A})+R_{\epsilon}^{(3)}(\bm{x},\,\mathcal{A})
\]
where
\[
|\,R_{\epsilon}^{(3)}(\bm{x},\,\mathcal{A})\,|\le\mathbb{P}_{\bm{x}}\left[\tau_{\mathcal{E}(\mathcal{M}_{0})}>\epsilon^{-1}\right]\,.
\]
So far, we obtain
\begin{align*}
 & \mathbb{P}_{\bm{x}}^{\epsilon}\left[\bm{x}_{\epsilon}(\theta_{\epsilon}^{(p)}\,t_{\epsilon})\in\mathcal{A}\right]\\
 & =\sum_{\mathcal{M}\in\mathscr{V}^{(p)}}\sum_{\bm{m}\in\mathcal{M}_{0}}\mathbb{P}_{\bm{x}}^{\epsilon}\left[\tau_{\mathcal{E}(\mathcal{M}_{0})}=\tau_{\mathcal{E}(\bm{m})}\right]P_{t}^{(p),\,\bm{m}}(\mathcal{M})\,\mu_{\epsilon}^{\mathcal{E}(\mathcal{M})}(\mathcal{A})+R_{\epsilon}(\bm{x},\,\mathcal{A})
\end{align*}
where
\[
|\,R_{\epsilon}(\bm{x},\,\mathcal{A})\,|\le2\mathbb{P}_{\bm{x}}\left[\tau_{\mathcal{E}(\mathcal{M}_{0})}>\epsilon^{-1}\right]+\max_{\bm{m}\in\mathcal{M}_{0}}|\,R_{\epsilon}^{(2)}(\bm{m},\,\mathcal{A})\,|
\]
Therefore, by Proposition \ref{p: hit_min}-(1) and \eqref{e: pf_l_main2-3},
we obtain
\[
\limsup_{\epsilon\to0}\sup_{\bm{x}\in\mathcal{K}}\sup_{\mathcal{A}\subset\mathbb{R}^{d}}|\,R_{\epsilon}(\bm{x},\,\mathcal{A})\,|=0\,.
\]
\end{proof}
Now, we are ready to prove Theorem \ref{t: main_TV}.
\begin{proof}[Proof of Theorem \ref{t: main_TV}]
It suffices to prove for $1$ dimensional case. Fix $p\in\llbracket1,\,\mathfrak{q}\rrbracket$,
$\bm{m}\in\mathcal{M}_{0}$, $\bm{x}\in\mathcal{D}(\bm{m})$, $t>0$
and $t_{\epsilon}\to t$. Recall the notation \eqref{e: P_t-Q}. By
Lemma \ref{l: main2}, we have
\begin{align*}
 & \mathbb{P}_{\bm{x}}^{\epsilon}\left[\bm{x}_{\epsilon}(\theta_{\epsilon}^{(p)}\,t_{\epsilon})\in\mathcal{A}\right]\\
 & =\sum_{\mathcal{M}\in\mathscr{V}^{(p)}}\sum_{\bm{m}'\in\mathcal{M}_{0}}\mathbb{P}_{\bm{x}}^{\epsilon}\left[\tau_{\mathcal{E}(\mathcal{M}_{0})}=\tau_{\mathcal{E}(\bm{m}')}\right]\,P_{t}^{(p),\,\bm{m}'}(\mathcal{M})\,\mu_{\epsilon}^{\mathcal{E}(\mathcal{M})}(\mathcal{A})+R_{\epsilon}^{(1)}(\mathcal{A})
\end{align*}
where
\[
\limsup_{\epsilon\to0}\sup_{\mathcal{A}\subset\mathbb{R}^{d}}|\,R_{\epsilon}^{(1)}(\mathcal{A})\,|=0\,.
\]
Therefore, since $\bm{x}\in\mathcal{D}(\bm{m})$, we obtain
\[
\mathbb{P}_{\bm{x}}^{\epsilon}\left[\bm{x}_{\epsilon}(\theta_{\epsilon}^{(p)}\,t_{\epsilon})\in\mathcal{A}\right]=\sum_{\mathcal{M}\in\mathscr{V}^{(p)}}P_{t}^{(p),\,\bm{m}}(\mathcal{M})\,\mu_{\epsilon}^{\mathcal{E}(\mathcal{M})}(\mathcal{A})+R_{\epsilon}(\mathcal{A})
\]
where
\[
|\,R_{\epsilon}(\mathcal{A})\,|\le\mathbb{P}_{\bm{x}}^{\epsilon}\left[\tau_{\mathcal{E}(\mathcal{M}_{0})}\ne\tau_{\mathcal{E}(\bm{m})}\right]+|\,R_{\epsilon}^{(1)}(\mathcal{A})\,|\,.
\]
By Proposition \ref{p: hit_min}-(2),
\[
\limsup_{\epsilon\to0}\sup_{\mathcal{A}\subset\mathbb{R}^{d}}|\,R_{\epsilon}(\mathcal{A})\,|=0\,.
\]
Hence, we conclude
\[
\lim_{\epsilon\to0}d_{{\rm TV}}\left(\bm{x}_{\epsilon}(\theta_{\epsilon}^{(p)}t_{\epsilon};\,\bm{x}),\,\sum_{\mathcal{M}\in\mathscr{V}^{(p)}}P_{t}^{(p),\,\bm{m}}(\mathcal{M})\,\mu_{\epsilon}^{\mathcal{E}(\mathcal{M})}\right)=0\,.
\]
\end{proof}

\section{\label{sec: TV_dist}Total variation distance}

In this section, we prove general properties of total variation distance.
Let $E$ be a certain space (a subset of $\mathbb{R}^{d}$ or finite
set). For any measure $\mu$ on $E$ and a subset $F\subset E$, denote
by $\mu^{F}$ the conditioned measure of $\mu$ on $F$, i.e.,
\[
\mu^{F}(A):=\frac{1}{\mu(F)}\mu(A\cap F)\ \ ;\ \ A\subset E\,.
\]
The first result gives a formula for total variation distance between
conditioned measures. This is crucial to prove the total variation
distance for reduced Markov chain.
\begin{lemma}
\label{l: TV_density}Let $(\pi_{\epsilon})_{\epsilon>0}$ be a sequence
of measure on $E$. Let $A_{i}\subset E$, $i\in\llbracket1,\,n\rrbracket$,
be disjoint sets and let $B\subset E$ satisfy $\cup_{i=1}^{n}A_{i}\subset B$.
Suppose that
\begin{equation}
\nu_{i}:=\lim_{\epsilon\to0}\,\frac{\pi_{\epsilon}(A_{i})}{\pi_{\epsilon}(B)}\in[0,\,1)\,.\label{e: TV_density}
\end{equation}
Then, for any $a_{i}$, $i\in\llbracket1,\,n\rrbracket$, such that
$a_{i}\ge0$ and $\sum_{i=1}^{n}a_{i}=1$, we have
\[
\lim_{\epsilon\to0}d_{{\rm TV}}\left(\sum_{i=1}^{n}a_{i}\,\pi_{\epsilon}^{A_{i}},\,\pi_{\epsilon}^{B}\right)=\sum_{i:a_{i}>\nu_{i}}(a_{i}-\nu_{i})\,.
\]
In addition, if $\sum_{i=1}^{n}\nu_{i}=1$,
\[
\lim_{\epsilon\to0}d_{{\rm TV}}\left(\sum_{i=1}^{n}a_{i}\,\pi_{\epsilon}^{A_{i}},\,\pi_{\epsilon}^{B}\right)=\frac{1}{2}\sum_{i=1}^{n}|\,a_{i}-\nu_{i}\,|\,.
\]
\end{lemma}

\begin{proof}
By definition, we have
\begin{align*}
d_{{\rm TV}}\left(\sum_{i=1}^{n}a_{i}\,\pi_{\epsilon}^{A_{i}},\,\pi_{\epsilon}^{B}\right) & =\sup_{F\subset E}\left|\,\sum_{i=1}^{n}a_{i}\,\pi_{\epsilon}^{A_{i}}(F)-\pi_{\epsilon}^{B}(F)\,\right|\,.\\
 & =\sup_{F\subset E}\left|\,\sum_{i=1}^{n}a_{i}\,\frac{\pi_{\epsilon}(F\cap A_{i})}{\pi_{\epsilon}(A_{i})}-\frac{\pi_{\epsilon}(F\cap B)}{\pi_{\epsilon}(B)}\,\right|\,.
\end{align*}
Since $A_{i}\subset B$ and $A_{i}$ are disjoint, we obtain
\begin{align*}
 & \sum_{i=1}^{n}a_{i}\,\frac{\pi_{\epsilon}(F\cap A_{i})}{\pi_{\epsilon}(A_{i})}-\frac{\pi_{\epsilon}(F\cap B)}{\pi_{\epsilon}(B)}\\
 & =\sum_{i=1}^{n}\frac{\pi_{\epsilon}(F\cap A_{i})}{\pi_{\epsilon}(A_{i})}\left(a_{i}\,-\frac{\pi_{\epsilon}(A_{i})}{\pi_{\epsilon}(B)}\right)-\frac{\pi_{\epsilon}(F\cap(B\setminus\bigcup_{i}A_{i}))}{\pi_{\epsilon}(B)}\\
 & =\sum_{i=1}^{n}\frac{\pi_{\epsilon}(F\cap A_{i})}{\pi_{\epsilon}(A_{i})}\left(a_{i}\,-\frac{\pi_{\epsilon}(A_{i})}{\pi_{\epsilon}(B)}\right)-\frac{\pi_{\epsilon}(F\cap(B\setminus\bigcup_{i}A_{i}))}{\pi_{\epsilon}(B\setminus\bigcup_{i}A_{i})}\left(1-\sum_{i=1}^{n}\frac{\pi_{\epsilon}(A_{i})}{\pi_{\epsilon}(B)}\right)\\
 & =\sum_{i=1}^{n}\frac{\pi_{\epsilon}(F\cap A_{i})}{\pi_{\epsilon}(A_{i})}\left(a_{i}\,-\nu_{i}(\epsilon)\right)-\frac{\pi_{\epsilon}(F\cap(B\setminus\bigcup_{i}A_{i}))}{\pi_{\epsilon}(B\setminus\bigcup_{i}A_{i})}\left(1-\sum_{i=1}^{n}\nu_{i}(\epsilon)\right)\,,
\end{align*}
where $\nu_{i}(\epsilon)=\pi_{\epsilon}(A_{i})/\pi_{\epsilon}(B)$.
Then, In order to make it as large as possible, one can take
\[
F_{\epsilon}=\bigcup_{i:a_{i}>\nu_{i}(\epsilon)}A_{i}\,.
\]
On the other hand, to make it as small as possible, one can take
\[
F_{\epsilon}=\bigcup_{i:a_{i}<\nu_{i}(\epsilon)}A_{i}\cup(B\setminus\bigcup_{i=1}^{n}A_{i})\,.
\]
Therefore, we have
\[
d_{{\rm TV}}\left(\sum_{i=1}^{n}a_{i}\,\pi_{\epsilon}^{A_{i}},\,\pi_{\epsilon}^{B}\right)=\max\left\{ \,\sum_{i:a_{i}>\nu_{i}(\epsilon)}\left(a_{i}-\nu_{i}(\epsilon)\right),\,\sum_{i:a_{i}<\nu_{i}(\epsilon)}\left(\nu_{i}(\epsilon)-a_{i}\right)+1-\sum_{i=1}^{n}\nu_{i}(\epsilon)\,\right\} \,.
\]
Additionally, since $\sum_{i=1}^{n}a_{i}=1$,
\begin{align*}
\sum_{i:a_{i}<\nu_{i}(\epsilon)}\left(\nu_{i}(\epsilon)-a_{i}\right)+1-\sum_{i=1}^{n}\nu_{i}(\epsilon) & =\sum_{i:a_{i}<\nu_{i}(\epsilon)}\left(\nu_{i}(\epsilon)-a_{i}\right)+\sum_{i=1}^{n}\left(a_{i}-\nu_{i}(\epsilon)\right)\\
& =\sum_{i:a_{i}>\nu_{i}(\epsilon)}\left(a_{i}-\nu_{i}(\epsilon)\right)
\end{align*}
so that
\[
d_{{\rm TV}}\left(\sum_{i=1}^{n}a_{i}\,\pi_{\epsilon}^{A_{i}},\,\pi_{\epsilon}^{B}\right)=\sum_{i:a_{i}>\nu_{i}(\epsilon)}\left(a_{i}-\nu_{i}(\epsilon)\right)\,.
\]
Finally, by \eqref{e: TV_density},
\[
\lim_{\epsilon\to0}d_{{\rm TV}}\left(\sum_{i=1}^{n}a_{i}\,\pi_{\epsilon}^{A_{i}},\,\pi_{\epsilon}^{B}\right)=\sum_{i:a_{i}>\nu_{i}}(a_{i}-\nu_{i})\,.
\]
If $\sum_{i=1}^{n}\nu_{i}=1$, the last sum is equal to $\frac{1}{2}\sum_{i=1}^{n}|\,a_{i}-\nu_{i}\,|$.
\end{proof}
Let $\{X(t):t\ge0\}$ be a continuous-time Markov process on $E$.
Let $C_{0}(E)$ be the space of continuous functions vanishing at
infinity if $E=\mathbb{R}^{d}$. If $E$ is finite, this is just the
set of functions on $E$. Let $\mathbf{P}_{t}$, $t\ge0$, be a semigroup
on $C_{0}(E)$ associated to the Markov process $\{X(t):t\ge0\}$.
By the following lemma, we have monotonicity of total variation distance.
\begin{lemma}
\label{l: d_TV}Let $\mu,\,\nu$ be probability measures on $E$.
Then, for any $t>0$,
\[
d_{{\rm TV}}(\mu,\,\nu)\ge d_{{\rm TV}}(\mu\mathbf{P}_{t},\,\nu\mathbf{P}_{t})\ .
\]
\end{lemma}

\begin{proof}
Denote by $\mathbb{P}_{\pi}$ the law of $X(\cdot)$ with initial
distribution $\pi$. By definition, we have
\begin{align*}
d_{{\rm TV}}(\mu\mathbf{P}_{t},\,\nu\mathbf{P}_{t}) & =\sup_{A\subset E}\left|\,\mathbb{P}_{\mu}\left[X(t)\in A\right]-\mathbb{P}_{\nu}\left[X(t)\in A\right]\,\right|\\
 & =\sup_{A\subset E}\left|\,\int_{E}\mathbb{P}_{x}\left[X(t)\in A\right]\,\left(\mu(dx)-\nu(dx)\right)\,\right|\,.
\end{align*}
Let $\left|\,\mu-\nu\,\right|$ be a total variation measure of the singed
measure $\mu-\nu$. Then, we obtain
\[
\int_{E}\mathbb{P}_{x}\left[X(t)\in A\right]\,\left(\mu(dx)-\nu(dx)\right)\le\int_{E}\mathbb{P}_{x}\left[X(t)\in A\right]\,\left|\,\mu-\nu\,\right|(dx)\le\left|\,\mu-\nu\,\right|(E)\,,
\]
so that
\[
d_{{\rm TV}}(\mu\mathbf{P}_{t},\,\nu\mathbf{P}_{t})\le\left|\,\mu-\nu\,\right|(E)\,.
\]
Since$\left|\,\mu-\nu\,\right|(E)=d_{{\rm TV}}(\mu,\,\nu)$, the proof
is completed.
\end{proof}
If $E$ is finite and the Markov chain has one irreducible class,
the inequality in the previous lemma is indeed strict. The following
is crucial to prove the upper bound for the mixing time below (cf.
Lemma \ref{l: mix_up}).
\begin{lemma}
\label{l: d_TV-2}Suppose that $E$ is finite and that the Markov
chain $X(\cdot)$ has only one irreducible class. Let $\mu,\,\nu$
be probability measures on $E$ such that $\mu\ne\nu$. Then, for
any $t>0$,
\[
d_{{\rm TV}}(\mu,\,\nu)>d_{{\rm TV}}(\mu\mathbf{P}_{t},\,\nu\mathbf{P}_{t})\,.
\]
In particular, if $\nu$ is a stationary distribution,
\[
d_{{\rm TV}}(\mu,\,\nu)>d_{{\rm TV}}(\mu\mathbf{P}_{t},\,\nu)\,.
\]
\end{lemma}

\begin{proof}
Let us consider a space $E\times E$ and let
\[
D:=\{\,(x,\,x)\,:\,x\in E\,\}\,.
\]
By \cite[Proposition 4.7]{LPW}, we have the following formula for
the total variation:
\[
d_{{\rm TV}}(\mu,\,\nu)=\inf\{\alpha(D^{c}):\alpha\,\text{is a probability measure on}\,E\times E\,\text{whose marginals}\,\text{are}\,\mu\,\text{and}\,\nu\,\}\,.
\]
By the proof of \cite[Proposition 4.7]{LPW}, there exists a minimizer
of the previous formula. Let $\pi$ be the minimizer. By definition,
\[
\pi(D^{c})=d_{{\rm TV}}(\mu,\,\nu)\,,\ \sum_{y\in E}\pi(x,\,y)=\mu(x),\ \sum_{y\in E}\pi(y,\,x)=\nu(x)\,.
\]

Fix $t>0$ and let $p_{t}:E\times E\to[0,\,1]$ be a transition probability
of $X(\cdot)$. Then, by the definition of $\pi$,
\begin{equation}
\begin{aligned}
\mu\mathbf{P}_{t}(x) & =\sum_{y}\mu(y)p_{t}(y,\,x)=\sum_{y}\sum_{z}\pi(y,\,z)p_{t}(y,\,x)\,,\\
\nu\mathbf{P}_{t}(x) & =\sum_{y}\nu(y)p_{t}(y,\,x)=\sum_{y}\sum_{z}\pi(z,\,y)p_{t}(y,\,x)\,.
\end{aligned}
\label{e: pf_l_d_TV-2-0}
\end{equation}
Let $n:=|E|$ and $E=\{x_1,\, \dots,\, x_n\}$ be a numeration of $E$. Define $c_{\cdot}:E \to E$ as $c_{x_i}=c_{x_{i+1}}$ where $x_{n+1}=x_1$. Then, $c_{\cdot}$ is one-to-one correspondence and satisfies
\begin{equation}
(c_x,\, x) \ne (y,\, c_y) \  \text{for all} \ x,\, y \in E \, .
\label{e: pf_l_d_TV-c}
\end{equation}
Now, define a transition matrix $\widetilde{P}:(E\times E)\times(E\times E)\to[0,\,1]$ as follows.
For $z,\, w,\, x \in E$, if $p_{t}(z,\,x)\le p_{t}(w,\,x)$, define
\[
\widetilde{P}\Big((z,\,w),\,(a,\,b)\Big):=\begin{cases}
\min\{p_{t}(z,\,x),\,p_{t}(w,\,x)\} & a=b=x\,,\\
p_{t}(w,\,x)-p_{t}(z,\,x) & a=c_{x},\,b=x\,,\\
0 & a=x,\,b=c_x\,;
\end{cases}
\]
otherwise, if $p_{t}(w,\,x) < p_{t}(z,\,x)$, define
\[
\widetilde{P}\Big((z,\,w),\,(a,\,b)\Big):=\begin{cases}
\min\{p_{t}(z,\,x),\,p_{t}(w,\,x)\} & a=b=x\,,\\
0 & a=c_x,\,b=x\,,\\
p_{t}(z,\,x)-p_{t}(w,\,x) & a=x,\,b=c_x\,.
\end{cases}
\]
For other cases, define $\widetilde{P}=0$. By \eqref{e: pf_l_d_TV-c}, $\widetilde{P}$ is well-defined. Also, by definition, $\widetilde{P}$ is a transition matrix, $\sum_{y}\widetilde{P}((z,\,w),\,(x,\,y))=p_{t}(z,\,x)$, and $\sum_{y}\widetilde{P}((z,\,w),\,(y,\,x))=p_{t}(w,\,x)$ so that by \eqref{e: pf_l_d_TV-2-0},
\begin{align*}
\sum_{y}\pi\widetilde{P}(x,\,y) & =\sum_{y}\sum_{z,\,w}\pi(z,\,w)\,\widetilde{P}((z,\,w),\,(x,\,y))\\
 & =\sum_{z,\,w}\sum_{y}\pi(z,\,w)\,\widetilde{P}((z,\,w),\,(x,\,y))\\
 & =\sum_{z}\sum_{w}\pi(z,\,w)p_{t}(z,\,x)\\
 & =\mu\mathbf{P}_{t}(x)\,,
\end{align*}
and
\begin{align*}
\sum_{y}\pi\widetilde{P}(y,\,x) & =\sum_{y}\sum_{z,\,w}\pi(z,\,w)\,\widetilde{P}((z,\,w),\,(y,\,x))\\
 & =\sum_{z,\,w}\sum_{y}\pi(z,\,w)\,\widetilde{P}((z,\,w),\,(y,\,x))\\
 & =\sum_{z}\sum_{w}\pi(z,\,w)p_{t}(w,\,x)\\
 & =\nu\mathbf{P}_{t}(x)\,.
\end{align*}
Therefore, 
\[
\pi\widetilde{P}\in\{\,\text{probability measures on}\,E\times E\,\text{whose marginals}\,\text{are}\,\mu\mathbf{P}_{t}\,\text{and}\,\nu\mathbf{P}_{t}\,\}
\]
so that by \cite[Proposition 4.7]{LPW},
\[
d_{{\rm TV}}(\mu\mathbf{P}_{t},\,\nu\mathbf{P}_{t})\le\pi\widetilde{P}(D^{c})\,.
\]

Since $\pi(D^{c})=d_{{\rm TV}}(\mu,\,\nu)$, it suffices to prove
$\pi\widetilde{P}(D^{c})<\pi(D^{c})$, i.e.,
\[
\pi\widetilde{P}(D)>\pi(D)\,.
\]
First, we obtain
\begin{align*}
\pi\widetilde{P}(x,\,x) & =\sum_{z,\,w}\pi(z,\,w)\,\widetilde{P}((z,\,w),\,(x,\,x))\\
 & =\sum_{z,\,w}\pi(z,\,w)\,(p_{t}(z,\,x)\wedge p_{t}(w,\,x))\\
 & =\sum_{z\ne w}\pi(z,\,w)\,(p_{t}(z,\,x)\wedge p_{t}(w,\,x))+\sum_{y}\pi(y,\,y)\,p_{t}(y,\,x)\\
 & \ge\sum_{y}\pi(y,\,y)\,p_{t}(y,\,x)
\end{align*}
so that
\begin{align*}
\sum_{x}\pi\widetilde{P}(x,\,x) & \ge\sum_{x}\sum_{y}\pi(y,\,y)\,p_{t}(y,\,x)\\
 & =\sum_{y}\pi(y,\,y)\sum_{x}\,p_{t}(y,\,x)\\
 & =\sum_{y}\pi(y,\,y)\,,
\end{align*}
which implies
\begin{equation}
\pi\widetilde{P}(D)\ge\pi(D)\,.\label{e: pf_l_d_TV-2}
\end{equation}
Suppose that the equality holds. Then, for all $x\in E$
\[
\sum_{z\ne w}(\pi(z,\,w)+\pi(w,\,z))\,(p_{t}(z,\,x)\wedge p_{t}(w,\,x))=0\,.
\]
Since all terms are positive, for all $x\in E$ and $z\ne w\in E$,
\[
\pi(z,\,w)\,(p_{t}(z,\,x)\wedge p_{t}(w,\,x))=0\,.
\]
Pick an element $x_{0}\in E$ of irreducible class of $X$. Since
there is only one irreducible class, $p_{t}(y,\,x_{0})>0$ for all
$y\in E$. Therefore, by the above displayed equation, $\pi(z,\,w)=0$
for all $z\ne w$, which implies $\pi(D^{c})=0$. However, since $\mu\ne\nu$,
we have 
\[
\pi(D^{c})=d_{{\rm TV}}(\mu,\,\nu)>0\,,
\]
which is a contradiction. In conclusion, the inequality \eqref{e: pf_l_d_TV-2}
is indeed strict and this completes the proof.
\end{proof}

\section{\label{sec: pf_t_main}Proof of Theorem \ref{t: main}}

In this section, we prove Theorem \ref{t: main}.

\subsection{Preliminaries}

First, we present preliminaries on Markov chains which are not obtained
in \cite{LLS-1st,LLS-2nd}. The first one is an equivalent definition
of the probability measure $\mathfrak{a}_{\bm{m}}^{(p)}$ on $\mathscr{V}^{(p+1)}$
(cf. display \eqref{e: def_a}). For simplicity, write
\[
{\color{blue}\mathscr{R}^{(p)}}:=\bigcup_{i\in\llbracket1,\,\mathfrak{n}_{p}\rrbracket}\mathscr{R}_{i}^{(p)}\,,
\]
and recall from \eqref{e: M_j-R_j} the definition of $\mathcal{M}_{i}^{(p+1)}\in\mathscr{V}^{(p+1)}$
for $p\in\llbracket1,\,\mathfrak{q}\rrbracket$ and $i\in\llbracket1,\,\mathfrak{n}_{p}\rrbracket$.
The following is \cite[Lemma 7.3]{LLS-2nd}.
\begin{lemma}[{\cite[Lemma 7.3]{LLS-2nd}}]
\label{l: 44-0}Fix $p\in\llbracket1,\,\mathfrak{q}-1\rrbracket$.
For all $\mathcal{M}\in\mathscr{N}^{(p+1)}$ and $i\in\llbracket1,\,\mathfrak{n}_{p}\rrbracket$,
\[
\widehat{\mathcal{Q}}_{\mathcal{M}}^{(p+1)}\left[\,\tau_{\mathscr{V}^{(p+1)}}=\tau_{\mathcal{M}_{i}^{(p+1)}}\,\right]=\widehat{\mathcal{Q}}_{\mathcal{M}}^{(p)}\left[\,\tau_{\mathscr{R}^{(p)}}=\tau_{\mathscr{R}_{i}^{(p)}}\,\right]\,.
\]
\end{lemma}

\begin{proof}
It suffices to recall the notation. In \cite[Lemma 7.3]{LLS-2nd},
by the construction, $\mathscr{V}^{(p)}(\mathcal{M}_{i}^{(p+1)})=\{\mathcal{M}'\in\mathscr{V}^{(p)}:\mathcal{M}'\subset\mathcal{M}_{i}^{(p+1)}\}=\mathscr{R}_{i}^{(p)}$.
\end{proof}
\begin{lemma}
\label{l: 44}Fix $p\in\llbracket1,\,\mathfrak{q}-1\rrbracket$. For
all $\bm{m}\in\mathcal{M}_{0}$ and $i\in\llbracket1,\,\mathfrak{n}_{p}\rrbracket$,
\[
\mathfrak{a}_{\bm{m}}^{(p)}(\mathcal{M}_{i}^{(p+1)})=\mathcal{Q}_{\mathfrak{a}_{\bm{m}}^{(p-1)}}^{(p)}\left[\tau_{\mathscr{R}^{(p)}}=\tau_{\mathscr{R}_{i}^{(p)}}\right]\,.
\]
\end{lemma}

\begin{proof}
By definition \eqref{e: def_a},
\begin{equation}
\mathfrak{a}_{\bm{m}}^{(p)}(\mathcal{M}_{i}^{(p+1)})=\widehat{\mathcal{Q}}_{\mathcal{M}(p+1,\,\boldsymbol{m})}^{(p+1)}\left[\,\tau_{\mathscr{V}^{(p+1)}}=\tau_{\mathcal{M}_{i}^{(p+1)}}\,\right]\,.\label{e: l_44}
\end{equation}
First, suppose that $\mathcal{M}(p+1,\,\bm{m})\in\mathscr{N}^{(p+1)}$.
Then, by the construction,$\mathcal{M}(p+1,\,\bm{m})\in\mathscr{S}^{(p)}$ so that $\mathcal{M}(p+1,\,\bm{m})=\mathcal{M}(p,\,\bm{m})$.
By Lemma \ref{l: 44-0},
\[
\widehat{\mathcal{Q}}_{\mathcal{M}(p+1,\,\boldsymbol{m})}^{(p+1)}\left[\,\tau_{\mathscr{V}^{(p+1)}}=\tau_{\mathcal{M}_{i}^{(p+1)}}\,\right]=\widehat{\mathcal{Q}}_{\mathcal{M}(p,\,\boldsymbol{m})}^{(p)}\left[\,\tau_{\mathscr{R}^{(p)}}=\tau_{\mathscr{R}_{i}^{(p)}}\,\right]\,.
\]
Since $\mathscr{R}^{(p)}\subset\mathscr{V}^{(p)}$, the last probability
can be decomposed as
\begin{align*}
 & \widehat{\mathcal{Q}}_{\mathcal{M}(p,\,\boldsymbol{m})}^{(p)}\left[\,\tau_{\mathscr{R}^{(p)}}=\tau_{\mathscr{R}_{i}^{(p)}}\,\right]\\
 & =\sum_{\mathcal{M}'\in\mathscr{V}^{(p)}}\widehat{\mathcal{Q}}_{\mathcal{M}(p,\,\boldsymbol{m})}^{(p)}\left[\,\tau_{\mathscr{R}^{(p)}}=\tau_{\mathscr{R}_{i}^{(p)}}\Big|\tau_{\mathscr{V}^{(p)}}=\tau_{\mathcal{M}'}\,\right]\,\widehat{\mathcal{Q}}_{\mathcal{M}(p,\,\boldsymbol{m})}^{(p)}\left[\,\tau_{\mathscr{V}^{(p)}}=\tau_{\mathcal{M}'}\,\right]\\
 & =\sum_{\mathcal{M}'\in\mathscr{V}^{(p)}}\widehat{\mathcal{Q}}_{\mathcal{M}'}^{(p)}\left[\,\tau_{\mathscr{R}^{(p)}}=\tau_{\mathscr{R}_{i}^{(p)}}\,\right]\,\widehat{\mathcal{Q}}_{\mathcal{M}(p,\,\boldsymbol{m})}^{(p)}\left[\,\tau_{\mathscr{V}^{(p)}}=\tau_{\mathcal{M}'}\,\right]\,,
\end{align*}
where the last equality holds by the strong Markov property. By the
definition of $\mathfrak{a}_{\bm{m}}^{(p-1)}$, the last sum can be
written as
\[
\sum_{\mathcal{M}'\in\mathscr{V}^{(p)}}\widehat{\mathcal{Q}}_{\mathcal{M}'}^{(p)}\left[\,\tau_{\mathscr{R}^{(p)}}=\tau_{\mathscr{R}_{i}^{(p)}}\,\right]\mathfrak{a}_{\bm{m}}^{(p-1)}(\mathcal{M}')\,.
\]
Finally, since ${\bf y}^{(p)}$ is the trace process of $\widehat{{\bf y}}^{(p)}$
on $\mathscr{V}^{(p)}$ and $\mathscr{R}^{(p)}\subset\mathscr{V}^{(p)}$,
we obtain
\begin{align*}
\sum_{\mathcal{M}'\in\mathscr{V}^{(p)}}\widehat{\mathcal{Q}}_{\mathcal{M}'}^{(p)}\left[\,\tau_{\mathscr{R}^{(p)}}=\tau_{\mathscr{R}_{i}^{(p)}}\,\right]\mathfrak{a}_{\bm{m}}^{(p-1)}(\mathcal{M}') & =\sum_{\mathcal{M}'\in\mathscr{V}^{(p)}}\mathcal{Q}_{\mathcal{M}'}^{(p)}\left[\,\tau_{\mathscr{R}^{(p)}}=\tau_{\mathscr{R}_{i}^{(p)}}\,\right]\mathfrak{a}_{\bm{m}}^{(p-1)}(\mathcal{M}')\\
 & =\mathcal{Q}_{\mathfrak{a}_{\bm{m}}^{(p-1)}}^{(p)}\left[\,\tau_{\mathscr{R}^{(p)}}=\tau_{\mathscr{R}_{i}^{(p)}}\,\right]\,,
\end{align*}
which completes the proof of the lemma for $\mathcal{M}(p+1,\,\bm{m})\in\mathscr{N}^{(p+1)}$.

Next, suppose that $\mathcal{M}(p+1,\,\bm{m})\in\mathscr{V}^{(p+1)}$.
Then, $\mathcal{M}(p+1,\,\bm{m})=\bigcup_{\mathcal{M}\in\mathscr{R}_{j}^{(p)}}\mathcal{M}$
for some $j\in\llbracket1,\,\mathfrak{n}_{p}\rrbracket$, $\mathcal{M}(p,\,\bm{m})\subset\mathcal{M}(p+1,\,\bm{m})$, and $\mathcal{M}(p,\,\bm{m})\in\mathscr{V}^{(p)}$
so that we have
\[
\widehat{\mathcal{Q}}_{\mathcal{M}(p+1,\,\boldsymbol{m})}^{(p+1)}\left[\,\tau_{\mathscr{V}^{(p+1)}}=\tau_{\mathcal{M}_{i}^{(p+1)}}\,\right]={\bf 1}\{\,\mathcal{M}(p+1,\,\bm{m})=\mathcal{M}_{i}^{(p+1)}\,\}={\bf 1}\{\,\mathcal{M}(p,\,\bm{m})\in\mathscr{R}_{i}^{(p)}\,\}
\]
and
\[
\mathfrak{a}_{\bm{m}}^{(p-1)}(\mathcal{M}')=\widehat{\mathcal{Q}}_{\mathcal{M}(p,\,\boldsymbol{m})}^{(p)}\left[\,\tau_{\mathscr{V}^{(p)}}=\tau_{\mathcal{M}'}\,\right]={\bf 1}\{\,\mathcal{M}'=\mathcal{M}(p,\,\bm{m})\,\}\,.
\]
Therefore, by the above two displayed equations and , we obtain
\begin{align*}
\mathcal{Q}_{\mathfrak{a}_{\bm{m}}^{(p-1)}}^{(p)}\left[\,\tau_{\mathscr{R}^{(p)}}=\tau_{\mathscr{R}_{i}^{(p)}}\,\right] & =\mathcal{Q}_{\mathcal{M}(p,\,\bm{m})}^{(p)}\left[\,\tau_{\mathscr{R}^{(p)}}=\tau_{\mathscr{R}_{i}^{(p)}}\,\right]\\
 & ={\bf 1}\{\,\mathcal{M}(p,\,\bm{m})\in\mathscr{R}_{i}^{(p)}\,\}\\
 & =\widehat{\mathcal{Q}}_{\mathcal{M}(p+1,\,\boldsymbol{m})}^{(p+1)}\left[\,\tau_{\mathscr{V}^{(p+1)}}=\tau_{\mathcal{M}_{i}^{(p+1)}}\,\right]\,.
\end{align*}
Finally, by \eqref{e: l_44}, the proof is completed.
\end{proof}
For $p\in\llbracket1,\,\mathfrak{q}\rrbracket$, define
\[
{\color{blue}\mathscr{R}_{\star}^{(p)}}:=\{i\in\llbracket1,\,\mathfrak{n}_{p}\rrbracket:\mathscr{R}_{i}^{(p)}\subset\mathscr{V}_{\star}^{(p)}\}\,.
\]
By the next lemma, there is one-to-one correspondence between $\mathscr{R}_{\star}^{(p)}$
and $\mathscr{V}_{\star}^{(p+1)}$.
\begin{lemma}
\label{l: R^p_V^p+1}Fix $p\in\llbracket1,\,\mathfrak{q}\rrbracket$
and define $\varphi:\llbracket1,\,\mathfrak{n}_{p}\rrbracket\to\mathscr{V}^{(p+1)}$
by
\[
\varphi(i)=\mathcal{M}_{i}^{(p+1)}\,.
\]
Then, $\varphi$ is one-to-one correspondence such that $\varphi(\mathscr{R}_{\star}^{(p)})=\mathscr{V}_{\star}^{(p+1)}$.
\end{lemma}

\begin{proof}
$\varphi$ is one-to-one correspondence by the construction. If $i\in\mathscr{R}_{\star}^{(p)}$,
$\mathcal{M}\subset\mathcal{M}_{\star}$ for all $\mathcal{M}\in\mathscr{R}_{i}^{(p)}$
so that $\mathcal{M}_{i}^{(p+1)}\subset\mathscr{V}_{\star}^{(p+1)}$
since $\mathcal{M}_{i}^{(p+1)}$ is simple. Therefore, $\varphi(\mathscr{R}_{\star}^{(p)})\subset\mathscr{V}_{\star}^{(p+1)}$.
Let $\mathcal{M}\in\mathscr{V}_{\star}^{(p+1)}$ and let $i\in\llbracket1,\,\mathfrak{n}_{p}\rrbracket$
be such that $\mathcal{M}=\mathcal{M}_{i}^{(p+1)}$. Since $\mathcal{M}\subset\mathcal{M}_{\star}$,
$\mathcal{M}'\subset\mathcal{M}_{\star}$ for $\mathcal{M}'\in\mathscr{R}_{i}^{(p)}$
as well. Therefore, $i\in\mathscr{R}_{\star}^{(p)}$ so that $\varphi(\mathscr{R}_{\star}^{(p)})\supset\mathscr{V}_{\star}^{(p+1)}$,
which completes the proof. 
\end{proof}
Recall the notation \eqref{e: P_t-Q}. Now, we have the following
long time behavior of ${\bf y}^{(p)}$.
\begin{lemma}
\label{l: 46}For $p\in\llbracket1,\,\mathfrak{q}\rrbracket$, $\bm{m}\in\mathcal{M}_{0}$
and $t>0$,
\[
\lim_{t\to\infty}P_{t}^{(p),\,\bm{m}}(\mathcal{M})=\begin{cases}
\mathfrak{a}_{\bm{m}}^{(p)}(\mathcal{M}_{i}^{(p+1)})\frac{\nu(\mathcal{M})}{\nu(\mathcal{M}_{i}^{(p+1)})} & \mathcal{M}\in\mathscr{R}_{i}^{(p)},\,i\in\llbracket1,\,\mathfrak{n}_{p}\rrbracket\,,\\
0 & \mathcal{M}\in\mathscr{T}^{(p)}\,,
\end{cases}
\]
where $\mathfrak{a}_{\bm{m}}^{(\mathfrak{q})}=\delta_{\mathcal{M}_{1}^{(\mathfrak{q}+1)}}$.
\end{lemma}

\begin{proof}
Fix $p\in\llbracket1,\,\mathfrak{q}\rrbracket$, $\bm{m}\in\mathcal{M}_{0}$
and $t>0$. If $\mathcal{M}\in\mathscr{T}^{(p)}$, by finite Markov
chain theory,
\[
\lim_{t\to\infty}P_{t}^{(p),\,\bm{m}}(\mathcal{M})=0\,.
\]

Now, let $\mathcal{M}\in\mathscr{R}_{i}^{(p)}$ for some $i\in\llbracket1,\,\mathfrak{n}_{p}\rrbracket$.
Then,
\begin{align*}
P_{t}^{(p),\,\bm{m}}(\mathcal{M}) & =\sum_{\mathcal{M}'\in\mathscr{R}^{(p)}}\mathcal{Q}_{\mathcal{M}'}^{(p)}\left[\,{\bf y}^{(p)}(t)=\mathcal{M}\,\right]\mathcal{Q}_{\mathfrak{a}_{\bm{m}}^{(p-1)}}^{(p)}\left[\,\tau_{\mathscr{R}^{(p)}}=\tau_{\mathcal{M}'}\,\right]\\
 & =\sum_{\mathcal{M}'\in\mathscr{R}_{i}^{(p)}}\mathcal{Q}_{\mathcal{M}'}^{(p)}\left[\,{\bf y}^{(p)}(t)=\mathcal{M}\,\right]\mathcal{Q}_{\mathfrak{a}_{\bm{m}}^{(p-1)}}^{(p)}\left[\,\tau_{\mathscr{R}^{(p)}}=\tau_{\mathcal{M}'}\,\right]\,.
\end{align*}
By \cite[Proposition 12.7]{LLS-2nd}, ${\bf y}^{(p)}$ restricted on $\mathscr{R}_{i}^{(p)}$ is irreducible
and a probability measure on $\mathscr{R}_{i}^{(p)}$
\[
\frac{\nu(\mathcal{M})}{\nu(\mathcal{M}_{i}^{(p+1)})}\ ;\ \mathcal{M}\in\mathscr{R}_{i}^{(p)}
\]
is the stationary distribution. Therefore, by finite Markov chain theory, we
obtain for $\mathcal{M}'\in\mathscr{R}_{i}^{(p)}$,
\[
\lim_{t\to\infty}\mathcal{Q}_{\mathcal{M}'}^{(p)}\left[\,{\bf y}^{(p)}(t)=\mathcal{M}\,\right]=\frac{\nu(\mathcal{M})}{\nu(\mathcal{M}_{i}^{(p+1)})}\,.
\]
Hence, we have
\[
\lim_{t\to\infty}P_{t}^{(p),\,\bm{m}}(\mathcal{M})=\sum_{\mathcal{M}'\in\mathscr{R}_{i}^{(p)}}\,\frac{\nu(\mathcal{M})}{\nu(\mathcal{M}_{i}^{(p+1)})}\,\mathcal{Q}_{\mathfrak{a}_{\bm{m}}^{(p-1)}}^{(p)}\left[\,\tau_{\mathscr{R}^{(p)}}=\tau_{\mathcal{M}'}\,\right]\,.
\]
Now, we have
\[
\sum_{\mathcal{M}'\in\mathscr{R}_{i}^{(p)}}\mathcal{Q}_{\mathfrak{a}_{\bm{m}}^{(p-1)}}^{(p)}\left[\,\tau_{\mathscr{R}^{(p)}}=\tau_{\mathcal{M}'}\,\right]=\mathcal{Q}_{\mathfrak{a}_{\bm{m}}^{(p-1)}}^{(p)}\left[\,\tau_{\mathscr{\mathscr{R}}^{(p)}}=\tau_{\mathscr{R}_{i}^{(p)}}\,\right]\,,
\]
so that
\[
\lim_{t\to\infty}P_{t}^{(p),\,\bm{m}}(\mathcal{M})=\frac{\nu(\mathcal{M})}{\nu(\mathcal{M}_{i}^{(p+1)})}\,\mathcal{Q}_{\mathfrak{a}_{\bm{m}}^{(p-1)}}^{(p)}\left[\,\tau_{\mathscr{\mathscr{R}}^{(p)}}=\tau_{\mathscr{R}_{i}^{(p)}}\,\right]\,.
\]
Therefore, if $p<\mathfrak{q}$, by Lemma \ref{l: 44}, we obtain
\[
\lim_{t\to\infty}\,P_{t}^{(p),\,\bm{m}}(\mathcal{M})=\frac{\nu(\mathcal{M})}{\nu(\mathcal{M}_{i}^{(p+1)})}\,\mathfrak{a}_{\bm{m}}^{(p)}(\mathcal{M}_{i}^{(p+1)})\,.
\]
If $p=\mathfrak{q}$, since $\mathfrak{n}_{\mathfrak{q}}=1$ (so that
$i=1$ and $\mathscr{R}^{(\mathfrak{q})}=\mathscr{R}_{1}^{(\mathfrak{q})}$),
we obtain
\[
\lim_{t\to\infty}P_{t}^{(p),\,\bm{m}}(\mathcal{M})=\frac{\nu(\mathcal{M})}{\nu(\mathcal{M}_{1}^{(\mathfrak{q}+1)})}\mathcal{Q}_{\mathfrak{a}_{\bm{m}}^{(\mathfrak{q}-1)}}^{(\mathfrak{q})}\left[\,\tau_{\mathscr{\mathscr{R}}^{(\mathfrak{q})}}=\tau_{\mathscr{R}_{1}^{(\mathfrak{q})}}\,\right]=\frac{\nu(\mathcal{M})}{\nu(\mathcal{M}_{1}^{(\mathfrak{q}+1)})}\,.
\]
\end{proof}

\subsection{Total variation distance of reduced processes}

By Theorem \ref{t: main_TV}, the distribution of $\bm{x}_{\epsilon}(\cdot)$
in time scale $\theta_{\epsilon}^{(p)}$ is reduced to that of Markov
chain ${\bf y}^{(p)}(\cdot)$. Therefore, by the triangle inequality,
we may compute the total variation distance between the distribution
of Markov chain ${\bf y}^{(p)}(\cdot)$ and $\mu_{\epsilon}$ instead
of the distance between that of $\bm{x}_{\epsilon}(\cdot)$ and $\mu_{\epsilon}$.
Recall from \eqref{e: nu^p} the definition of the measure $\nu^{(p)}$.
\begin{lemma}
\label{l: 45}For $p\in\llbracket1,\,\mathfrak{q}\rrbracket$, $\bm{m}\in\mathcal{M}_{0}$
and $t>0$,
\[
\lim_{\epsilon\to0}d_{{\rm TV}}\left(\sum_{\mathcal{M}\in\mathscr{V}^{(p)}}P_{t}^{(p),\,\bm{m}}(\mathcal{M})\,\mu_{\epsilon}^{\mathcal{E}(\mathcal{M})},\,\mu_{\epsilon}\right)=d_{{\rm TV}}\left({\bf y}^{(p)}\left(t;\,\mathfrak{a}_{\bm{m}}^{(p-1)}\right),\,\nu^{(p)}\right)\,.
\]
\end{lemma}

\begin{proof}
Fix $p\in\llbracket1,\,\mathfrak{q}\rrbracket$, $\bm{m}\in\mathcal{M}_{0}$
and $t>0$. Then, by Lemma \ref{l: TV_density} for $\pi_{\epsilon}=\mu_{\epsilon}$
and $A_{i}=\mathcal{E}(\mathcal{M})$ and $B=\mathbb{R}^{d}$,
\begin{align*}
 & \lim_{\epsilon\to0}d_{{\rm TV}}\left(\sum_{\mathcal{M}\in\mathscr{V}^{(p)}}P_{t}^{(p),\,\bm{m}}(\mathcal{M})\,\mu_{\epsilon}^{\mathcal{E}(\mathcal{M})},\,\mu_{\epsilon}\right)\\
 & =\frac{1}{2}\sum_{\mathcal{M}\in\mathscr{V}^{(p)}\setminus\mathscr{V}_{\star}^{(p)}}P_{t}^{(p),\,\bm{m}}(\mathcal{M})+\frac{1}{2}\sum_{\mathcal{M}\in\mathscr{V}_{\star}^{(p)}}\left|\,P_{t}^{(p),\,\bm{m}}(\mathcal{M})-\frac{\nu(\mathcal{M})}{\nu_{\star}}\,\right|\,.
\end{align*}
The right-hand side of the previous equation is equal to $d_{{\rm TV}}(P_{t}^{(p),\,\bm{m}},\,\nu^{(p)})$. 
\end{proof}
Now, we compute limit of the total variation distance of the Markov
chain as $t\to\infty$. This is related to the tree-structure of the
Markov chains.
\begin{lemma}
\label{l: 48}For all $p\in\llbracket1,\,\mathfrak{q}\rrbracket$
and $\bm{m}\in\mathcal{M}_{0}$, we have
\[
\lim_{t\to\infty}\,d_{{\rm TV}}\left({\bf y}^{(p)}\left(t;\,\mathfrak{a}_{\bm{m}}^{(p-1)}\right),\,\nu^{(p)}\right)=d_{{\rm TV}}(\mathfrak{a}_{\bm{m}}^{(p)},\,\nu^{(p+1)})\,.
\]
\end{lemma}

\begin{proof}
Fix $p\in\llbracket1,\,\mathfrak{q}\rrbracket$ and $\bm{m}\in\mathcal{M}_{0}$.
Since $\mathscr{V}^{(p)}$ is finite, by definition \eqref{e: P_t-Q}
of $P_{t}^{(p),\,\bm{m}}$, for $t>0$,
\[
d_{{\rm TV}}\left({\bf y}^{(p)}\left(t;\,\mathfrak{a}_{\bm{m}}^{(p-1)}\right),\,\nu^{(p)}\right)=\frac{1}{2}\sum_{\mathcal{M}\notin\mathscr{V}_{\star}^{(p)}}P_{t}^{(p),\,\bm{m}}(\mathcal{M})+\frac{1}{2}\sum_{\mathcal{M}\in\mathscr{V}_{\star}^{(p)}}\left|\,P_{t}^{(p),\,\bm{m}}(\mathcal{M})-\frac{\nu(\mathcal{M})}{\nu_{\star}}\,\right|\,.
\]
 Since $\sum_{\mathcal{M}\in\mathscr{R}_{i}^{(p)}}\nu(\mathcal{M})=\nu(\mathcal{M}_{i}^{(p+1)})$
for all $i\in\llbracket1,\,\mathfrak{n}_{p}\rrbracket$, by Lemma
\ref{l: 46}, we obtain
\begin{align*}
\lim_{t\to\infty}\sum_{\mathcal{M}\notin\mathscr{V}_{\star}^{(p)}}P_{t}^{(p),\,\bm{m}}(\mathcal{M}) & =\lim_{t\to\infty}\sum_{i\in\llbracket1,\,\mathfrak{n}_{p}\rrbracket\setminus\mathscr{R}_{\star}^{(p)}}\sum_{\mathcal{M}\in\mathscr{R}_{i}^{(p)}}P_{t}^{(p),\,\bm{m}}(\mathcal{M})\\
 & =\sum_{i\in\llbracket1,\,\mathfrak{n}_{p}\rrbracket\setminus\mathscr{R}_{\star}^{(p)}}\sum_{\mathcal{M}\in\mathscr{R}_{i}^{(p)}}\mathfrak{a}_{\bm{m}}^{(p)}(\mathcal{M}_{i}^{(p+1)})\frac{\nu(\mathcal{M})}{\nu(\mathcal{M}_{i}^{(p+1)})}\\
 & =\sum_{i\in\llbracket1,\,\mathfrak{n}_{p}\rrbracket\setminus\mathscr{R}_{\star}^{(p)}}\mathfrak{a}_{\bm{m}}^{(p)}(\mathcal{M}_{i}^{(p+1)})\,,
\end{align*}
and
\begin{align*}
\lim_{t\to\infty}\sum_{\mathcal{M}\in\mathscr{V}_{\star}^{(p)}}\left|\,P_{t}^{(p),\,\bm{m}}(\mathcal{M})-\frac{\nu(\mathcal{M})}{\nu_{\star}}\,\right| & =\lim_{t\to\infty}\sum_{i\in\mathscr{R}_{\star}^{(p)}}\sum_{\mathcal{M}\in\mathscr{R}_{i}^{(p)}}\left|\,P_{t}^{(p),\,\bm{m}}(\mathcal{M})-\frac{\nu(\mathcal{M})}{\nu_{\star}}\,\right|\\
 & =\sum_{i\in\mathscr{R}_{\star}^{(p)}}\sum_{\mathcal{M}\in\mathscr{R}_{i}^{(p)}}\left|\,\mathfrak{a}_{\bm{m}}^{(p)}(\mathcal{M}_{i}^{(p+1)})\frac{\nu(\mathcal{M})}{\nu(\mathcal{M}_{i}^{(p+1)})}-\frac{\nu(\mathcal{M})}{\nu_{\star}}\,\right|\\
 & =\sum_{i\in\mathscr{R}_{\star}^{(p)}}\left|\,\mathfrak{a}_{\bm{m}}^{(p)}(\mathcal{M}_{i}^{(p+1)})-\frac{\nu(\mathcal{M}_{i}^{(p+1)})}{\nu_{\star}}\,\right|\,.
\end{align*}
Therefore,
\begin{align*}
 & \lim_{t\to\infty}\,d_{{\rm TV}}\left({\bf y}^{(p)}\left(t;\,\mathfrak{a}_{\bm{m}}^{(p-1)}\right),\,\nu^{(p)}\right)\\
 & =\frac{1}{2}\sum_{i\in\llbracket1,\,\mathfrak{n}_{p}\rrbracket\setminus\mathscr{R}_{\star}^{(p)}}\mathfrak{a}_{\bm{m}}^{(p)}(\mathcal{M}_{i}^{(p+1)})+\frac{1}{2}\sum_{i\in\mathscr{R}_{\star}^{(p)}}\left|\,\mathfrak{a}_{\bm{m}}^{(p)}(\mathcal{M}_{i}^{(p+1)})-\frac{\nu(\mathcal{M}_{i}^{(p+1)})}{\nu_{\star}}\,\right|\,.
\end{align*}
Since there exists a one-to-one correspondence $\varphi:\llbracket1,\,\mathfrak{n}_{p}\rrbracket\to\mathscr{V}^{(p+1)}$
such that $\varphi(i)=\mathcal{M}_{i}^{(p+1)}$ and $\varphi(\mathscr{R}_{\star}^{(p)})=\mathscr{V}_{\star}^{(p+1)}$
by Lemma \ref{l: R^p_V^p+1}, the right-hand side of the previous
equation is equal to
\begin{align*}
 & \frac{1}{2}\sum_{\varphi(i)\in\mathscr{V}^{(p+1)}\setminus\mathscr{V}_{\star}^{(p+1)}}\mathfrak{a}_{\bm{m}}^{(p)}(\varphi(i))+\frac{1}{2}\sum_{\varphi(i)\in\mathscr{V}_{\star}^{(p+1)}}\left|\,\mathfrak{a}_{\bm{m}}^{(p)}(\varphi(i))-\frac{\nu(\varphi(i))}{\nu_{\star}}\,\right|\\
 & =\frac{1}{2}\sum_{\mathcal{M}\in\mathscr{V}^{(p+1)}\setminus\mathscr{V}_{\star}^{(p+1)}}\mathfrak{a}_{\bm{m}}^{(p)}(\mathcal{M})+\frac{1}{2}\sum_{\mathcal{M}\in\mathscr{V}_{\star}^{(p+1)}}\left|\,\mathfrak{a}_{\bm{m}}^{(p)}(\mathcal{M})-\frac{\nu(\mathcal{M})}{\nu_{\star}}\,\right|\\
 & =d_{{\rm TV}}(\mathfrak{a}_{\bm{m}}^{(p)},\,\nu^{(p+1)})\,,
\end{align*}
which completes the proof.
\end{proof}

\subsection{Proof of Theorem \ref{t: main}}

From now on, for $\bm{x}\in\mathbb{R}^{d}$ and $t>0$, we write
\[
{\color{blue}D_{\epsilon}(t,\,\bm{x})}:=d_{{\rm TV}}(\boldsymbol{x}_{\epsilon}(t;\,\bm{x}),\,\mu_{\epsilon})\,.
\]
First, we need the following elementary property.
\begin{lemma}
\label{l: d_TV-1}For all $\bm{x}\in\mathbb{R}^{d}$ and $t_{1}<t_{2}$,
we have
\[
D_{\epsilon}(t_{1},\,\bm{x})\ge D_{\epsilon}(t_{2},\,\bm{x})\,.
\]
\end{lemma}

\begin{proof}
Let $\mu(d\bm{y})$ be the distribution of $\boldsymbol{x}_{\epsilon}(t_{1};\,\bm{x})$
and write $\nu(d\bm{y})=\mu_{\epsilon}(d\bm{y})$, $t=t_{1}-t_{2}$.
The proof is immediate from Lemma \ref{l: d_TV}.
\end{proof}
We need to compute the total variation distance $d_{{\rm TV}}(\boldsymbol{x}_{\epsilon}(t;\,\bm{x}),\,\mu_{\epsilon})$
in a time scale smaller than $\theta_{\epsilon}^{(1)}$ without Theorem
\ref{t: main_TV} since we do not have such result in time scale smaller
than $\theta_{\epsilon}^{(1)}$. The next lemma is the total variation
distance in the smallest time scale.
\begin{lemma}
\label{l: theta^0}Fix $\bm{m}\in\mathcal{M}_{0}$. For $\bm{x}\in\mathcal{D}(\bm{m})$,
we have
\[
\lim_{\epsilon\to0}D_{\epsilon}(2\epsilon^{-1},\,\bm{x})=\begin{cases}
1-\frac{\nu(\bm{m})}{\nu_{\star}} & \bm{m}\in\mathcal{M}_{\star}\,,\\
1 & \bm{m}\notin\mathcal{M}_{\star}\,.
\end{cases}
\]
\end{lemma}

\begin{proof}
Fix $\bm{m}\in\mathcal{M}_{0}$ and $\bm{x}\in\mathcal{D}(\bm{m})$.
Recall the definition of $\mathcal{U}^{(1)}(\bm{m})$ after Lemma
\ref{l: U_M-0}. By triangle inequality, we obtain
\begin{equation}
\left|\,D_{\epsilon}(2\epsilon^{-1},\,\bm{x})-d_{{\rm TV}}\left(\mu_{\epsilon}^{\mathcal{U}^{(1)}(\bm{m})},\,\mu_{\epsilon}\right)\,\right|\le d_{{\rm TV}}\left(\bm{x}_{\epsilon}(2\epsilon^{-1};\,\bm{x}),\,\mu_{\epsilon}^{\mathcal{U}^{(1)}(\bm{m})}\right)\,.\label{e: theta^0}
\end{equation}
By \cite[Theorem 2.1.2]{FW}, there exists $T>0$ such that
\[
\lim_{\epsilon\to0}\mathbb{P}_{\bm{x}}^{\epsilon}\left[\bm{x}_{\epsilon}(T)\notin\mathcal{E}(\bm{m})\right]=0\,.
\]
Therefore, by the Markov property, for $\mathcal{A}\subset\mathbb{R}^{d}$,
we have
\[
\mathbb{P}_{\bm{x}}^{\epsilon}\left[\bm{x}_{\epsilon}(2\epsilon^{-1})\in\mathcal{A}\right]=\mathbb{E}_{\bm{x}}^{\epsilon}\left[\mathbb{P}_{\bm{x}_{\epsilon}(T)}^{\epsilon}\left[\bm{x}_{\epsilon}(2\epsilon^{-1}-T)\in\mathcal{A}\right],\,{\bf 1}\{\bm{x}_{\epsilon}(T)\in\mathcal{E}(\bm{m})\}\right]+R_{\epsilon}^{(1)}\,,
\]
where
\[
\limsup_{\epsilon\to0}|\,R_{\epsilon}^{(1)}\,|=0\ .
\]
For $\bm{y}\in\mathcal{E}(\bm{m})$, we obtain
\begin{equation}
\begin{aligned} & \mathbb{P}_{\bm{y}}^{\epsilon}\left[\bm{x}_{\epsilon}(2\epsilon^{-1}-T)\in\mathcal{A}\right]\\
 & =\mathbb{P}_{\bm{y}}^{\epsilon}\left[\bm{x}_{\epsilon}(2\epsilon^{-1}-T)\in\mathcal{A},\,\tau_{\partial\mathcal{U}^{(1)}(\bm{m})}>2\epsilon^{-1}-T\right]+R_{\epsilon}^{(2)}(\bm{y})\\
 & =\mathbb{P}_{\bm{y}}^{\epsilon,\,\mathcal{U}^{(1)}(\bm{m})}\left[\bm{x}_{\epsilon}^{\mathcal{U}^{(1)}(\bm{m})}(2\epsilon^{-1}-T)\in\mathcal{A},\,\tau_{\partial\mathcal{U}^{(1)}(\bm{m})}>2\epsilon^{-1}-T\right]+R_{\epsilon}^{(2)}(\bm{y})\\
 & =\mathbb{P}_{\bm{y}}^{\epsilon,\,\mathcal{U}^{(1)}(\bm{m})}\left[\bm{x}_{\epsilon}^{\mathcal{U}^{(1)}(\bm{m})}(2\epsilon^{-1}-T)\in\mathcal{A}\right]+R_{\epsilon}^{(2)}(\bm{y})+R_{\epsilon}^{(3)}(\bm{y})
\end{aligned}
\label{e: theta^0-2}
\end{equation}
where
\[
|\,R_{\epsilon}^{(2)}(\bm{y})\,|\le\mathbb{P}_{\bm{y}}^{\epsilon}\left[\tau_{\partial\mathcal{U}^{(1)}(\bm{m})}\le2\epsilon^{-1}-T\right]\ \ ,\ \ |\,R_{\epsilon}^{(3)}(\bm{y})\,|\le\mathbb{P}_{\bm{y}}^{\epsilon,\,\mathcal{U}^{(1)}(\bm{m})}\left[\tau_{\partial\mathcal{U}^{(1)}(\bm{m})}\le2\epsilon^{-1}-T\right]\,,
\]
and by \eqref{e: FW_U},
\[
\limsup_{\epsilon\to0}\sup_{\bm{y}\in\mathcal{E}(\bm{m})}|\,R_{\epsilon}^{(2)}(\bm{y})\,|=\limsup_{\epsilon\to0}\sup_{\bm{y}\in\mathcal{E}(\bm{m})}|\,R_{\epsilon}^{(3)}(\bm{y})\,|=0\,.
\]
By the Markov property and Proposition \ref{p: TV-mix}-(2), the last
probability in \eqref{e: theta^0-2} can be written as
\begin{equation}
\begin{aligned} & \mathbb{E}_{\bm{y}}^{\epsilon,\,\mathcal{U}^{(1)}(\bm{m})}\left[\mathbb{P}_{\bm{x}_{\epsilon}^{\mathcal{U}^{(1)}(\bm{m})}(\epsilon^{-1})}^{\epsilon,\,\mathcal{U}^{(1)}(\bm{m})}\left[\bm{x}_{\epsilon}^{\mathcal{U}^{(1)}(\bm{m})}(\epsilon^{-1}-T)\in\mathcal{A}\right]\right]\\
 & =\mathbb{P}_{\mu_{\epsilon}^{\mathcal{U}^{(1)}(\bm{m})}}^{\epsilon,\,\mathcal{U}^{(1)}(\bm{m})}\left[\bm{x}_{\epsilon}^{\mathcal{U}^{(1)}(\bm{m})}(\epsilon^{-1}-T)\in\mathcal{A}\right]+R_{\epsilon}^{(4)}(\bm{y})\,,
\end{aligned}
\label{e: theta^0-3}
\end{equation}
where
\[
\limsup_{\epsilon\to0}\sup_{\bm{y}\in\mathcal{E}(\bm{m})}|\,R_{\epsilon}^{(4)}(\bm{y})\,|=0\,.
\]
Since $\mu_{\epsilon}^{\mathcal{U}(\bm{m})}$ is the stationary measure
of $\bm{x}_{\epsilon}^{\mathcal{U}(\bm{m})}$, the last probability
in \eqref{e: theta^0-3} is equal to $\mu_{\epsilon}^{\mathcal{U}^{(1)}(\bm{m})}(\mathcal{A})$.

So far, we obtain
\[
\mathbb{P}_{\bm{x}}^{\epsilon}\left[\bm{x}_{\epsilon}(2\epsilon^{-1})\in\mathcal{A}\right]=\mu_{\epsilon}^{\mathcal{U}^{(1)}(\bm{m})}(\mathcal{A})+R_{\epsilon}(\mathcal{A})\,,
\]
where
\[
\limsup_{\epsilon\to0}\sup_{\mathcal{A}\subset\mathbb{R}^{d}}|\,R_{\epsilon}(\mathcal{A})\,|=0\,.
\]
Therefore, the right-hand side of \eqref{e: theta^0} vanishes
as $\epsilon\to0$. Consequently,
\[
\lim_{\epsilon\to0}D_{\epsilon}(2\epsilon^{-1},\,\bm{x})=\lim_{\epsilon\to0}d_{{\rm TV}}\left(\mu_{\epsilon}^{\mathcal{U}^{(1)}(\bm{m})},\,\mu_{\epsilon}\right)
\]
if the limit of the right-hand side exists. Indeed, by Lemma \ref{l: TV_density},
\[
\lim_{\epsilon\to0}d_{{\rm TV}}\left(\mu_{\epsilon}^{\mathcal{U}^{(1)}(\bm{m})},\,\mu_{\epsilon}\right)=\begin{cases}
1-\frac{\nu(\bm{m})}{\nu_{\star}} & \bm{m}\in\mathcal{M}_{\star}\,,\\
1 & \bm{m}\notin\mathcal{M}_{\star}\,,
\end{cases}
\]
which completes the proof.
\end{proof}
Now, we are ready to prove Theorem \ref{t: main}.
\begin{proof}[Proof of Theorem \ref{t: main}]
Fix $p\in\llbracket1,\,\mathfrak{q}\rrbracket$, $\bm{m}\in\mathcal{M}_{0}$,
and $\bm{x}\in\mathcal{D}(\bm{m})$. By triangle inequality,
\begin{align*}
 & \left|\,D_{\epsilon}(\theta_{\epsilon}^{(p)}\,t,\,\bm{x})-d_{{\rm TV}}\left(\sum_{\mathcal{M}\in\mathscr{V}^{(p)}}P_{t}^{(p),\,\bm{m}}(\mathcal{M})\,\mu_{\epsilon}^{\mathcal{E}(\mathcal{M})},\,\mu_{\epsilon}\right)\,\right|\\
 & \le d_{{\rm TV}}\left(\bm{x}_{\epsilon}(\theta_{\epsilon}^{(p)}\,t;\,\bm{x}),\,\sum_{\mathcal{M}\in\mathscr{V}^{(p)}}P_{t}^{(p),\,\bm{m}}(\mathcal{M})\,\mu_{\epsilon}^{\mathcal{E}(\mathcal{M})}\right)\,.
\end{align*}
Therefore, by Theorem \ref{t: main_TV} and Lemma \ref{l: 45}, we
obtain
\begin{equation}
\lim_{\epsilon\to0}D_{\epsilon}(\theta_{\epsilon}^{(p)}\,t,\,\bm{x})=d_{{\rm TV}}\left({\bf y}^{(p)}\left(t;\,\mathfrak{a}_{\bm{m}}^{(p-1)}\right),\,\nu^{(p)}\right)\,.\label{e: pf_t_main-1}
\end{equation}

For the second assertion, let $\theta_{\epsilon}^{(p-1)}\prec\varrho_{\epsilon}\prec\theta_{\epsilon}^{(p)}$.
By Lemma \ref{l: d_TV-1}, we have
\[
D_{\epsilon}(\theta_{\epsilon}^{(p-1)}t,\,\bm{x})\ge D_{\epsilon}(\varrho_{\epsilon},\,\bm{x})\ge D_{\epsilon}(\theta_{\epsilon}^{(p)}s,\,\bm{x})
\]
for all $t,\,s>0$ and sufficiently small $\epsilon>0$. First, let
$p\ge2$. Then,
\[
\lim_{t\to\infty}\,\lim_{\epsilon\to0}\,D_{\epsilon}(\theta_{\epsilon}^{(p-1)}t,\,\bm{x})\ge\lim_{\epsilon\to0}\,D_{\epsilon}(\varrho_{\epsilon},\,\bm{x})\ge\lim_{s\to0}\lim_{\epsilon\to0}\,D_{\epsilon}(\theta_{\epsilon}^{(p)}s,\,\bm{x})\,.
\]
By Lemma \ref{l: 48} and \eqref{e: pf_t_main-1}, we obtain
\[
\lim_{t\to\infty}\,\lim_{\epsilon\to0}\,D_{\epsilon}(\theta_{\epsilon}^{(p-1)}t,\,\bm{x})=d_{{\rm TV}}(\mathfrak{a}_{\bm{m}}^{(p-1)},\,\nu^{(p)})\,.
\]
On the other hand, by \eqref{e: pf_t_main-1}, we have
\[
\lim_{s\to0}\lim_{\epsilon\to0}\,D_{\epsilon}(\theta_{\epsilon}^{(p)}s,\,\bm{x})=d_{{\rm TV}}\left(\mathfrak{a}_{\bm{m}}^{(p-1)},\,\nu^{(p)}\right)
\]
so that
\[
\lim_{\epsilon\to0}\,D_{\epsilon}(\varrho_{\epsilon},\,\bm{x})=d_{{\rm TV}}\left(\mathfrak{a}_{\bm{m}}^{(p-1)},\,\nu^{(p)}\right)\,.
\]
For $p=1$, letting $t=2$, we have
\[
D_{\epsilon}(\varrho_{\epsilon},\,\bm{x})\le D_{\epsilon}(2\epsilon^{-1},\,\bm{x})\,,
\]
so that by Lemma \ref{l: theta^0},
\[
\lim_{\epsilon\to0}\,D_{\epsilon}(\varrho_{\epsilon},\,\bm{x})\le\begin{cases}
1-\frac{\nu(\bm{m})}{\nu_{\star}} & \bm{m}\in\mathcal{M}_{\star}\,,\\
1 & \bm{m}\notin\mathcal{M}_{\star}\,.
\end{cases}
\]
Also, since $\mathfrak{a}_{\bm{m}}^{(0)}=\delta_{\bm{m}}$, we obtain
\[
d_{{\rm TV}}\left(\mathfrak{a}_{\bm{m}}^{(0)},\,\nu^{(1)}\right)=\begin{cases}
1-\frac{\nu(\bm{m})}{\nu_{\star}} & \bm{m}\in\mathcal{M}_{\star}\,,\\
1 & \bm{m}\notin\mathcal{M}_{\star}\,.
\end{cases}
\]
By the same argument in the proof of the first assertion, we obtain
\[
\lim_{\epsilon\to0}\,D_{\epsilon}(\varrho_{\epsilon},\,\bm{x})\ge\lim_{s\to0}\lim_{\epsilon\to0}\,D_{\epsilon}(\theta_{\epsilon}^{(1)}s,\,\bm{x})=d_{{\rm TV}}\left(\mathfrak{a}_{\bm{m}}^{(0)},\,\nu^{(1)}\right)\,,
\]
which completes the proof for the second assertion.

For the last assertion, let $\varrho_{\epsilon}\succ\theta_{\epsilon}^{(\mathfrak{q})}$
and fix $\bm{x}\in\mathbb{R}^{d}$. Let $\mathcal{A}\subset\mathbb{R}^{d}$
be a measurable subset. By Lemma \ref{l: main2}, for all $t>0$,
\begin{equation*}
\begin{aligned}
& \mathbb{P}_{\bm{x}}^{\epsilon}\left[\bm{x}_{\epsilon}(\theta_{\epsilon}^{(\mathfrak{q})}\,t)\in\mathcal{A}\right]\\
& =\sum_{\mathcal{M}\in\mathscr{V}^{(\mathfrak{q})}}\sum_{\bm{m}'\in\mathcal{M}_{0}}\mathbb{P}_{\bm{x}}^{\epsilon}\left[\tau_{\mathcal{E}(\mathcal{M}_{0})}=\tau_{\mathcal{E}(\bm{m}')}\right]P_{t}^{(\mathfrak{q}),\,\bm{m}'}(\mathcal{M})\,\mu_{\epsilon}^{\mathcal{E}(\mathcal{M})}(\mathcal{A})+R_{\epsilon}^{(1)}(t,\,\mathcal{A})
\end{aligned}
\end{equation*}
where
\[
\limsup_{\epsilon\to0}\sup_{\mathcal{A}\subset\mathbb{R}^{d}}|\,R_{\epsilon}^{(1)}(t,\,\mathcal{A})\,|=0\,.
\]
Fix $\eta>0$. Since ${\bf y}^{(\mathfrak{q})}$ is ergodic, there
exists $T_{\eta}>0$ such that
\[
d_{{\rm TV}}\left({\bf y}^{(\mathfrak{q})}(T_{\eta};\,\mathfrak{a}_{\bm{m}}^{(\mathfrak{q}-1)}),\,\nu^{(\mathfrak{q})}\right)\le\frac{1}{|\mathcal{M}_{0}|}\eta\ .
\]
Then,
\begin{align*}
 & \mathbb{P}_{\bm{x}}^{\epsilon}\left[\bm{x}_{\epsilon}(\theta_{\epsilon}^{(\mathfrak{q})}\,T_{\eta})\in\mathcal{A}\right]\\
 & =\sum_{\mathcal{M}\in\mathscr{R}_{1}^{(\mathfrak{q})}}\sum_{\bm{m}'\in\mathcal{M}_{0}}\mathbb{P}_{\bm{x}}^{\epsilon}\left[\tau_{\mathcal{E}(\mathcal{M}_{0})}=\tau_{\mathcal{E}(\bm{m}')}\right]\,\nu^{(\mathfrak{q})}(\mathcal{M})\,\mu_{\epsilon}^{\mathcal{E}(\mathcal{M})}(\mathcal{A})+R_{\epsilon}^{(2)}(\mathcal{A},\,\eta)\\
 & =\sum_{\mathcal{M}\in\mathscr{R}_{1}^{(\mathfrak{q})}}\frac{\nu(\mathcal{M})}{\nu_{\star}}\,\mu_{\epsilon}^{\mathcal{E}(\mathcal{M})}(\mathcal{A})+R_{\epsilon}^{(2)}(\mathcal{A},\,\eta)\,,
\end{align*}
where
\[
|\,R_{\epsilon}^{(2)}(\mathcal{A},\,\eta)\,|\le|\,R_{\epsilon}^{(1)}(T_{\eta},\,\mathcal{A})\,|+\eta\,.
\]
Hence, we obtain
\[
\limsup_{\epsilon\to0}d_{{\rm TV}}\left(\bm{x}_{\epsilon}(\theta_{\epsilon}^{(\mathfrak{q})}T_{\eta};\,\bm{x}),\,\sum_{\mathcal{M}\in\mathscr{R}_{1}^{(\mathfrak{q})}}\frac{\nu(\mathcal{M})}{\nu_{\star}}\,\mu_{\epsilon}^{\mathcal{E}(\mathcal{M})}\right)\le\eta\,.
\]
Therefore, by triangle inequality,
\[
\limsup_{\epsilon\to0}D_{\epsilon}(\theta_{\epsilon}^{(\mathfrak{q})}T_{\eta},\,\bm{x})\le\eta+\limsup_{\epsilon\to0}d_{{\rm TV}}\left(\sum_{\mathcal{M}'\in\mathscr{R}_{1}^{(\mathfrak{q})}}\frac{\nu(\mathcal{M}')}{\nu_{\star}}\,\mu_{\epsilon}^{\mathcal{E}(\mathcal{M}')},\,\mu_{\epsilon}\right)\,.
\]
Since $\sum_{\mathcal{M}'\in\mathscr{R}_{1}^{(\mathfrak{q})}}\frac{\nu(\mathcal{M}')}{\nu_{\star}}=1$
and $\lim_{\epsilon\to0}\mu_{\epsilon}(\mathcal{E}(\mathcal{M}'))=\frac{\nu(\mathcal{M}')}{\nu_{\star}}$
for $\mathcal{M}'\in\mathscr{R}_{1}^{(\mathfrak{q})}=\mathscr{R}_{\star}^{(\mathfrak{q})}$,
by Lemma \ref{e: TV_density}, we obtain
\[
\limsup_{\epsilon\to0}d_{{\rm TV}}(\sum_{\mathcal{M}'\in\mathscr{R}_{1}^{(\mathfrak{q})}}\frac{\nu(\mathcal{M}')}{\nu_{\star}}\,\mu_{\epsilon}^{\mathcal{E}(\mathcal{M}')},\,\mu_{\epsilon})=0
\]
so that by Lemma \ref{l: d_TV-1},
\[
0\le\limsup_{\epsilon\to0}\,D_{\epsilon}(\varrho_{\epsilon},\,\bm{x})\le\limsup_{\epsilon\to0}\,D_{\epsilon}(\theta_{\epsilon}^{(\mathfrak{q})}T_{\eta},\,\bm{x})\le\eta\,.
\]
Since $\eta>0$ is arbitrary, we conclude
\[
\lim_{\epsilon\to0}D_{\epsilon}(\varrho_{\epsilon},\,\bm{x})=0\,.
\]
\end{proof}

\section{\label{sec: pf_mixing_time}Mixing time}

In this section, we prove Theorem \ref{t: mixing_time}. Fix $H\ge H_{0}$
and $\delta>0$. Also, for $\mathcal{M}\in\mathscr{V}^{(\mathfrak{q})}$
and $t>0$, denote by
\[
{\color{blue}d(t,\,\mathcal{M})}:=d_{{\rm TV}}({\bf y}^{(\mathfrak{q})}(t;\,\mathcal{M}),\,\nu^{(\mathfrak{q})})\,.
\]

\subsection{Lower bound}
\begin{proposition}
\label{p: mix_low}We have
\[
\liminf_{\epsilon\to0}\frac{T_{\epsilon}^{{\rm mix}}(\delta;\,H)}{\theta_{\epsilon}^{(\mathfrak{q})}}\ge\mathfrak{T}^{{\rm mix}}(\delta)\,.
\]
\end{proposition}

\begin{proof}
Suppose not. Then, there exist $a>0$ and decreasing sequence $(\epsilon_{n})_{n\ge1}$
of positive number such that $\lim_{n\to\infty}\epsilon_{n}=0$ and
\[
\frac{T_{\epsilon_{n}}^{{\rm mix}}(\delta;\,H)}{\theta_{\epsilon_{n}}^{(\mathfrak{q})}}<\mathfrak{T}^{{\rm mix}}(\delta)-a\,.
\]
Since
\[
\theta_{\epsilon_{n}}^{(\mathfrak{q})}(\mathfrak{T}^{{\rm mix}}(\delta)-a)>T_{\epsilon_{n}}^{{\rm mix}}(\delta;\,H)\,,
\]
by Lemma \ref{l: d_TV-1}, we have
\begin{equation}
\sup_{\bm{x}\in\mathcal{K}(H)}D_{\epsilon_{n}}(\theta_{\epsilon_{n}}^{(\mathfrak{q})}(\mathfrak{T}^{{\rm mix}}(\delta)-a),\,\bm{x})\le\sup_{\bm{x}\in\mathcal{K}(H)}D_{\epsilon_{n}}(T_{\epsilon_{n}}^{{\rm mix}}(\delta;\,H),\,\bm{x})\le\delta\label{e: p_mix_low}
\end{equation}
for all $n\in\mathbb{N}$. By definition of $\mathfrak{T}^{{\rm mix}}(\delta)$,
there exists $\mathcal{M}\in\mathscr{V}^{(\mathfrak{q})}$ such that
\[
d(\mathfrak{T}^{{\rm mix}}(\delta)-a,\,\mathcal{M})>\delta\,.
\]
Fix $\bm{m}\in\mathcal{M}$ and $\bm{x}\in\mathcal{D}(\bm{m})\cap\mathcal{K}(H)$.
By Theorem \ref{t: main}-(1),
\[
\lim_{n\to\infty}D_{\epsilon_{n}}(\theta_{\epsilon_{n}}^{(\mathfrak{q})}(\mathfrak{T}^{{\rm mix}}(\delta)-a),\,\bm{x})=d(\mathfrak{T}^{{\rm mix}}(\delta)-a,\,\mathcal{M})>\delta\,,
\]
which contradicts to \eqref{e: p_mix_low}.
\end{proof}

\subsection{Upper bound}

Since the mixing time is defined in terms of the worst-case scenario,
deriving the lower bound is straightforward. However, for the upper
bound, we need the following strict inequality uniform over $\bm{x}\in\mathcal{K}(H)$.
\begin{lemma}
\label{l: mix_up}For all $a>0$,
\[
\limsup_{\epsilon\to0}\sup_{\bm{x}\in\mathcal{K}(H)}D_{\epsilon}(\theta_{\epsilon}^{(\mathfrak{q})}\mathfrak{T}^{{\rm mix}}(\delta+a),\,\bm{x})<\delta\,.
\]
\end{lemma}

\begin{proof}
For $t>0$, $\bm{m}\in\mathcal{M}_{0}$, and $\mathcal{M},\,\mathcal{M}'\in\mathscr{V}^{(\mathfrak{q})}$,
let us write
\begin{align*}
P_{t}(\mathcal{M},\,\mathcal{M}') & :=\mathcal{Q}_{\mathcal{M}}^{(\mathfrak{q})}\left[\,{\bf y}^{(\mathfrak{q})}(t)=\mathcal{M}'\,\right]\,,\\
P_{t}(\bm{m},\,\mathcal{M}') & :=\sum_{\mathcal{M}\in\mathscr{V}^{(\mathfrak{q})}}\mathfrak{a}_{\bm{m}}^{(\mathfrak{q}-1)}(\mathcal{M})P_{t}(\mathcal{M},\,\mathcal{M}')\,.
\end{align*}
By Lemma \ref{l: main2}, for all $\bm{x}\in\mathcal{K}(H)$ and $\mathcal{A}\subset\mathbb{R}^{d}$,
we have
\begin{align*}
 & \mathbb{P}_{\bm{x}}^{\epsilon}\left[\bm{x}_{\epsilon}(\theta_{\epsilon}^{(\mathfrak{q})}\,t)\in\mathcal{A}\right]\\
 & =\sum_{\mathcal{M}'\in\mathscr{V}^{(\mathfrak{q})}}\sum_{\bm{m}\in\mathcal{M}_{0}}\mathbb{P}_{\bm{x}}^{\epsilon}\left[\tau_{\mathcal{E}(\mathcal{M}_{0})}=\tau_{\mathcal{E}(\bm{m})}\right]P_{t}(\bm{m},\,\mathcal{M}')\,\mu_{\epsilon}^{\mathcal{E}(\mathcal{M}')}(\mathcal{A})+R_{\epsilon}^{(1)}(\bm{x},\,\mathcal{A})\,,
\end{align*}
where 
\[
\limsup_{\epsilon\to0}\sup_{\bm{x}\in\mathcal{K}(H)}\sup_{\mathcal{A}\subset\mathbb{R}^{d}}\,|\,R_{\epsilon}^{(1)}(\bm{x},\,\mathcal{A})\,|=0\,.
\]
Now, we have
\begin{align*}
 & \sum_{\mathcal{M}'\in\mathscr{V}^{(\mathfrak{q})}}\sum_{\bm{m}\in\mathcal{M}_{0}}\mathbb{P}_{\bm{x}}^{\epsilon}\left[\tau_{\mathcal{E}(\mathcal{M}_{0})}=\tau_{\mathcal{E}(\bm{m})}\right]P_{t}(\bm{m},\,\mathcal{M}')\,\mu_{\epsilon}^{\mathcal{E}(\mathcal{M}')}(\mathcal{A})-\mu_{\epsilon}(\mathcal{A})\\
 & =\sum_{\bm{m}\in\mathcal{M}_{0}}\mathbb{P}_{\bm{x}}^{\epsilon}\left[\tau_{\mathcal{E}(\mathcal{M}_{0})}=\tau_{\mathcal{E}(\bm{m})}\right]\left\{ \sum_{\mathcal{M}'\in\mathscr{V}^{(\mathfrak{q})}}\sum_{\mathcal{M}\in\mathscr{V}^{(\mathfrak{q})}}\mathfrak{a}_{\bm{m}}^{(\mathfrak{q}-1)}(\mathcal{M})P_{t}(\mathcal{M},\,\mathcal{M}')\,\mu_{\epsilon}^{\mathcal{E}(\mathcal{M}')}(\mathcal{A})-\mu_{\epsilon}(\mathcal{A})\right\} \\
 & =\sum_{\bm{m}\in\mathcal{M}_{0}}\mathbb{P}_{\bm{x}}^{\epsilon}\left[\tau_{\mathcal{E}(\mathcal{M}_{0})}=\tau_{\mathcal{E}(\bm{m})}\right]\sum_{\mathcal{M}\in\mathscr{V}^{(\mathfrak{q})}}\mathfrak{a}_{\bm{m}}^{(\mathfrak{q}-1)}(\mathcal{M})\left\{ \sum_{\mathcal{M}'\in\mathscr{V}^{(\mathfrak{q})}}P_{t}(\mathcal{M},\,\mathcal{M}')\,\mu_{\epsilon}^{\mathcal{E}(\mathcal{M}')}(\mathcal{A})-\mu_{\epsilon}(\mathcal{A})\right\} \,.
\end{align*}
Therefore, since $\sum_{\bm{m}\in\mathcal{M}_{0}}\mathbb{P}_{\bm{x}}^{\epsilon}\left[\tau_{\mathcal{E}(\mathcal{M}_{0})}=\tau_{\mathcal{E}(\bm{m})}\right]=\sum_{\mathcal{M}\in\mathscr{V}^{(\mathfrak{q})}}\mathfrak{a}_{\bm{m}}^{(\mathfrak{q}-1)}(\mathcal{M})=1$,
\begin{equation}
\begin{aligned} & D_{\epsilon}(\theta_{\epsilon}^{(\mathfrak{q})}\,t,\,\bm{x})\\
 & =\sup_{\mathcal{A}\subset\mathbb{R}^{d}}\left\{ \mathbb{P}_{\bm{x}}^{\epsilon}\left[\bm{x}_{\epsilon}(\theta_{\epsilon}^{(\mathfrak{q})}\,t)\in\mathcal{A}\right]-\mu_{\epsilon}(\mathcal{A})\right\} \\
 & \le\sup_{\mathcal{M}\in\mathscr{V}^{(\mathfrak{q})}}\sup_{\mathcal{A}\subset\mathbb{R}^{d}}\left\{ \sum_{\mathcal{M}'\in\mathscr{V}^{(\mathfrak{q})}}P_{t}(\mathcal{M},\,\mathcal{M}')\,\mu_{\epsilon}^{\mathcal{E}(\mathcal{M}')}(\mathcal{A})-\mu_{\epsilon}(\mathcal{A})\right\} +\sup_{\mathcal{A}\subset\mathbb{R}^{d}}\,|\,R_{\epsilon}^{(1)}(\bm{x},\,\mathcal{A})\,|\\
 & =\sup_{\mathcal{M}\in\mathscr{V}^{(\mathfrak{q})}}d_{{\rm TV}}\left(\sum_{\mathcal{M}'\in\mathscr{V}^{(\mathfrak{q})}}P_{t}(\mathcal{M},\,\mathcal{M}')\,\mu_{\epsilon}^{\mathcal{E}(\mathcal{M}')},\,\mu_{\epsilon}\right)+\sup_{\mathcal{A}\subset\mathbb{R}^{d}}\,|\,R_{\epsilon}^{(1)}(\bm{x},\,\mathcal{A})\,|\,.
\end{aligned}
\label{e: l_mix_up-1}
\end{equation}
For $\mathcal{M}\in\mathscr{V}^{(\mathfrak{q})}$, by Lemma \ref{l: TV_density},
\begin{equation}
\begin{aligned} & \lim_{\epsilon\to0}d_{{\rm TV}}\left(\sum_{\mathcal{M}'\in\mathscr{V}^{(\mathfrak{q})}}P_{t}(\mathcal{M},\,\mathcal{M}')\,\mu_{\epsilon}^{\mathcal{E}(\mathcal{M}')},\,\mu_{\epsilon}\right)\\
 & =\frac{1}{2}\sum_{\mathcal{M}'\in\mathscr{V}^{(\mathfrak{q})}\setminus\mathscr{V}_{\star}^{(\mathfrak{q})}}P_{t}(\mathcal{M},\,\mathcal{M}')+\sum_{\mathcal{M}'\in\mathscr{R}_{1}^{(\mathfrak{q})}}\left|\,P_{t}(\mathcal{M},\,\mathcal{M}')-\frac{\nu(\mathcal{M}')}{\nu_{\star}}\,\right|\\
 & =d(t;\,\mathcal{M})\,.
\end{aligned}
\label{e: l_mix_up-2}
\end{equation}
Since $\mathscr{V}^{(\mathfrak{q})}$ is finite set, by \eqref{e: l_mix_up-1}
and \eqref{e: l_mix_up-2},
\[
\limsup_{\epsilon\to0}\sup_{\bm{x}\in\mathcal{K}(H)}D_{\epsilon}(\theta_{\epsilon}^{(\mathfrak{q})}\,t,\,\bm{x})\le\max_{\mathcal{M}\in\mathscr{V}^{(\mathfrak{q})}}d(t,\,\mathcal{M})\,.
\]
Therefore, by Lemma \ref{l: d_TV-2},
\begin{align*}
\limsup_{\epsilon\to0}\sup_{\bm{x}\in\mathcal{K}(H)}D_{\epsilon}(\theta_{\epsilon}^{(\mathfrak{q})}(\mathfrak{T}^{{\rm mix}}(\delta)+a)),\,\bm{x}) & \le\max_{\mathcal{M}\in\mathscr{V}^{(\mathfrak{q})}}d(\mathfrak{T}^{{\rm mix}}(\delta)+a,\,\mathcal{M})\\
 & <\max_{\mathcal{M}\in\mathscr{V}^{(\mathfrak{q})}}d(\mathfrak{T}^{{\rm mix}}(\delta),\,\mathcal{M})\\
 & \le\delta\,.
\end{align*}
\end{proof}
\begin{proposition}
\label{p: mix_up}For all $\eta>0$, we have
\[
\limsup_{\epsilon\to0}\frac{T_{\epsilon}^{{\rm mix}}(\delta;\,H)}{\theta_{\epsilon}^{(\mathfrak{q})}}\le\mathfrak{T}^{{\rm mix}}(\delta)\,.
\]
\end{proposition}

\begin{proof}
Suppose not. Then, there exist $a>0$ and decreasing sequence $(\epsilon_{n})_{n\ge1}$
of positive number such that $\lim_{n\to\infty}\epsilon_{n}=0$ and
\[
\frac{T_{\epsilon_{n}}^{{\rm mix}}(\delta;\,H)}{\theta_{\epsilon_{n}}^{(\mathfrak{q})}}\ge\mathfrak{T}^{{\rm mix}}(\delta)+a\,.
\]
Since
\[
\theta_{\epsilon_{n}}^{(\mathfrak{q})}(\mathfrak{T}^{{\rm mix}}(\delta)+a)\le T_{\epsilon_{n}}^{{\rm mix}}(\delta;\,H)
\]
we have
\[
\sup_{\bm{x}\in\mathcal{K}(H)}D_{\epsilon_{n}}(\theta_{\epsilon_{n}}^{(\mathfrak{q})}(\mathfrak{T}^{{\rm mix}}(\delta)+a),\,\bm{x})\ge\delta
\]
for all $n\in\mathbb{N}$. Therefore, we obtain
\[
\limsup_{\epsilon\to0}\sup_{\bm{x}\in\mathcal{K}(H)}D_{\epsilon}(\theta_{\epsilon}^{(\mathfrak{q})}(\mathfrak{T}^{{\rm mix}}(\delta)+a),\,\bm{x})\ge\limsup_{n\to\infty}\sup_{\bm{x}\in\mathcal{K}(H)}D_{\epsilon_{n}}(\theta_{\epsilon_{n}}^{(\mathfrak{q})}(\mathfrak{T}^{{\rm mix}}(\delta)+a),\,\bm{x})\ge\delta\,,
\]
which contradicts to Lemma \ref{l: mix_up}.
\end{proof}

\subsection*{Proof of Theorem \ref{t: mixing_time}}

Theorem \ref{t: mixing_time} is a direct consequence of Propositions
\ref{p: mix_low} and \ref{p: mix_up}.

\section{\label{sec: pf_p_TV-mix}Proof of Proposition \ref{p: TV-mix}}

We first prove the first assertion of Proposition \ref{p: TV-mix}.
\begin{proof}[Proof of Proposition \ref{p: TV-mix}-(1)]
Fix $p\in\llbracket1,\,\mathfrak{q}-1\rrbracket$, $\mathcal{M}\in\mathscr{V}^{(p+1)}$
and assume $\mathfrak{C}_{{\rm TV}}^{(p)}$ holds. Let $i\in\llbracket1,\,\mathfrak{n}_{p}\rrbracket$
be such that $\mathcal{M}=\mathcal{M}_{i}^{(p+1)}=\bigcup_{\mathcal{M}''\in\mathscr{R}_{i}^{(p)}}\mathcal{M}''$
and fix $\mathcal{M}'\in\mathscr{R}_{i}^{(p)}$. Let
\[
d^{(p)}<A<R^{(p+1)}-r_{0}
\]
and $\rho_{\epsilon}=e^{A/\epsilon}$. Then, by \eqref{e: FW_U},
$\mathfrak{M}^{(p+1)}$-(1) holds with $\rho_{\epsilon}\prec\theta_{\epsilon}^{(p+1)}$.

By triangle inequality, for $\bm{y}\in\mathcal{E}(\mathcal{M}')$,
\begin{equation}
\begin{aligned}d_{{\rm TV}}\left(\bm{x}_{\epsilon}^{\mathcal{U}^{(p+1)}(\mathcal{M})}(\rho_{\epsilon};\,\bm{y}),\,\mu_{\epsilon}^{\mathcal{U}^{(p+1)}(\mathcal{M})}\right) & \le d_{{\rm TV}}\left(\bm{x}_{\epsilon}^{\mathcal{U}^{(p+1)}(\mathcal{M})}(\rho_{\epsilon};\,\bm{y}),\,\bm{x}_{\epsilon}(\rho_{\epsilon};\,\bm{y})\right)\\
 & +d_{{\rm TV}}\left(\bm{x}_{\epsilon}(\rho_{\epsilon};\,\bm{y}),\,\mu_{\epsilon}^{\mathcal{U}^{(p+1)}(\mathcal{M})}\right)
\end{aligned}
\label{e: p_TV-mix-1}
\end{equation}
First, we claim that
\begin{equation}
\lim_{\epsilon\to0}\sup_{\bm{y}\in\mathcal{E}(\mathcal{M}')}d_{{\rm TV}}\left(\bm{x}_{\epsilon}^{\mathcal{U}^{(p+1)}(\mathcal{M})}(\rho_{\epsilon};\,\bm{y}),\,\bm{x}_{\epsilon}(\rho_{\epsilon};\,\bm{y})\right)=0\,.\label{e: p_TV-mix-2}
\end{equation}
Let $\mathcal{A}\subset\mathbb{R}^{d}$. Then, we have
\begin{equation}
\begin{aligned} & \mathbb{P}_{\bm{y}}^{\epsilon,\,\mathcal{U}^{(p+1)}(\mathcal{M})}\left[\bm{x}_{\epsilon}^{\mathcal{U}^{(p+1)}(\mathcal{M})}(\rho_{\epsilon})\in\mathcal{A}\right]\\
 & =\mathbb{P}_{\bm{y}}^{\epsilon,\,\mathcal{U}^{(p+1)}(\mathcal{M})}\left[\bm{x}_{\epsilon}^{\mathcal{U}^{(p+1)}(\mathcal{M})}(\rho_{\epsilon})\in\mathcal{A}\cap\mathcal{U}^{(p+1)}(\mathcal{M}),\,\tau_{\partial\mathcal{U}^{(p+1)}(\mathcal{M})}>\rho_{\epsilon}\right]+R_{\epsilon}^{(1)}(\bm{y},\,\mathcal{A})
\end{aligned}
\label{e: p_TV-mix-3}
\end{equation}
where
\[
|\,R_{\epsilon}^{(1)}(\bm{y},\,\mathcal{A})\,|\le\mathbb{P}_{\bm{y}}^{\epsilon,\,\mathcal{U}^{(p+1)}(\mathcal{M})}\left[\tau_{\partial\mathcal{U}^{(p+1)}(\mathcal{M})}\le\rho_{\epsilon}\right]=\mathbb{P}_{\bm{y}}^{\epsilon}\left[\tau_{\partial\mathcal{U}^{(p+1)}(\mathcal{M})}\le\rho_{\epsilon}\right]\,.
\]
By the coupling argument, the last probability in \eqref{e: p_TV-mix-3}
can be written as
\begin{equation}
\begin{aligned}
& \mathbb{P}_{\bm{y}}^{\epsilon}\left[\bm{x}_{\epsilon}(\rho_{\epsilon})\in\mathcal{A}\cap\mathcal{U}^{(p+1)}(\mathcal{M}),\,\tau_{\partial\mathcal{U}^{(p+1)}(\mathcal{M})}>\rho_{\epsilon}\right]\\
& =\mathbb{P}_{\bm{y}}^{\epsilon}\left[\bm{x}_{\epsilon}(\rho_{\epsilon})\in\mathcal{A}\cap\mathcal{U}^{(p+1)}(\mathcal{M})\right]+R_{\epsilon}^{(2)}(\bm{y},\,\mathcal{A})
\end{aligned}
\label{e: p_TV-mix-4}
\end{equation}
where
\[
|\,R_{\epsilon}^{(2)}(\bm{y},\,\mathcal{A})\,|\le\mathbb{P}_{\bm{y}}^{\epsilon}\left[\tau_{\mathcal{U}^{(p+1)}(\mathcal{M})}\le\rho_{\epsilon}\right]\,.
\]
Now, by \eqref{e: p_TV-mix-3}, \eqref{e: p_TV-mix-4}, and the following
inequality
\begin{align*}
\left|\,\mathbb{P}_{\bm{y}}^{\epsilon}\left[\bm{x}_{\epsilon}(\rho_{\epsilon})\in\mathcal{A}\cap\mathcal{U}^{(p+1)}(\mathcal{M})\right]-\mathbb{P}_{\bm{y}}^{\epsilon}\left[\bm{x}_{\epsilon}(\rho_{\epsilon})\in\mathcal{A}\right]\,\right| & \le\mathbb{P}_{\bm{y}}^{\epsilon}\left[\bm{x}_{\epsilon}(\rho_{\epsilon})\notin\mathcal{U}^{(p+1)}(\mathcal{M})\right]\\
 & \le\mathbb{P}_{\bm{y}}^{\epsilon}\left[\tau_{\partial\mathcal{U}^{(p+1)}(\mathcal{M})}\le\rho_{\epsilon}\right]\,,
\end{align*}
we obtain
\begin{align*}
 & d_{{\rm TV}}(\bm{x}_{\epsilon}^{\mathcal{U}^{(p+1)}(\mathcal{M})}(\rho_{\epsilon};\,\bm{y}),\,\bm{x}_{\epsilon}(\rho_{\epsilon};\,\bm{y}))\\
 & =\left|\,\mathbb{P}_{\bm{y}}^{\epsilon,\,\mathcal{U}^{(p+1)}(\mathcal{M})}\left[\bm{x}_{\epsilon}^{\mathcal{U}^{(p+1)}(\mathcal{M})}(\rho_{\epsilon})\in\mathcal{A}\right]-\mathbb{P}_{\bm{y}}^{\epsilon}\left[\bm{x}_{\epsilon}(\rho_{\epsilon})\in\mathcal{A}\right]\,\right|\\
 & \le3\mathbb{P}_{\bm{y}}^{\epsilon}\left[\tau_{\partial\mathcal{U}^{(p+1)}(\mathcal{M})}\le\rho_{\epsilon}\right]\,.
\end{align*}
Hence, by \eqref{e: FW_U}, we obtain \eqref{e: p_TV-mix-2}.

We turn to the second term of \eqref{e: p_TV-mix-1}. Since $\rho_{\epsilon}\succ\theta_{\epsilon}^{(p)}$,
by Lemma \ref{l: d_TV-1}, the second term of the right-hand side
of \eqref{e: p_TV-mix-1} is bounded by
\[
d_{{\rm TV}}\left(\bm{x}_{\epsilon}(\theta_{\epsilon}^{(p)}t;\,\bm{y}),\,\mu_{\epsilon}^{\mathcal{U}^{(p+1)}(\mathcal{M})}\right)
\]
for all $t>0$ and sufficiently small $\epsilon>0$. By triangle inequality,
this is bounded by
\begin{align*}
 & d_{{\rm TV}}\left(\bm{x}_{\epsilon}(\theta_{\epsilon}^{(p)}t;\,\bm{y}),\,\sum_{\mathcal{M}''\in\mathscr{R}_{i}^{(p)}}P_{t}(\mathcal{M}',\,\mathcal{M}'')\,\mu_{\epsilon}^{\mathcal{E}(\mathcal{M}'')}\right)\\
 & \ \ +d_{{\rm TV}}\left(\sum_{\mathcal{M}''\in\mathscr{R}_{i}^{(p)}}P_{t}(\mathcal{M}',\,\mathcal{M}'')\,\mu_{\epsilon}^{\mathcal{E}(\mathcal{M}'')},\,\mu_{\epsilon}^{\mathcal{U}^{(p+1)}(\mathcal{M})}\right)\,,
\end{align*}
where $P_{t}(\mathcal{M}',\,\mathcal{M}'')=\mathcal{Q}_{\mathcal{M}'}^{(p)}\left[\,{\bf y}^{(\mathfrak{q})}(t)=\mathcal{M}''\,\right]$.
Since $\mathcal{M}'\in\mathscr{R}_{i}^{(p)}$ and $\mathscr{R}_{i}^{(p)}$
is an irreducible class of ${\bf y}^{(p)}$, the first term converges
to zero uniformly in $\bm{y}\in\mathcal{E}(\mathcal{M}')$ by $\mathfrak{C}_{{\rm TV}}^{(p)}$.
On the other hand, by Lemma \ref{l: TV_density}, 
\[
\lim_{\epsilon\to0}d_{{\rm TV}}\left(\sum_{\mathcal{M}''\in\mathscr{R}_{i}^{(p)}}\,P_{t}(\mathcal{M}',\,\mathcal{M}'')\,\mu_{\epsilon}^{\mathcal{E}(\mathcal{M}'')},\,\mu_{\epsilon}^{\mathcal{U}^{(p+1)}(\mathcal{M})}\right)=\frac{1}{2}\sum_{\mathcal{M}''\in\mathscr{R}_{i}^{(p)}}\left|\,P_{t}(\mathcal{M}',\,\mathcal{M}'')-\frac{\nu(\mathcal{M}'')}{\nu(\mathcal{M})}\,\right|\,,
\]
so that by the ergodicity of ${\bf y}^{(p)}$ in $\mathscr{R}_{i}^{(p)}$,
\[
\lim_{t\to\infty}\lim_{\epsilon\to0}d_{{\rm TV}}\left(\sum_{\mathcal{M}''\in\mathscr{R}_{i}^{(p)}}\,P_{t}(\mathcal{M}',\,\mathcal{M}'')\,\mu_{\epsilon}^{\mathcal{E}(\mathcal{M}'')},\,\mu_{\epsilon}^{\mathcal{U}^{(p+1)}(\mathcal{M})}\right)=0\,.
\]
Therefore, we have
\[
\lim_{t\to\infty}\lim_{\epsilon\to0}\sup_{\bm{y}\in\mathcal{E}(\mathcal{M}')}d_{{\rm TV}}\left(\bm{x}_{\epsilon}(\theta_{\epsilon}^{(p)}t;\,\bm{y}),\,\mu_{\epsilon}^{\mathcal{U}^{(p+1)}(\mathcal{M})}\right)=0\,.
\]
By Lemma \ref{l: d_TV-1} and the above displayed equation, we obtain
\begin{equation}
\limsup_{\epsilon\to0}\sup_{\bm{y}\in\mathcal{E}(\mathcal{M}')}d_{{\rm TV}}\left(\bm{x}_{\epsilon}(\rho_{\epsilon};\,\bm{y}),\,\mu_{\epsilon}^{\mathcal{U}^{(p+1)}(\mathcal{M})}\right)=0\,.\label{e: p_TV-mix-5}
\end{equation}

By \eqref{e: p_TV-mix-1}, \eqref{e: p_TV-mix-2}, and \eqref{e: p_TV-mix-5},
\[
\lim_{\epsilon\to0}\sup_{\bm{y}\in\mathcal{E}(\mathcal{M}')}d_{{\rm TV}}\left(\bm{x}_{\epsilon}^{\mathcal{U}^{(p+1)}(\mathcal{M})}(\rho_{\epsilon};\,\bm{y}),\,\mu_{\epsilon}^{\mathcal{U}^{(p+1)}(\mathcal{M})}\right)=0\ \ ;\ \ \mathcal{M}'\in\mathscr{R}_{i}^{(p)}\,.
\]
Since $\mathscr{R}_{i}^{(p)}$ is finite and $\mathcal{M}=\bigcup_{\mathcal{M}''\in\mathscr{R}_{i}^{(p)}}\mathcal{M}''$,
we obtain
\begin{align*}
 & \lim_{\epsilon\to0}\sup_{\bm{y}\in\mathcal{E}(\mathcal{M})}d_{{\rm TV}}\left(\bm{x}_{\epsilon}^{\mathcal{U}^{(p+1)}(\mathcal{M})}(\rho_{\epsilon};\,\bm{y}),\,\mu_{\epsilon}^{\mathcal{U}^{(p+1)}(\mathcal{M})}\right)\\
 & =\lim_{\epsilon\to0}\sup_{\mathcal{M}''\in\mathscr{R}_{i}^{(p)}}\sup_{\bm{y}\in\mathcal{E}(\mathcal{M}'')}d_{{\rm TV}}\left(\bm{x}_{\epsilon}^{\mathcal{U}^{(p+1)}(\mathcal{M})}(\rho_{\epsilon};\,\bm{y}),\,\mu_{\epsilon}^{\mathcal{U}^{(p+1)}(\mathcal{M})}\right)=0\,,
\end{align*}
which implies $\mathfrak{M}^{(p+1)}$-(2)
\end{proof}
Next, we prove the initial condition $\mathfrak{M}^{(1)}$.
\begin{proof}[Proof of Proposition \ref{p: TV-mix}-(2)]
Let $\rho_{\epsilon}=\epsilon^{-1}$. By \eqref{e: FW_U}, $\mathfrak{M}^{(1)}$-(1)
holds. Let $\bm{x}_{\epsilon}^{{\rm F}}$ be the diffusion process
defined in Appendix \ref{subsec: mix_F}. By triangle inequality,
$d_{{\rm TV}}\left(\bm{x}_{\epsilon}^{\mathcal{U}^{(1)}(\bm{m})}(\rho_{\epsilon};\,\bm{y}),\,\mu_{\epsilon}^{\mathcal{U}^{(1)}(\bm{m})}\right)$
is bounded by
\[
d_{{\rm TV}}\left(\bm{x}_{\epsilon}^{\mathcal{U}^{(1)}(\bm{m})}(\rho_{\epsilon};\,\bm{y}),\,\bm{x}_{\epsilon}^{{\rm F}}(\rho_{\epsilon};\,\bm{y})\right)+d_{{\rm TV}}\left(\bm{x}_{\epsilon}^{{\rm F}}(\rho_{\epsilon};\,\bm{y}),\,\mu_{\epsilon}^{{\rm F}}\right)+d_{{\rm TV}}\left(\mu_{\epsilon}^{{\rm F}},\,\mu_{\epsilon}^{\mathcal{U}^{(1)}(\bm{m})}\right)\,.
\]
Note that $\mu_{\epsilon}^{\mathcal{U}(\bm{m})}$ is the same as $\mu_{\epsilon}^{{\rm R}}$
defined in the first paragraph after \cite[Lemma 4.3]{LLS-1st}. Then,
by Theorem \ref{t: mix_F} and \cite[display (4.10)]{LLS-1st}, the
last two terms vanish as $\epsilon\to0$ uniformly over $\bm{y}\in\mathcal{E}(\bm{m})$.

Let $\mathcal{A}\subset\mathbb{R}^{d}$. Then, we have
\begin{equation}
\begin{aligned} & \mathbb{P}_{\bm{y}}^{\epsilon,\,\mathcal{U}^{(1)}(\bm{m})}\left[\bm{x}_{\epsilon}^{\mathcal{U}^{(1)}(\bm{m})}(\rho_{\epsilon})\in\mathcal{A}\right]\\
 & =\mathbb{P}_{\bm{y}}^{\epsilon,\,\mathcal{U}^{(1)}(\bm{m})}\left[\bm{x}_{\epsilon}^{\mathcal{U}^{(1)}(\bm{m})}(\rho_{\epsilon})\in\mathcal{A}\cap\mathcal{U}^{(1)}(\bm{m}),\,\tau_{\partial\mathcal{U}^{(1)}(\bm{m})}>\rho_{\epsilon}\right]+R_{\epsilon}^{(1)}(\bm{y},\,\mathcal{A})
\end{aligned}
\label{e: pf_TV-mix-2-1}
\end{equation}
where
\[
|\,R_{\epsilon}^{(1)}(\bm{y},\,\mathcal{A})\,|\le\mathbb{P}_{\bm{y}}^{\epsilon,\,\mathcal{U}^{(1)}(\bm{m})}\left[\tau_{\partial\mathcal{U}^{(1)}(\bm{m})}\le\rho_{\epsilon}\right]\,.
\]
By the coupling of $\bm{x}_{\epsilon}^{\mathcal{U}^{(1)}(\bm{m})}$
and $\bm{x}_{\epsilon}^{{\rm F}}$, the last probability in \eqref{e: pf_TV-mix-2-1}
can be written as
\begin{equation}
\mathbb{P}_{\bm{y}}^{\epsilon,\,{\rm F}}\left[\bm{x}_{\epsilon}^{{\rm F}}(\rho_{\epsilon})\in\mathcal{A}\cap\mathcal{U}^{(1)}(\bm{m}),\,\tau_{\partial\mathcal{U}^{(1)}(\bm{m})}>\rho_{\epsilon}\right]=\mathbb{P}_{\bm{y}}^{\epsilon,\,{\rm F}}\left[\bm{x}_{\epsilon}^{{\rm F}}(\rho_{\epsilon})\in\mathcal{A}\cap\mathcal{U}^{(1)}(\bm{m})\right]+R_{\epsilon}^{(2)}(\bm{y},\,\mathcal{A})\label{e: pf_TV-mix-2-2}
\end{equation}
where $\mathbb{P}_{\bm{y}}^{\epsilon,\,{\rm F}}$ is the law of $\bm{x}_{\epsilon}^{{\rm F}}$
starting at $\bm{y}$ and
\[
|\,R_{\epsilon}^{(2)}(\bm{y},\,\mathcal{A})\,|\le\mathbb{P}_{\bm{y}}^{\epsilon,\,{\rm F}}\left[\tau_{\partial\mathcal{U}^{(1)}(\bm{m})}\le\rho_{\epsilon}\right]\,.
\]
Now, we have
\begin{equation}
\begin{aligned}\left|\,\mathbb{P}_{\bm{y}}^{\epsilon,\,{\rm F}}\left[\bm{x}_{\epsilon}^{{\rm F}}(\rho_{\epsilon})\in\mathcal{A}\cap\mathcal{U}(\bm{m})\right]-\mathbb{P}_{\bm{y}}^{\epsilon,\,{\rm F}}\left[\bm{x}_{\epsilon}^{{\rm F}}(\rho_{\epsilon})\in\mathcal{A}\right]\,\right| & \le\mathbb{P}_{\bm{y}}^{\epsilon,\,{\rm F}}\left[\bm{x}_{\epsilon}^{{\rm F}}(\rho_{\epsilon})\notin\mathcal{U}(\bm{m})\right]\\
 & \le\mathbb{P}_{\bm{y}}^{\epsilon,\,{\rm F}}\left[\tau_{\partial\mathcal{U}(\bm{m})}\le\rho_{\epsilon}\right]
\end{aligned}
\label{e: pf_TV-mix-2-3}
\end{equation}
so that by \eqref{e: pf_TV-mix-2-1}, \eqref{e: pf_TV-mix-2-2}, and
\eqref{e: pf_TV-mix-2-3},
\begin{align*}
 & \left|\,\mathbb{P}_{\bm{y}}^{\epsilon}\left[\bm{x}_{\epsilon}^{\mathcal{U}^{(1)}(\bm{m})}(\rho_{\epsilon})\in\mathcal{A}\right]-\mathbb{P}_{\bm{y}}^{\epsilon,\,{\rm F}}\left[\bm{x}_{\epsilon}^{{\rm F}}(\rho_{\epsilon})\in\mathcal{A}\right]\,\right|\\
 & \le\mathbb{P}_{\bm{y}}^{\epsilon,\,\mathcal{U}^{(1)}(\bm{m})}\left[\tau_{\partial\mathcal{U}^{(1)}(\bm{m})}\le\rho_{\epsilon}\right]+2\mathbb{P}_{\bm{y}}^{\epsilon,\,{\rm F}}\left[\tau_{\partial\mathcal{U}^{(1)}(\bm{m})}\le\rho_{\epsilon}\right]\\
 & =3\mathbb{P}_{\bm{y}}^{\epsilon}\left[\tau_{\partial\mathcal{U}^{(1)}(\bm{m})}\le\rho_{\epsilon}\right]\,.
\end{align*}
Hence, by \eqref{e: FW_U}, we have
\[
\lim_{\epsilon\to0}\sup_{\bm{y}\in\mathcal{E}(\mathcal{M}')}\sup_{\mathcal{A}\subset\mathbb{R}^{d}}\left|\,\mathbb{P}_{\bm{y}}^{\epsilon}\left[\bm{x}_{\epsilon}^{\mathcal{U}^{(1)}(\bm{m})}(\rho_{\epsilon})\in\mathcal{A}\right]-\mathbb{P}_{\bm{y}}^{\epsilon}\left[\bm{x}_{\epsilon}^{{\rm F}}(\rho_{\epsilon})\in\mathcal{A}\right]\,\right|=0\,,
\]
which completes the proof.
\end{proof}

\appendix

\section{\label{app: Tree}Construction of tree-structure}

In this appendix, we recall the rigorous construction of the tree-structure
presented in \cite[Section 4]{LLS-2nd}. The construction of tree-structure
consists of a positive integer $\mathfrak{q}\in\mathbb{N}$ and quintuples
\[
{\color{blue}\Lambda^{(n)}}:=\left(d^{(n)},\,\mathscr{V}^{(n)},\,\mathscr{N}^{(n)},\,\mathbf{\widehat{y}}^{(n)},\,\mathbf{y}^{(n)}\right)\;,\;\;n\in\llbracket1,\,\mathfrak{q}\rrbracket\,.
\]

\subsection{\label{subsec: MC1}The first layer}

Before we start, we recall several notions for the landscape of $U$
which were introduced in \cite[Section 4.1]{LLS-2nd}.
\begin{itemize}
\item For each pair $\boldsymbol{m}'\neq\boldsymbol{m}''\in\mathcal{M}_{0}$,
denote by $\Theta(\boldsymbol{m}',\,\boldsymbol{m}'')$ the \emph{communication
height} between $\boldsymbol{m}'$ and $\boldsymbol{m}''$: 
\[
{\color{blue}\Theta(\boldsymbol{m}',\,\boldsymbol{m}'')}:=\inf_{\substack{\boldsymbol{z}:[0\,1]\rightarrow\mathbb{R}^{d}}
}\max_{t\in[0,\,1]}U(\boldsymbol{z}(t))\,,
\]
where the infimum is carried over all continuous paths $\boldsymbol{z}(\cdot)$
such that $\boldsymbol{z}(0)=\boldsymbol{m}'$ and $\boldsymbol{z}(1)=\boldsymbol{m}''$.
Clearly, $\Theta(\boldsymbol{m}',\,\boldsymbol{m}'')=\Theta(\boldsymbol{m}'',\,\boldsymbol{m}')$.
\item Denote by $\nabla^{2}U(\bm{x})$ and $D\boldsymbol{\ell}(\bm{x})$
the Hessian of $U$ and the Jacobian of $\boldsymbol{\ell}$ at $\bm{x}\in\mathbb{R}^{d}$,
respectively. By \cite[Lemma 3.3]{LS-22}, for each saddle point $\boldsymbol{\sigma}\in\mathcal{S}_{0}$,
the matrix $(\nabla^{2}U)(\boldsymbol{\sigma})+(D\boldsymbol{\ell})(\boldsymbol{\sigma})$
has one negative eigenvalue, represented by ${\color{blue}-\mu_{\boldsymbol{\sigma}}<0}$.
For $\boldsymbol{\sigma}\in\mathcal{S}_{0}$, let the weight $\omega(\boldsymbol{\sigma})$,
so-called \textit{Eyring--Kramers constant}, be defined by 
\[
{\color{blue}\omega(\boldsymbol{\sigma})}:=\frac{\mu(\boldsymbol{\sigma})}{2\pi\sqrt{-\,\det(\nabla^{2}U)(\boldsymbol{\sigma})}}\,.
\]
\end{itemize}
Let
\[
\mathscr{S}^{(1)}=\mathscr{V}^{(1)}:=\left\{ \,\{\boldsymbol{m}\}:\boldsymbol{m}\in\mathcal{M}_{0}\,\right\} \;\;\text{and}\;\;\mathscr{N}^{(1)}:=\varnothing\,.
\]
For $\bm{m}\in\mathscr{V}^{(1)}$\footnote{For $\bm{m}\in\mathcal{M}_{0}$, we abuse notation as $\bm{m}=\{\bm{m}\}$
without confusion.}, define
\[
{\color{blue}\Xi(\boldsymbol{m})}:=\inf\left\{ \Theta(\boldsymbol{m},\,\boldsymbol{m}'):\boldsymbol{m}'\in\mathcal{M}_{0}\setminus\{\boldsymbol{m}\}\text{ such that }U(\boldsymbol{m}')\le U(\boldsymbol{m})\right\} 
\]
and
\[
d^{(1)}:=\min_{\bm{m}\in\mathscr{V}^{(1)}}\Xi(\bm{m})\,.
\]
Mind that since $|\mathcal{M}_{0}|\ge2$, there exists $\bm{m}\in\mathcal{M}_{0}$
such that $\Xi(\bm{m})<\infty$ so that $d^{(1)}<\infty$. For $\bm{m}\in\mathscr{V}^{(1)}$,
let $\mathcal{S}^{(1)}(\boldsymbol{m})$ be the set of saddle points
connected to the local minimum $\boldsymbol{m}$:
\begin{gather*}
{\color{blue}\mathcal{S}^{(1)}(\boldsymbol{m})}:=\big\{\,\bm{\sigma}\in\mathcal{S}_{0}:\boldsymbol{\sigma}\curvearrowright\boldsymbol{m}\ ,\ \ U(\bm{\sigma})=U(\bm{m})+\Xi(\bm{m})\,\big\}\,.
\end{gather*}
Denote by $\mathcal{S}(\boldsymbol{m},\boldsymbol{m}')$, $\boldsymbol{m}'\neq\boldsymbol{m}$,
the set of saddle points which separate $\boldsymbol{m}$ from $\boldsymbol{m}'$:
\[
{\color{blue}\mathcal{S}(\boldsymbol{m},\boldsymbol{m}')}:=\big\{\,\boldsymbol{\sigma}\in\mathcal{S}^{(1)}(\boldsymbol{m}):\boldsymbol{\sigma}\curvearrowright\boldsymbol{m}\;,\;\;\boldsymbol{\sigma}\curvearrowright\boldsymbol{m}'\,\big\}\,,
\]
and denote by $\omega(\boldsymbol{m},\boldsymbol{m}')$ the sum of
the weights of the saddle points in $\mathcal{S}(\boldsymbol{m},\boldsymbol{m}')$:
\[
{\color{blue}\omega(\boldsymbol{m},\,\boldsymbol{m}')}:=\sum_{\boldsymbol{\sigma}\in\mathcal{S}(\boldsymbol{m},\boldsymbol{m}')}\,\omega(\boldsymbol{\sigma})\;,\ \ {\color{blue}\omega_{1}(\bm{m},\,\bm{m}')}:=\omega(\bm{m},\,\bm{m}')\,\boldsymbol{1}\left\{ \,\Xi(\boldsymbol{m})=d^{(1)}\,\right\} \,.
\]
Define
\[
{\color{blue}r^{(1)}(\boldsymbol{m},\,\boldsymbol{m}')}:=\frac{1}{\nu(\boldsymbol{m})}\,\omega_{1}(\boldsymbol{m},\boldsymbol{m}')
\]
and let $\widehat{\bf y}^{(1)}(\cdot)={\bf y}^{(1)}(\cdot)$ be a Markov chain
on $\mathscr{V}^{(1)}$ with jump rates $r^{(1)}(\cdot,\,\cdot)$.
If there exists only one irreducible class of ${\bf y}^{(1)}$, the
construction is complete.

\subsection{\label{subsec: MC2}The upper level}

First, we recall several notions defined in \cite[Section 4.2]{LLS-2nd}.
\begin{itemize}
\item Recall that $\mathcal{M}\subset\mathcal{M}_0$ is \textit{\textcolor{blue}{simple}} if
\[
U(\bm{m})=U(\bm{m}')\ \ \text{for all}\ \bm{m},\,\bm{m}'\in\mathcal{M}\,,
\]
and the common value is denoted by \textcolor{blue}{$U(\mathcal{M})$}.
\item For two disjoint non-empty subsets $\mathcal{M}$ and $\mathcal{M}'$
of $\mathcal{M}_{0}$, let $\Theta(\mathcal{M},\,\mathcal{M}')$ be
the \textit{communication height} between the two sets: 
\[
{\color{blue}\Theta(\mathcal{M},\,\mathcal{M}')}:=\min_{\boldsymbol{m}\in\mathcal{M},\,\boldsymbol{m}'\in\mathcal{M}'}\Theta(\boldsymbol{m},\,\boldsymbol{m}')\,,
\]
with the convention that $\Theta(\mathcal{\mathcal{M}},\,\varnothing)=+\infty$.\smallskip{}
\item For a simple set $\mathcal{\mathcal{M}}\subset\mathcal{M}_{0}$, denote
by $\widetilde{\mathcal{M}}$ the set of local minima of $U$ which
do not belong to $\mathcal{M}$ and which have lower or equal energy
than $\mathcal{M}$: 
\[
{\color{blue}\widetilde{\mathcal{M}}}:=\big\{\,\boldsymbol{m}\in\mathcal{M}_{0}\setminus\mathcal{M}:U(\boldsymbol{m})\le U(\mathcal{M})\,\big\}\,.
\]
Note that $\widetilde{\mathcal{\mathcal{M}}}=\varnothing$ if and
only if $\mathcal{M}$ contains all the global minima of $U$.\smallskip{}
\item For a saddle point $\boldsymbol{\sigma}\in\mathcal{S}_{0}$ and local
minimum $\boldsymbol{m}\in\mathcal{M}_{0}$, we write\textit{\textcolor{blue}{{}
$\boldsymbol{\sigma}\rightsquigarrow\boldsymbol{m}$}} if $\boldsymbol{\sigma}\curvearrowright\boldsymbol{m}$
or if there exist $n\ge1$, $\boldsymbol{\sigma}_{1},\,\dots,\,\boldsymbol{\sigma}_{n}\in\mathcal{S}_{0}$
and $\boldsymbol{m}_{1}\,\,\dots,\,\boldsymbol{m}_{n}\in\mathcal{M}_{0}$
such that 
\[
\max\{U(\boldsymbol{\sigma}_{1}),\,\dots,\,U(\boldsymbol{\sigma}_{n})\,\}<U(\boldsymbol{\sigma})\;\;\;\text{and\;\;\;}\boldsymbol{\sigma}\curvearrowright\boldsymbol{m}_{1}\curvearrowleft\boldsymbol{\sigma}_{1}\curvearrowright\cdots\curvearrowright\boldsymbol{m}_{n}\curvearrowleft\boldsymbol{\sigma}_{n}\curvearrowright\boldsymbol{m}\,.
\]
For $\mathcal{\mathcal{M}}\subset\mathcal{M}_{0}$, write ${\color{blue}\boldsymbol{\sigma}\rightsquigarrow\mathcal{\mathcal{M}}}$
and \textcolor{blue}{$\bm{\sigma}\curvearrowright\mathcal{M}$} if
for some $\boldsymbol{m}\in\mathcal{\mathcal{M}}$, $\boldsymbol{\sigma}\rightsquigarrow\boldsymbol{m}$
and $\bm{\sigma}\curvearrowright\bm{m}$ , respectively.\smallskip{}
\item Fix a non-empty simple set $\mathcal{\mathcal{M}}\subset\mathcal{M}_{0}$
such that $\widetilde{\mathcal{M}}\neq\varnothing$. For a set $\mathcal{\mathcal{M}}'\subset\mathcal{M}_{0}$
such that $\mathcal{M}'\cap\mathcal{M}=\varnothing$, we write\textit{\textcolor{blue}{{}
$\mathcal{M}\rightarrow\mathcal{M}'$}} if there exists $\boldsymbol{\sigma}\in\mathcal{S}_{0}$
such that\textcolor{red}{{} }
\begin{equation}
U(\boldsymbol{\sigma})=\Theta(\mathcal{M},\,\widetilde{\mathcal{M}})=\Theta(\mathcal{M},\,\mathcal{M}')\;\;\text{and}\;\;\mathcal{M}'\,\curvearrowleft\,\boldsymbol{\sigma}\,\rightsquigarrow\,\mathcal{M}\,.\label{eq:con_gate}
\end{equation}
To emphasize the saddle point $\boldsymbol{\sigma}$ between $\mathcal{M}$
and $\mathcal{M}'$ we sometimes write ${\color{blue}\mathcal{M}\rightarrow_{\boldsymbol{\sigma}}\mathcal{M}'}$.
\smallskip{}
\item Denote by $\mathcal{S}(\mathcal{M},\,\mathcal{M}')$ the set of saddle
points $\boldsymbol{\sigma}\in\mathcal{S}_{0}$ satisfying \eqref{eq:con_gate},
\[
{\color{blue}\mathcal{S}(\mathcal{M},\,\mathcal{M}')}:=\{\,\boldsymbol{\sigma}\in\mathcal{S}_{0}:\mathcal{M}\to_{\boldsymbol{\sigma}}\mathcal{M}'\,\}\,.
\]
The set $\mathcal{S}(\mathcal{M},\,\mathcal{M}')$ represents the
collection of lowest connection points which separate $\mathcal{M}$
from $\mathcal{M}'$.
\end{itemize}
Now, fix $k\ge1$ and suppose that $\Lambda^{(n)}$, $n\in\llbracket1,\,k\rrbracket$,
have been defined and there exist more than one irreducible class
of the Markov chain ${\bf y}^{(k)}$. Denote by $\mathfrak{n}_{k}$
the number of irreducible classes, by $\mathscr{R}_{1}^{(k)},\,\dots,\,\mathscr{R}_{\mathfrak{n}_{k}}^{(k)}$
such irreducible classes and by $\mathscr{T}^{(k)}$ the collection
of transient states of ${\bf y}^{(k)}$, respectively. Define
\begin{equation*}
\mathcal{M}_{i}^{(k+1)}:=\bigcup_{\mathcal{M}\in\mathscr{R}_{i}^{(k)}}\mathcal{M}\ ;\ i\in\llbracket 1,\, \mathfrak{n}_k\rrbracket\, ,
\end{equation*}
and
\begin{equation}
\mathscr{V}^{(k+1)}:=\left\{\mathcal{M}_{1}^{(k+1)},\, \dots,\, \mathcal{M}_{\mathfrak{n}_k}^{(k+1)} \right\},\, \mathscr{N}^{(k+1)}:=\mathscr{N}^{(k)}\cup\mathscr{T}^{(k)},\, \mathscr{S}^{(k+1)}:=\mathscr{V}^{(k+1)}\cup\mathscr{N}^{(k+1)}\,.
\label{e: V^k+1,N^k+1}
\end{equation}
By Proposition \ref{p: tree}-(1) below, all $\mathcal{M}\in\mathscr{S}^{(k+1)}$
are simple. For $\mathcal{M}\in\mathscr{V}^{(k+1)}$, define
\[
\Xi(\mathcal{M}):=\Theta(\mathcal{M},\,\widetilde{\mathcal{M}})-U(\mathcal{M})\ \ \text{and}\ \ d^{(k+1)}:=\min_{\mathcal{M}\in\mathscr{V}^{(k+1)}}\Xi(\mathcal{M})\,.
\]
Mind that since $\mathfrak{n}_{k}\ge2$, there exists $\mathcal{M}\in\mathscr{V}^{(k+1)}$
such that $\Xi(\mathcal{M})<\infty$ so that $d^{(k+1)}<\infty$.
Denote by $\mathscr{S}^{(k+1)}:=\mathscr{V}^{(k+1)}\cup\mathscr{N}^{(k+1)}$.

Denote by $\widehat{r}^{(k)}(\cdot,\,\cdot)$ the jump rates of the
$\mathscr{S}^{(k)}$-valued Markov chain $\widehat{\mathbf{y}}^{(k)}(\cdot)$.
Since $\mathscr{S}^{(k+1)}=\mathscr{V}^{(k+1)}\cup\mathscr{N}^{(k+1)}$,
we can divide the definition of the jump rate \textcolor{blue}{\emph{$\widehat{r}^{(k+1)}(\cdot,\,\cdot)$
}}of $\widehat{\mathbf{y}}^{(k+1)}(\cdot)$ into three cases: 
\begin{itemize}
\item {[}\textbf{Case }1: $\mathcal{M}\in\mathscr{N}^{(k+1)}$ and $\mathcal{M}'\in\mathscr{N}^{(k+1)}${]}
By \eqref{e: V^k+1,N^k+1}, $\mathscr{N}^{(k+1)}\subset\mathscr{S}^{(k)}$ so that in this case, $\mathcal{M},\,\mathcal{M}'\in\mathscr{S}^{(k)}$. Now, we set 
\[
\widehat{r}^{(k+1)}(\mathcal{M},\,\mathcal{M}'):=\ \widehat{r}^{(k)}(\mathcal{M},\,\mathcal{M}')\,.
\]
\item {[}\textbf{Case }2: $\mathcal{M}\in\mathscr{N}^{(k+1)}$ and $\mathcal{M}'\in\mathscr{V}^{(k+1)}${]}
By \eqref{e: V^k+1,N^k+1}, as in the first case, $\mathcal{M}\in\mathscr{S}^{(k)}$. Since $\mathcal{M}'$
is the union of elements (may be just one) in $\mathscr{V}^{(k)}$,
we set 
\[
\widehat{r}^{(k+1)}(\mathcal{M},\,\mathcal{M}'):=\sum_{\mathcal{M}''\in\mathscr{V}^{(k)}(\mathcal{M}')}\,\widehat{r}^{(k)}(\mathcal{M},\,\mathcal{M}'')\,.
\]
\item {[}\textbf{Case 3}: $\mathcal{M}\in\mathscr{V}^{(k+1)}$ and $\mathcal{M}'\in\mathscr{S}^{(k+1)}${]}
Let 
\[
\boldsymbol{\omega}(\mathcal{M},\,\mathcal{M}')=\sum_{\boldsymbol{\sigma}\in\mathcal{S}(\mathcal{M},\,\mathcal{M}')}\omega(\boldsymbol{\sigma})\;,\ \ {\color{blue}\omega_{k+1}(\mathcal{M},\,\mathcal{M}')}:=\omega(\mathcal{M},\,\mathcal{M}')\,\boldsymbol{1}\{\,\Xi(\mathcal{M})=d^{(k+1)}\,\}\,.
\]
It is understood here that $\omega(\mathcal{M},\,\mathcal{M}')=0$
if the set $\mathcal{S}(\mathcal{M},\,\mathcal{M}')$ is empty, that
is if $\mathcal{M}'$ is not adjacent to $\mathcal{M}$. Then, we
set
\begin{equation}
\widehat{r}^{(k+1)}(\mathcal{M},\,\mathcal{M}'):=\frac{1}{\nu(\mathcal{M})}\,\omega_{k+1}(\mathcal{M},\,\mathcal{M}')\,.\label{eq:rate_3}
\end{equation}
\end{itemize}
Define $\widehat{\mathbf{y}}^{(k+1)}(\cdot)$ as the $\mathscr{S}^{(k+1)}$-valued,
continuous-time Markov process with jump rates $\widehat{r}^{(k+1)}(\cdot,\,\cdot)$.
By Lemma \ref{l: F_assu_trace}, the trace process of $\widehat{\mathbf{y}}^{(k+1)}$
on $\mathscr{V}^{(k+1)}$ is well defined. Denote by $\mathbf{y}^{(k+1)}(\cdot)$
the trace process. From now on the quintuples $\Lambda^{(1)},\,\dots,\,\Lambda^{(k+1)}$
have been defined. If $\mathfrak{n}_{k+1}$, the number of irreducible
classes of ${\bf y}^{(k+1)}$, is $1$, the construction is complete
and $\mathfrak{q}=k+1$. If $\mathfrak{n}_{k+1}>1$, we add a new
layer as in this subsection.

Finally, we have the following proposition on tree-structure. We recall
from \cite[display (4.14)]{LLS-2nd} that $\mathcal{M}\subset\mathcal{M}_{0}$
is said to be \textit{\textcolor{blue}{bound}} if
\[
\max_{\bm{m},\,\bm{m}'}\Theta(\bm{m},\,\bm{m}')<\Theta(\mathcal{M},\,\widetilde{\mathcal{M}})\,.
\]

\begin{proposition}
\label{p: tree}We have the following.
\begin{enumerate}
\item If $\mathfrak{n}_{n}>1$, $\mathfrak{n}_{n}>\mathfrak{n}_{n+1}$.
In particular, there exists $\mathfrak{q}\in\mathbb{N}$ such that
$\mathfrak{n}_{1}>\cdots>\mathfrak{n}_{\mathfrak{q}}=1$.
\item For all $n\in\llbracket1,\,\mathfrak{q}\rrbracket$ and $\mathcal{M}\in\mathscr{S}^{(n)}$,
$\mathcal{M}$ is simple and bound.
\item $0<d^{(1)}<\cdots<d^{(\mathfrak{q})}<\infty$.
\end{enumerate}
\end{proposition}

\begin{proof}
The first property is \cite[Theorem 4.7-(3)]{LLS-2nd}. The last two properties are $\mathfrak{P}_{1}$ and $\mathfrak{P}_{2}$ defined
in \cite[Definition 4.4]{LLS-2nd} and it was proved in \cite[Corollary 4.8]{LLS-2nd}
that properties $\mathfrak{P}_{1}$ and $\mathfrak{P}_{2}$ hold. 
\end{proof}

\subsection{\label{subsec: example}Example of tree-structure}

\begin{figure}
\center
\includegraphics[scale=0.25]{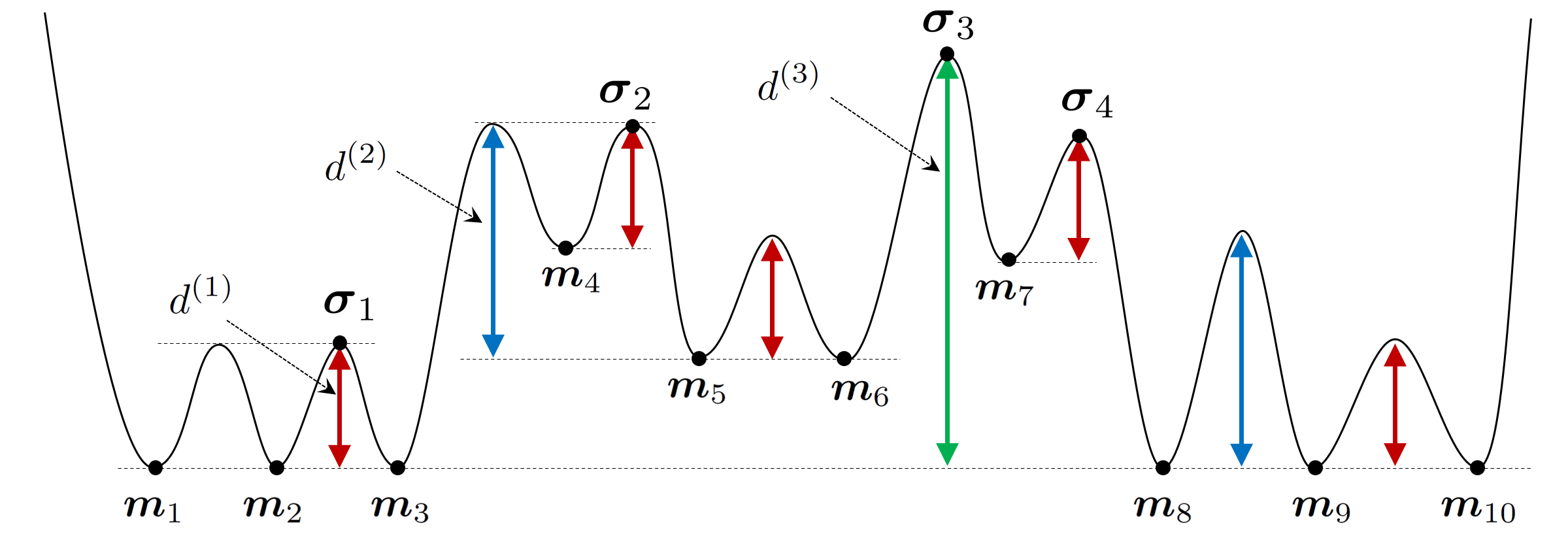}
\caption{An example of potential function $U$}
\label{fig:potential}
\end{figure}

In this appendix, we briefly explain the tree-structure with a simple
one-dimensional example. Suppose that we have a potential $U:\mathbb{R}\to\mathbb{R}$
whose shape is given as in Figure \ref{fig:potential}. This example
was given in \cite{LLS-2nd}. In this example, there are three depths
$d^{(1)}<d^{(2)}<d^{(3)}$. Since there are ten local minima, we have$\mathscr{V}^{(1)}=\mathcal{M}_{0}=\{\bm{m}_{1},\,\dots,\,\bm{m}_{10}\}$
and $\mathscr{N}^{(1)}=\varnothing$. The first Markov chain ${\bf y}^{(1)}(\cdot)$
describing the metastable behavior in the smallest time scale $e^{d^{(1)}/\epsilon}$,
has the following jump diagram:

\[
\bm{m}_{1}\longleftrightarrow\bm{m}_{2}\longleftrightarrow\bm{m}_{3}\leftarrow\bm{m}_{4}\rightarrow\bm{m}_{5}\longleftrightarrow\bm{m}_{6}\ \Big|\ \bm{m}_{7}\rightarrow\bm{m}_{8}\ \Big|\ \bm{m}_{9}\longleftrightarrow\bm{m}_{10}
\]
Therefore, there are four irreducible classes

\[
\mathscr{R}_{1}^{(1)}=\{\bm{m}_{1},\,\bm{m}_{2},\,\bm{m}_{3}\}\ ,\ \ \mathscr{R}_{2}^{(1)}=\{\bm{m}_{5},\,\bm{m}_{6}\}\ ,\ \ \mathscr{R}_{3}^{(1)}=\{\bm{m}_{8}\}\ ,\ \ \mathscr{R}_{4}^{(1)}=\{\bm{m}_{9},\,\bm{m}_{10}\}\,,
\]
and two transient states
\[
\mathscr{T}^{(1)}=\{\bm{m}_{4},\,\bm{m}_{7}\}\ .
\]
Based on the decomposition of $\mathscr{V}^{(1)}$, we get $\mathscr{V}^{(2)}=\{\mathscr{R}_{1}^{(1)},\,\mathscr{R}_{2}^{(1)},\,\mathscr{R}_{3}^{(1)},\,\mathscr{R}_{4}^{(1)}\}$
and $\mathscr{N}^{(2)}=\mathscr{T}^{(1)}$ so that
\[
\mathscr{V}^{(2)}=\left\{ \{\bm{m}_{1},\,\bm{m}_{2},\,\bm{m}_{3}\},\,\{\bm{m}_{5},\,\bm{m}_{6}\},\,\{\bm{m}_{8}\},\,\{\bm{m}_{9},\,\bm{m}_{10}\}\right\} \ ,\ \ \mathscr{N}^{(2)}=\{\bm{m}_{4},\,\bm{m}_{7}\}\,.
\]
In the second time scale $e^{d^{(2)}/\epsilon}$, since $\mathscr{N}^{(2)}\ne\varnothing$,
we need a auxiliary Markov chain $\widehat{{\bf y}}^{(2)}(\cdot)$
whose jump diagram is given by

\[
\{\bm{m}_{1},\,\bm{m}_{2},\,\bm{m}_{3}\}\leftarrow\{\bm{m}_{4}\}\longleftrightarrow\{\bm{m}_{5},\,\bm{m}_{6}\}\ \Big|\ \{\bm{m}_{7}\}\rightarrow\{\bm{m}_{8}\}\longleftrightarrow\{\bm{m}_{9},\,\bm{m}_{10}\}
\]
Therefore, the second Markov chain ${\bf y}^{(2)}(\cdot)$, which
is a trace process of $\widehat{{\bf y}}^{(2)}(\cdot)$, has the following
jump diagram:

\[
\{\bm{m}_{1},\,\bm{m}_{2},\,\bm{m}_{3}\}\leftarrow\{\bm{m}_{5},\,\bm{m}_{6}\}\ \Big|\ \{\bm{m}_{8}\}\longleftrightarrow\{\bm{m}_{9},\,\bm{m}_{10}\}
\]
In the second layer, the state space $\mathscr{V}^{(2)}$ is decomposed
into $\mathscr{V}^{(2)}=\mathscr{R}_{1}^{(2)}\cup\mathscr{R}_{2}^{(2)}\cup\mathscr{T}^{(2)}$
where

\[
\mathscr{R}_{1}^{(2)}=\left\{ \{\bm{m}_{1},\,\bm{m}_{2},\,\bm{m}_{3}\}\right\} \ ,\ \ \mathscr{R}_{2}^{(2)}=\left\{ \{\bm{m}_{8}\},\,\{\bm{m}_{9},\,\bm{m}_{10}\}\right\} \ ,\ \ \mathscr{T}^{(2)}=\left\{ \{\bm{m}_{5},\,\bm{m}_{6}\}\right\} \,.
\]
For the last layer, we have $\mathscr{V}^{(3)}=\{\bigcup_{\mathcal{M}\in\mathscr{R}_{1}^{(2)}}\mathcal{M},\,\bigcup_{\mathcal{M}\mathscr{R}_{2}^{(2)}}\mathcal{M}\}$
and $\mathscr{N}^{(3)}=\mathscr{N}^{(2)}\cup\mathscr{T}^{(2)}$ so
that
\[
\mathscr{V}^{(3)}=\left\{ \{\bm{m}_{1},\,\bm{m}_{2},\,\bm{m}_{3}\},\,\{\bm{m}_{8},\,\bm{m}_{9},\,\bm{m}_{10}\}\right\} \ ,\ \ \mathscr{N}^{(3)}=\left\{ \bm{m}_{4},\,\{\bm{m}_{5},\,\bm{m}_{6}\},\,\bm{m}_{7}\right\} \,.
\]
The third Markov chain $\widehat{{\bf y}}^{(3)}(\cdot)$ has the following
jump diagram 

\[
\{\bm{m}_{1},\,\bm{m}_{2},\,\bm{m}_{3}\}\leftarrow\{\bm{m}_{4}\}\longleftrightarrow\{\bm{m}_{5},\,\bm{m}_{6}\}\leftarrow\{\bm{m}_{8},\,\bm{m}_{9},\,\bm{m}_{10}\}
\]
and
\[
\{\bm{m}_{1},\,\bm{m}_{2},\,\bm{m}_{3}\}\rightarrow\{\bm{m}_{7}\}\rightarrow\{\bm{m}_{8},\,\bm{m}_{9},\,\bm{m}_{10}\}
\]
Hence, the last Markov chain ${\bf y}^{(3)}(\cdot)$, which describes
the metastable behavior between global minima in the last time scale,
has the following diagram:

\[
\{\bm{m}_{1},\,\bm{m}_{2},\,\bm{m}_{3}\}\longleftrightarrow\{\bm{m}_{8},\,\bm{m}_{9},\,\bm{m}_{10}\}
\]
Finally, we obtain $\mathscr{R}_{1}^{(3)}=\left\{ \{\bm{m}_{1},\,\bm{m}_{2},\,\bm{m}_{3}\},\,\{\bm{m}_{8},\,\bm{m}_{9},\,\bm{m}_{10}\}\right\} $,
$\mathscr{T}^{(3)}=\varnothing$, and $\mathcal{M}_{\star}=\mathcal{M}_{1}^{(3+1)}=\{\bm{m}_{1},\,\bm{m}_{2},\,\bm{m}_{3},\,\bm{m}_{8},\,\bm{m}_{9},\,\bm{m}_{10}\}$.
We can illustrate the tree-structure associated to our example as
in Figure \ref{fig:example_tree}.

\begin{figure}
\center
\includegraphics[scale=0.07]{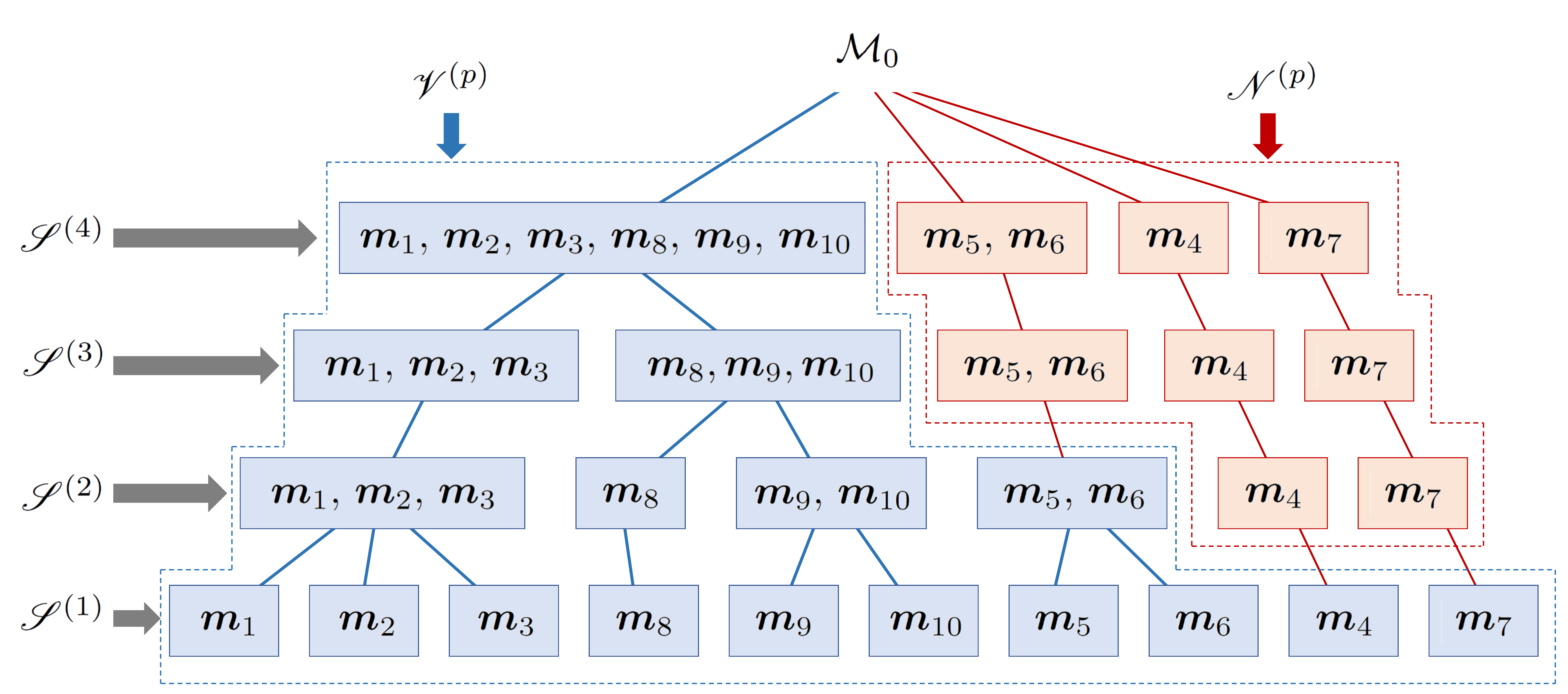}
\caption{Tree-structure associated with the potential given in Figure \ref{fig:potential}}
\label{fig:example_tree}
\end{figure}

\section{Preliminaries from \cite{LLS-1st,LLS-2nd}}

\subsection{\label{subsec: mix_F}Local ergodicity}

Fix $\bm{m}\in\mathcal{M}_{0}$ and assume $\bm{m}=\bm{0}$. In this
subsection, we recall the definition of $\bm{x}_{\epsilon}^{{\rm F}}(\cdot)$,
which is a content of \cite[Sections 3, 4 and Appendix B]{LLS-1st}.
This process is needed to prove the initial mixing condition $\mathfrak{M}^{(1)}$.
Denote by $D\bm{a}$ the Jacobian matrix of a vector
field $\bm{a}:\mathbb{R}^{d}\to\mathbb{R}^{d}$ and by $\|\mathbb{M}\|:=\sup_{|\bm{y}|=1}|\mathbb{M}\bm{y}|$
the matrix norm of $\mathbb{M}$. By \cite[Proposition B.1]{LLS-1st},
there exist $r_{3}>0$ and vector field $\bm{b}_{0}\in C^{1}(\mathbb{R}^{d};\,\mathbb{R}^{d})$
satisfying
\begin{enumerate}
\item $\bm{b}_{0}(\bm{x})=\bm{b}(\bm{x})=-\nabla U(\bm{x})-\bm{\ell}(\bm{x})$
for $\bm{x}\in B(\bm{0},\,r_{3})$.
\item $\bm{b}_{0}(\bm{x})=\bm{0}$ if and only if $\bm{x}=\bm{0}$.
\item There exists $R>0$ and $C>0$ such that for all $|\bm{x}|\ge R$,
\[
|\bm{b}_{0}(\bm{x})|\le C_{1}\bm{x}\ \ \text{and}\ \ \|D\bm{b}_{0}(\bm{x})\|\le C_{1}|\bm{x}|\,.
\]
\item For all $\bm{x}\in\mathbb{R}^{d}$,
\[
-\bm{b}_{0}(\bm{x})\cdot\mathbb{H}\bm{x}\,\ge\,\frac{1}{2}|\mathbb{H}\bm{x}|^{2}
\]
where $\mathbb{H}:=\nabla^{2}U(\bm{0})$.
\end{enumerate}
For $r>0$, let
\[
\mathcal{D}_{r}\,:=\,\{\,\bm{x}\in\mathbb{R}^{d}\,:\,\bm{x}\cdot\mathbb{H}\bm{x}\le r^{2}\,\}\,.
\]
By \cite[display (3.9)]{LLS-1st}, there exists $r_{4}>0$ such that
$\mathcal{D}_{2r_{4}}\subset B(\bm{0},\,r_{3})$.

Let $\bm{x}_{\epsilon}^{{\rm F}}(\cdot)$ be a diffusion process described
by the following SDE:
\[
d\bm{x}_{\epsilon}^{{\rm F}}(t)=\bm{b}_{0}(\bm{x}_{\epsilon}^{{\rm F}}(t))\,dt\,+\,\sqrt{2\epsilon}\,dW_{t}\,.
\]
Let $\mu_{\epsilon}^{{\rm F}}$ be a unique stationary distribution
of $\bm{x}_{\epsilon}^{{\rm F}}(\cdot)$. Then, we have the following.
\begin{theorem}
\label{t: mix_F}We have
\[
\lim_{\epsilon\to0}\sup_{\bm{y}\in\mathcal{W}^{2r_{0}}(\bm{m})}d_{{\rm TV}}(\bm{x}_{\epsilon}^{{\rm F}}(\epsilon^{-1};\,\bm{y}),\,\mu_{\epsilon}^{{\rm F}})=0\,.
\]
\end{theorem}

\begin{proof}
By \cite[Theorem 3.1]{LLS-1st}, for $A\in(0,\,1/3)$, we have
\[
\lim_{\epsilon\to0}\sup_{\bm{y}\in\mathcal{D}_{r_{4}}}d_{{\rm TV}}(\bm{x}_{\epsilon}^{{\rm F}}(\epsilon^{-A};\,\bm{y}),\,\mu_{\epsilon}^{{\rm F}})=0\,.
\]
By the condition (e) above \cite[display (2.12)]{LLS-1st}, we have
$\mathcal{W}^{2r_{0}}(\bm{m})\subset\mathcal{D}_{r_{4}}$ (Mind that
we assumed $\bm{m}=\bm{0}$). In addition, by Lemma \ref{l: d_TV},
$d_{{\rm TV}}(\bm{x}_{\epsilon}^{{\rm F}}(\epsilon^{-1};\,\bm{y}),\,\mu_{\epsilon}^{{\rm F}})\le d_{{\rm TV}}(\bm{x}_{\epsilon}^{{\rm F}}(\epsilon^{-A};\,\bm{y}),\,\mu_{\epsilon}^{{\rm F}})$.
Therefore,
\[
\lim_{\epsilon\to0}\sup_{\bm{y}\in\mathcal{W}^{2r_{0}}(\bm{m})}d_{{\rm TV}}(\bm{x}_{\epsilon}^{{\rm F}}(\epsilon^{-1};\,\bm{y}),\,\mu_{\epsilon}^{{\rm F}})=0\,.
\]
\end{proof}

\subsection{Hitting time}

The following proposition is from \cite[Corollary 6.2 and Lemma 6.7]{LLS-1st}.
\begin{proposition}
\label{p: hit_min}We have the following
\begin{enumerate}
\item For all $\bm{x}\in\mathbb{R}^{d}$ and $C>0$,
\[
\limsup_{\epsilon\to0}\mathbb{P}_{\bm{x}}^{\epsilon}\left[\tau_{\mathcal{E}(\mathcal{M}_{0})}>\frac{C}{\epsilon}\right]=0\,.
\]
\item For all $\bm{m}\in\mathcal{M}_{0}$ and $\bm{x}\in\mathcal{D}(\bm{m})$,
\[
\liminf_{\epsilon\to0}\mathbb{P}_{\bm{x}}^{\epsilon}\left[\tau_{\mathcal{E}(\mathcal{M}_{0})}=\tau_{\mathcal{E}(\bm{m})}\right]=1\,.
\]
\end{enumerate}
\end{proposition}

The following proposition is from \cite[Proposition 4.2]{LLS-1st}.
\begin{proposition}
\label{p_FW}Fix $h<H$, and denote by $\mathcal{A}$, $\mathcal{B}$
connected components of the set $\{\boldsymbol{x}\in\mathbb{R}^{d}:U(\boldsymbol{x})<h\}$,
$\{\boldsymbol{x}\in\mathbb{R}^{d}:U(\boldsymbol{x})<H\}$, respectively.
Assume that $\mathcal{A}\subset\mathcal{B}$. Suppose that all critical
points $\boldsymbol{c}$ of $U$ in $\mathcal{A}$ are such that $U(\boldsymbol{c})\le h_{0}$
for some $h_{0}<h$. Then, for all $\eta>0$, 
\[
\limsup_{\epsilon\rightarrow0}\,\sup_{\boldsymbol{x}\in\mathcal{A}}\,\mathbb{P}_{\boldsymbol{x}}^{\epsilon}\left[\,\tau_{\partial\mathcal{B}}<e^{(H-h_{0}-\eta)/\epsilon}\,\right]=0\,.
\]
\end{proposition}

\subsection{\label{app: trace}Trace processes}

In this subsection, we recall definition and condition of well-definedness
of trace process, which is a part of \cite[Appendix C]{LLS-2nd}.
Let $\{\bm{z}(t)\}_{t\ge0}$ be a non-explosive, continuous-time Markov
process on a certain state space $E$. Let $F$ be a non-empty proper
subset of $F$ and denote by $T^{F}:[0,\infty)\to[0,\infty)$ the
time spent by $\bm{z}(\cdot)$ staying in $F$:
\[
{\color{blue}T^{F}(t)}:=\int_{0}^{t}\,{\bf 1}\left\{ \,\bm{z}(s)\in F\,\right\} \,ds\,.
\]
Define generalized inverse $S^{F}:[0,\infty)\to[0,\infty)$ by
\[
{\color{blue}S^{F}(t)}:=\sup\,\left\{ \,s\ge0:T^{F}(s)\le t\,\right\} \,.
\]
Suppose that (cf. \cite[Section 2.2]{BeltranLandim1}). 
\begin{equation}
\lim_{t\to\infty}\,T^{F}(t)=\infty\;\ \text{almost surely}\,.\label{e_F_assu_trace}
\end{equation}
Then, the trace process $\bm{z}^{F}(\cdot)$ of $\boldsymbol{z}(\cdot)$
on the set $F$ is defined by
\[
\bm{z}^{F}(t):=\bm{z}(S^{F}(t))\,.
\]
Now, we have the following.
\begin{lemma}[{\cite[Lemma C.1]{LLS-2nd}}]
\label{l: F_assu_trace}Suppose that $\boldsymbol{z}(\cdot)$ is
a Markov chain on the finite set $E$ and that $F$ contains at least
one element of each irreducible class of $\boldsymbol{z}(\cdot)$.
Then, the condition \eqref{e_F_assu_trace} holds.
\end{lemma}

\section{\label{app: pf_l_U_M}Proof of Lemmas \ref{l: U_M-0} and \ref{l: U_M}}
\begin{proof}[Proof of Lemma \ref{l: U_M-0}]
It is obvious when $n=1$. Fix $n\in\llbracket2,\,\mathfrak{q}\rrbracket$
and $\mathcal{M}\in\mathscr{V}^{(n)}$. By Proposition \ref{p: tree}-(2),
$\mathcal{M}$ is simple and bound. By \cite[Lemma 5.3]{LLS-2nd}
and Proposition \ref{p: tree}-(3),
\[
\max_{\bm{m},\,\bm{m}'}\Theta(\bm{m},\,\bm{m}')-U(\mathcal{M})\le d^{(n-1)}<R^{(n)}\,.
\]
Then, by \cite[Lemma B.9]{LLS-2nd}, there exists a connected component
of $\{U<U(\mathcal{M})+R^{(n)}\}$ containing $\mathcal{M}$.
\end{proof}
We recall from \cite[Definition 5.1]{LLS-2nd} that for a set $\mathcal{A\subset\mathbb{R}}^{d}$,
we write
\[
{\color{blue}\mathcal{M}^{*}(\mathcal{A})}:=\bigg\{\,\boldsymbol{m}\in\mathcal{M}_{0}:U(\boldsymbol{m})=\min_{\boldsymbol{x}\in\mathcal{A}}U(\boldsymbol{x})\,\bigg\}\,.
\]
Let $n\in\llbracket1,\,\mathfrak{q}\rrbracket$ and $\mathcal{M}\in\mathscr{V}^{(n)}$.
By \cite[Lemma 5.7]{LLS-2nd}, we have
\[
d^{(n)}\le\Theta(\mathcal{M},\,\widetilde{\mathcal{M}})-U(\mathcal{M})\,,
\]
which implies $R^{(n)}<\Theta(\mathcal{M},\,\widetilde{\mathcal{M}})-U(\mathcal{M})$.
Therefore, by \cite[Lemma B.9]{LLS-2nd}, 
\begin{equation}
\mathcal{M}^{*}(\mathcal{U}^{(n)}(\mathcal{M}))=\mathcal{M}\,.\label{e: M^*_U}
\end{equation}
Now, we are ready to prove Lemma \ref{l: U_M}.
\begin{proof}[Proof of Lemma \ref{l: U_M}]
Fix $n\in\llbracket1,\,\mathfrak{q}\rrbracket$ and $\mathcal{M},\,\mathcal{M}'\in\mathscr{V}^{(n)}$
such that $\mathcal{M}\ne\mathcal{M}'$. First, by definition, $\mathcal{E}(\bm{m})\subset\mathcal{U}^{(n)}(\mathcal{M})$
for all $\bm{m}\in\mathcal{M}$ so that $\mathcal{E}(\mathcal{M})\subset\mathcal{U}^{(n)}(\mathcal{M})$.
Suppose that there exists $\bm{x}\in\overline{\mathcal{U}^{(n)}(\mathcal{M})}\cap\overline{\mathcal{U}^{(n)}(\mathcal{M}')}$.
By \eqref{e: M^*_U},
\[
\mathcal{M}^{*}(\mathcal{U}^{(n)}(\mathcal{M}))\ne\mathcal{M}^{*}(\mathcal{U}^{(n)}(\mathcal{M}'))\,.
\]
If $\mathcal{U}^{(n)}(\mathcal{M})\cap\mathcal{U}^{(n)}(\mathcal{M}')\ne\varnothing$,
since they are connected components of $\{U<U(\mathcal{M})+R^{(n)}\}$,
$\mathcal{U}^{(n)}(\mathcal{M})=\mathcal{U}^{(n)}(\mathcal{M}')$
which contradicts to the above displayed equation. Therefore,
\[
\mathcal{U}^{(n)}(\mathcal{M})\cap\mathcal{U}^{(n)}(\mathcal{M}')=\varnothing\,.
\]
Then, by \cite[Lemma A.3]{LLS-1st}, $\bm{x}\in\partial\mathcal{U}^{(n)}(\mathcal{M})\cap\partial\mathcal{U}^{(n)}(\mathcal{M}')$
and $\bm{x}$ is a saddle point such that $U(\bm{x})=U(\mathcal{M})+R^{(n)}$.
However, this contradicts to the definition of $R^{(n)}$. Finally,
it is clear from Laplace asymptotics that $\mu_{\epsilon}(\mathcal{U}^{(n)}(\mathcal{M})\setminus\mathcal{E}(\mathcal{M}))\prec\mu_{\epsilon}(\mathcal{E}(\mathcal{M}))$
by \eqref{e: M^*_U}, which implies
\[
\frac{\mu_{\epsilon}(\mathcal{E}(\mathcal{M}))}{\mu_{\epsilon}(\mathcal{U}^{(n)}(\mathcal{M}))}=1\,.
\]
\end{proof}
\begin{acknowledgement*}
This work was supported by the KIAS Individual Grant (HP093101) at Korea Institute for Advanced Study and the National Research Foundation of Korea (NRF) grant funded by the Korea government(MSIT) (No. RS-2019-NR040050). We would like to thank Claudio Landim for providing the idea to prove Lemma 5.3. We also express our gratitude to IMPA for the invitation, which facilitated a fruitful discussion.
\end{acknowledgement*}

\end{document}